\documentclass [11pt,reqno]{amsart}
\usepackage {amsmath,amssymb,verbatim,geometry,color,esint}
\usepackage[all]{xy}
\usepackage{mathrsfs}
\usepackage[pdftex,hyperindex]{hyperref}
\usepackage[pdftex]{graphicx}
\usepackage{tikz}
\usetikzlibrary{math}

\usepackage{accents}

\newcommand{\B}{\mathbb{B}}
\newcommand{\C}{\mathbb{C}}
\newcommand{\DD}{\mathbb{D}}

\newcommand{\N}{\mathbb{N}}
\renewcommand{\P}{\mathbb{P}}
\newcommand{\Q}{\mathbb{Q}}
\newcommand{\R}{\mathbb{R}}
\newcommand{\Z}{\mathbb{Z}}

\newcommand{\fb}{\mathfrak{b}}

\newcommand{\hcE}{\widehat{\cE}}
\newcommand{\hcF}{\widehat{\cF}}

\newcommand{\cET}{\cE^{1,T}}

\newcommand{\ra}{\mathrm{rad}}

\newcommand{\tv}{\widetilde{v}}

\newcommand{\tmu}{\widetilde{\mu}}
\newcommand{\tnu}{\widetilde{\nu}}

\newcommand{\tphi}{\widetilde{\phi}}

\newcommand{\tPhi}{\widetilde{\Phi}}

\newcommand{\refe}{\mathrm{ref}}

\newcommand{\cA}{\mathcal{A}}

\newcommand{\cC}{\mathcal{C}}
\newcommand{\cD}{\mathcal{D}}
\newcommand{\cE}{\mathcal{E}}
\newcommand{\cF}{\mathcal{F}}
\newcommand{\cG}{\mathcal{G}}
\newcommand{\cH}{\mathcal{H}}

\newcommand{\cJ}{\mathcal{J}}
\newcommand{\cL}{\mathcal{L}}
\newcommand{\cM}{\mathcal{M}}

\newcommand{\cO}{\mathcal{O}}
\newcommand{\cP}{\mathcal{P}}

\newcommand{\cS}{\mathcal{S}}
\newcommand{\cT}{\mathcal{T}}

\newcommand{\cX}{\mathcal{X}}

\renewcommand{\a}{\alpha}
\renewcommand{\b}{\beta}
\renewcommand{\d}{\delta}
\newcommand{\e}{\varepsilon}
\newcommand{\f}{\varphi}

\newcommand{\g}{\gamma}
\newcommand{\la}{\lambda}
\newcommand{\om}{\omega}

\newcommand{\Ga}{\Gamma}

\newcommand{\p}{\psi}

\newcommand{\hK}{\widehat{\mathrm{K}}}

\newcommand{\sgh}{semi-globally H\"older}

\newcommand{{\EL}}{Euler--Lagrange}

\newcommand{\ie}{{\rm i.e.\ }}

\newcommand{\rel}{\mathrm{rel}}

\newcommand{\winter}{\wedge\dots\wedge}

\newcommand{\CAT}{\mathrm{CAT}(0)}

\newcommand{\hto}{\hookrightarrow}

\newcommand{\hyb}{\mathrm{hyb}}

\newcommand{\Bs}{\mathrm{Bs}}

\newcommand{\dd}{\mathrm{d}}
\newcommand{\dT}{\mathrm{d}_{\mathrm{Tits}}}
\newcommand{\done}{\delta_1}

\renewcommand{\DH}{\mathrm{DH}}

\DeclareMathOperator{\en}{E}

\newcommand{\ent}{\mathrm{H}}

\DeclareMathOperator{\ii}{I}
\DeclareMathOperator{\jj}{J}
\DeclareMathOperator{\rr}{R}
\DeclareMathOperator{\env}{P}
\DeclareMathOperator{\mab}{M}

\DeclareMathOperator{\Cz}{C^0}

\newcommand{\redu}{\mathrm{red}}
\newcommand{\ext}{\mathrm{ext}}

\newcommand{\enR}{\mathrm{R}}

\DeclareMathOperator{\CH}{CH}

\DeclareMathOperator{\Aut}{Aut}

\DeclareMathOperator{\DF}{DF}

\DeclareMathOperator{\Ent}{Ent}
\DeclareMathOperator{\Exc}{Exc}
\DeclareMathOperator{\FS}{FS}

\DeclareMathOperator{\Fut}{Fut}

\DeclareMathOperator{\Hom}{Hom}

\DeclareMathOperator{\MA}{MA}

\DeclareMathOperator{\Spec}{Spec}

\DeclareMathOperator{\supp}{supp}
\DeclareMathOperator{\vol}{vol}

\DeclareMathOperator{\Pic}{Pic}

\DeclareMathOperator{\id}{id}
\DeclareMathOperator{\im}{Im}
 
\DeclareMathOperator{\ord}{ord}

\newcommand{\slo}{{\prime\infty}}

\DeclareMathOperator{\PSH}{PSH}

\DeclareMathOperator{\Ric}{Ric}

\DeclareMathOperator{\GL}{GL}

\DeclareMathOperator{\PL}{PL}

\DeclareMathOperator{\Hnot}{H^0}

\DeclareMathOperator{\Lie}{\mathrm{Lie}}

\newcommand{\ddc}{dd^c}

\newcommand{\tr}{\mathrm{tr}}

\newcommand{\ddcT}{dd^c_T}

\newcommand{\supstar}{\mathrm{sup}^\star }

\renewcommand{\div}{\mathrm{div}}

\newcommand{\triv}{\mathrm{triv}}

\newcommand{\NA}{\mathrm{na}}
\newcommand{\na}{non-Archimedean }
\newcommand{\ma}{Monge--Amp\`ere }

\newcommand{\D}{\Delta}

\newcommand{\simto}{\overset\sim\to}

\numberwithin{equation}{section}       

\newtheorem{prop} {Proposition} [section]
\newtheorem{thm}[prop] {Theorem} 
\newtheorem{defi}[prop] {Definition}
\newtheorem{lem}[prop] {Lemma}
\newtheorem{cor}[prop]{Corollary}
\newtheorem{prop-def}[prop]{Proposition-Definition}

\newtheorem*{thmA}{Theorem A} 
\newtheorem*{thmAp}{Theorem A'} 
\newtheorem*{thmB}{Theorem B} 
\newtheorem*{thmC}{Theorem C}

\newtheorem*{defI}{Definition} 
\newtheorem{exam}[prop]{Example}
\newtheorem{rmk}[prop]{Remark}

\theoremstyle{remark}

\title[The Yau--Tian--Donaldson conjecture]{On the Yau--Tian--Donaldson conjecture for weighted cscK metrics}
\date{\today}

\author{S{\'e}bastien Boucksom \and Mattias Jonsson}

\address{Sorbonne Universit\'e and Universit\'e Paris Cit\'e\\
CNRS\\
IMJ-PRG\\
F-75005 Paris\\
France}
\email{sebastien.boucksom@imj-prg.fr}

\address{Dept of Mathematics\\
  University of Michigan\\
  Ann Arbor, MI 48109-1043\\
  USA}
\email{mattiasj@umich.edu}

\begin{document}

\begin{abstract}
We establish a version, formulated in terms of non-Archimedean pluripotential theory, of the Yau--Tian--Donaldson conjecture for constant scalar curvature and, more generally, weighted extremal K\"ahler metrics with prescribed compact symmetry group on an arbitrary polarized projective manifold. This is accomplished by extending important previous work of Chi Li to the weighted case, and by establishing general slope formulas for the relevant weighted entropy and energy functionals. Among other things, our approach relies on key \emph{a priori} estimates for cscK metrics due to Chen--Cheng (recently extended to the weighted case by Di Nezza--Jubert--Lahdili and Han--Liu), and on a crucial estimate for psh envelopes due to Berman--Demailly--Di Nezza--Trapani. 
\end{abstract}

\maketitle

\setcounter{tocdepth}{1}
\tableofcontents
%
%
%
%
%
%
%
%
%
%

%
%
%



\section*{Introduction}
Let $X$ be a smooth complex projective variety, and $L$ an ample line bundle on $X$. Broadly understood, the Yau--Tian--Donaldson (YTD) conjecture~\cite{Yau93,Tia97,Don02} predicts that the existence of a constant scalar curvature K\"ahler (cscK) metric (and, more generally, of a `canonical' metric) $\om\in c_1(L)$ is equivalent to a stability condition on $(X,L)$ of purely algebro-geometric nature. In the special case where $X$ is Fano and $L=-K_X$ (for which cscK metrics are K\"ahler--Einstein), this is now a famous result, see~\cite{CDS15}, and~\cite{Tia15}. 

Extending the variational/non-Archimedean approach initiated in~\cite{YTD}, we provide in this paper a solution to a general version of the YTD conjecture, formulated in terms of non-Archimedean pluripotential theory, and dealing with weighted extremal K\"ahler metrics with prescribed compact symmetry group. In the simpler and better known case of cscK metrics, our result can be stated as follows\footnote{As this manuscript was nearing completion, we were informed by T.~Darvas and K.~Zhang~\cite{DZ25} that they were able to prove a version of the equivalence between~(i) and~(iii) in Theorem~A when $X$ has discrete automorphism group. The relation between the results is discussed later in the introduction.}:

\begin{thmA} The following conditions are equivalent: 
\begin{itemize}
\item[(i)] there exists a cscK metric $\om\in c_1(L)$; 
\item[(ii)] $(X,L)$ is $\hK$-polystable;
\item[(iii)] $(X,L)$ is uniformly $\hK$-polystable.
\end{itemize}
\end{thmA}
More generally, we establish an equivariant version with respect to a given compact group of automorphisms (see Theorem~B). 

We need to make sense of conditions~(ii) and~(iii), but the general idea is to strengthen the classical definition of K-(poly)stability using objects more general than test configurations---an idea already explored in several works, see~\cite{Don12,Szefil,Mab16,LiGeod,nakstab2,DR24}. Specifically, we replace the space $\cT$ of ample test configurations with its completion $\widehat\cT$ with respect to a natural metric (which is what the chosen notation is meant to reflect). However, let us point out already now that~(iii) admits a more concrete formulation in terms of divisorial valuations, based on~\cite{nakstab2}, see~\S\ref{sec:divstab}.  

Let us also mention that (modulo the terminology) the implication (iii)$\Rightarrow$(i) was already established in~\cite{LiGeod} in the case $\Aut^0(X,L)=\C^\times$ is trivial\footnote{Li actually proves, more generally, that (i) follows from `$\mathbb{G}$-uniform K-stability for models', a condition that is in general strictly stronger than the equivariant version of (iii), see Lemma~\ref{lem:Licomp}.}, \ie reduced to the fiberwise scaling action on $L$ (see also~\cite{MP} for a general K\"ahler version).

The objects we use here arise from \emph{non-Archimedean geometry}, specifically the Berkovich analytification $(X_\NA,L_\NA)$ of $(X,L)$ with respect to the trivial absolute value on the ground field $\C$. The Berkovich space $X_\NA$, whose elements can be viewed as semivaluations on $X$, provides a natural compactification of the set of divisorial valuations.
This type of non-Archimedean geometry allows us to study asymptotic algebro-geometric constructions in a systematic way. The idea that non-Archimedean geometry can be used in K-stability goes back at least to S.-W.~Zhang, see~\cite[Remark~6]{PRS08}, and is systematically developed in~\cite{BHJ1,BHJ2,nakstab2,nakstab1}.

Recall that $(X,L)$ is called \emph{K-polystable} if, for each ample test configuration\footnote{All test configurations in this paper are assumed to be normal and defined over $\P^1$.} $(\cX,\cL)$ for $(X,L)$, the \emph{Donaldson--Futaki invariant} $\DF(\cX,\cL)\in\Q$ is nonnegative, and vanishes only on \emph{product test configurations}, arising from cocharacters $ \C^\times\to\Aut(X,L)$, see~\cite{Don02}. 

The existence of a cscK metric is known to imply K-polystability~\cite{Don05,Sto09,BDL2}, but the expectation, based on examples in~\cite{ACGTF}, seems to be that the converse may not be true; see also~\cite{Hat21}.

Following~\cite{BHJ1} we identify the space of ample test configurations with a space $\cH_\NA$ of continuous metrics on $L_\NA$. Thanks to the canonical `trivial' metric on $L_\NA$, the elements of $\cH_\NA$ can alternatively be viewed more concretely as continuous functions (potentials) on $X_\NA$, and hence on the set of divisorial valuations. As the notation suggests, we view $\cH_\NA$ as a \na analogue of the space $\cH$ of smooth positive Hermitian metrics on $L$ (viewed as a holomorphic line bundle.) 

In analogy to the complex analytic situation~\cite{Dar15}, the space $\cH_\NA$ can be equipped with a natural \emph{Darvas metric} $\dd_{1,\NA}$, and its metric completion can be identified with the space $\cE^1_\NA$ of \emph{potentials of finite energy}. For example, any reasonable \na graded norm $\chi\colon R(X,L)\to\R\cup\{+\infty\}$ (or, equivalently, filtration) of the section ring $R(X,L)$ gives rise to an element of $\cE^1_\NA$, see~\cite{nakstab1}. In particular, $\cE^1_\NA$ contains the set of \emph{real test configurations} (or $\R$-test configurations), associated to graded norms of finite type, and which also admit geometric realizations, see~\cite{CSW,DS,HLi24,InoEnt2,OdaNovikov}. Specializing further, $\cE^1_\NA$ contains the set $\cP_\R$ of \emph{real product test configurations}, corresponding to graded norms $\chi_\xi$ associated to an algebraic torus $T_\C\subset\Aut(X,L)$ and a vector $\xi\in\Lie T\simeq\in N_\R(T)$ on the Lie algebra of the maximal compact torus $T\subset T_\C$ (see~\S\ref{sec:lattices}). Elements in $\cP_\R\cap\cH_\NA=\cP_\Q$ correspond to usual product test configurations, up to the scaling action of $\Q_{>0}$, \ie up to (normalized) base change. 

We replace the Donaldson--Futaki invariant by the closely related \emph{\na Mabuchi functional} $\mab_\NA\colon\cH_\NA\to\Q$, introduced in~\cite{BHJ1}, which is homogeneous for the action of $\Q_{>0}$, and coincides with $\DF$ on test configurations with reduced central fiber. Via \na pluripotential theory, as developed in \cite{trivval,nakstab1,nakstab2}, this functional admits a natural lower semicontinuous extension $\mab_\NA\colon\cE^1_\NA\to\R\cup\{+\infty\}$. 

\begin{defI} We say that $(X,L)$ is: 
\begin{itemize}
\item[(a)] \emph{$\hK$-polystable} if $\mab_\NA(\f)\ge0$ for all $\f\in\cE^1_\NA$, with equality only if $\f\in\cP_\R$; 
\item[(b)]  \emph{uniformly $\hK$-polystable} if $\Aut^0(X,L)$ is reductive and there exists $\sigma>0$ such that $\mab_\NA(\f)\ge\sigma\dd_{1,\NA}(\f,\cP_\R)$ for all $\f\in\cE^1_\NA$. 
\end{itemize}
\end{defI}
Here $\dd_{1,\NA}(\f,\cP_\R)=\inf_{\p\in\cP_\R}\dd_{1,\NA}(\f,\p)$ denotes as usual the distance of $\f\in\cE^1_\NA$ to the subset $\cP_\R\subset\cE^1_\NA$; we will prove that the latter is closed as soon as $\Aut^0(X,L)$ is reductive, so that (b) implies (a), as suggested by the terminology. Note that this reductivity condition is natural in the setting of Theorem~A, as it holds as soon as a cscK metric exists, by the classical Matsushima--Lichnerowicz theorem. 

The notion of K-polystability, as introduced in~\cite{Tia97,Don02}, is equivalent to restricting (a) to $\f\in\cH_\NA\subset\cE^1_\NA$. We could similarly replace $\cE^1_\NA$ by $\cH_\NA$ in (b) and obtain a notion of \emph{uniform K-polystability}. When $\Aut^0(X,L)$ is further trivial, this amounts to \emph{uniform K-stability} as studied in~\cite{SzeThesis,Der,BHJ1}, which may in fact imply (b) (see~\cite[Conjecture~5.15]{nakstab2}). 

As already mentioned, the idea of using more general objects than test configurations is not new. For example, Sz\'ekelyhidi introduced in~\cite{Szefil} a condition that he later called $\hK$-stability, see~\cite[p.123]{Szebook}. His definition is similar to (a) but uses the strict subset of $\cE^1_\NA$ corresponding to norms/filtrations of the section ring (see~\cite{nakstab1} for a precise description), and involves a slightly different functional in the left-hand side. The main result of~\cite{Szefil} states that his condition holds when $(X,L)$ admits a cscK metric and $\Aut^0(X,L)$ is trivial. Our notion of $\hK$-(poly)stability is stronger than his; we do not know whether the two are equivalent in general.  

Beyond the set $\cH_\NA$ of ample test configurations, there are several natural subsets of $\cE^1_\NA$. We do not know whether any of them can \emph{a priori} replace $\cE^1_\NA$ in (a), but the situation is more favorable for (b). For example, we can equivalently use the subset of $\cE^1_\NA$ arising from test configurations for $(X,L)$ that are \emph{not necessarily ample}. These play a key role in the work of Chi Li~\cite{LiGeod,LiFuj} (who called them \emph{models}), and they are crucial here, too; see below. Alternatively, as in~\cite{nakstab2}, we can consider \emph{divisorial measures}, that is, (rational) convex combinations of divisorial valuations, leading to a valuative characterization of uniform $\hK$-polystability, see~\S\ref{sec:divstab}. From this valuative point of view, the difference between uniform $\hK$-polystability and uniform K-polystability is actually rather slim, as the $\beta$-invariant does not appear to be more easily computed for divisorial measures coming from an ample test configuration.

In the Fano case $L=-K_X$, the notions of K-polystability and $\hK$-polystability, as well as their uniform counterparts, are all equivalent. This follows from combining Theorem~A with~\cite{CDS15,LiKE} since a cscK metric is a K\"ahler--Einstein metric in this case, but it can also be seen algebraically, using the many reformulations of K-polystability available in the Fano setting. 

%
%
\medskip\noindent\textbf{Weighted extremal metrics}.
We next generalize Theorem~A to an algebro-geometric characterization of the existence of extremal K\"ahler metrics in the sense of Calabi, and, more generally, of \emph{weighted extremal metrics} in the sense of Lahdili~\cite{Lah19}, with a prescribed compact group of symmetries. 

More specifically, let $T\subset S\subset\Aut(X,L)$ be, respectively, a compact torus and a compact group commuting with $T$. Let also 
$$
P=P_L\subset M_\R(T)\simeq(\Lie T)^\vee
$$
be the moment polytope of the $T$-action on $(X,L)$, and fix positive `weights' $v,w\in C^\infty(P,\R_{>0})$. 

\begin{thmB} Assume that $v$ is log-concave. Then the following are equivalent: 
\begin{itemize}
\item[(i)] there exists an $S$-invariant weighted extremal metric $\om\in c_1(L)$; 
\item[(ii)] $(X,L)$ is $S$-equivariantly relatively weighted $\hK$-polystable; 
\item[(iii)] $(X,L)$ is uniformly $S$-equivariantly relatively weighted $\hK$-polystable.
\end{itemize}
\end{thmB}

In the usual extremal case (\ie $v=w=1$), a slightly weaker version of (iii)$\Rightarrow$(i) was very recently established by Hashimoto~\cite{Has25}, extending the approach of~\cite{LiGeod}, and a version of~(i)$\Rightarrow$(ii) was proved earlier by Stoppa and Sz\'ekelyhidi~\cite{SSz11}. We refer to~\S\ref{sec:wext2} for the definition of weighted extremal metrics, which correspond to critical points of the relative weighted Mabuchi energy $\mab^\rel$. Beyond extremal metrics in the sense of Calabi, Theorem~B covers K\"ahler--Ricci solitons~\cite{TZ02}, certain types of generalized solitons in the sense of~\cite{HLi23}, $\mu$-cscK metrics in the sense of Inoue~\cite{InoFound,InoEnt1}; see~\cite{Lah19,AJL} for details. Theorem~A is a special case of Theorem~B, obtained by taking $S=T=\{e\}$.

\medskip

The relative $\hK$-polystability notions in Theorem~B extend the ones in Theorem~A, and are obtained from the latter by replacing:
\begin{itemize}
    \item the \na Mabuchi energy $\mab_\NA$ with its relative weighted version $\mab^\rel_\NA$, constructed in this paper; 
    \item the set $\cE^1_\NA$ with the subset $\cE^{1,S}_\NA$ of $S$-invariant \na potentials; 
    \item the set $\cP_\R$ with the subset $\cP_\R^S$ of $S$-equivariant real product test configurations;
    \item the condition of reductivity of $\Aut^0(X,L)$ with that of the identity component $\Aut^S(X,L)$ of the centralizer\footnote{Note that~\cite{His,LiGeod} instead proposed to use the \emph{center} of $S$, which turns out to be too small to get a converse implication, see Lemma~\ref{lem:Licomp}.} of $S$ (which holds when an $S$-invariant weighted extremal metric exists~\cite{Cal,AJL}.) 
\end{itemize}

The elements of $\cP_\R^S$ are induced by tori in $\Aut^S(X,L)$; we show here again that $\cP_\R^S$ is closed in $\cE^{1,S}_\NA$ when $\Aut^S(X,L)$ is reductive. This ensures that $S$-equivariant relative weighted $\hK$-polystability follows, as it should, from its uniform version. 

As above, replacing $\cE^{1,S}_\NA$ with the subspace $\cH^S_\NA$ of $S$-equivariant ample test configurations yields the notion of \emph{$S$-equivariant relative weighted K-polystability} and its uniform version. When $S=T$ is a maximal compact torus of $\Aut^0(X,L)$ and $v=w=1$, they respectively correspond to \emph{relative K-polystability} in the sense of~\cite{SzeExtKstab} and \emph{reduced uniform K-stability} in the sense of~\cite{His,LiKE,XZ20} (assuming $\Aut^T(X,L)$ to be reductive, and hence equal to $T_\C$). 

\medskip

In the nontrivially weighted case, Lahdili~\cite{Lah19} gave an analytic definition of a weighted Donaldson--Futaki invariant on ample $T$-equivariant test configurations with smooth total space, and proved that the existence of a weighted cscK metric implies that the Donaldson-Futaki invariant is nonnegative on such test configurations when the central fiber is furthermore reduced. A uniform version of this result, for weighted extremal metrics, was proved in~\cite{AJL}. In this paper we give a purely algebraic definition of the \na relative weighted Mabuchi functional $\mab^\rel_\NA$ on $\cH_\NA^T$, and then use non-Archimedean pluripotential theory to extend this functional to all of $\cE^{1,T}_\NA$. This should be compared with the work of Han and Li~\cite{HLi23}, who defined the weighted Ding functional on $\cE^{1,T}_\NA$ (see also~\cite{DJ}). For technical reasons, the non-Archimedean weighted Mabuchi functional turns out to be more delicate.

%
%
\medskip\noindent\textbf{Outline of proofs}. The proof of Theorem~A builds upon deep work by many different authors. It follows the variational method developed in~\cite{YTD} for the Fano case, but is more subtle since we do not have access to the Ding functional. Let us give an outline, emphasizing where our main new contributions lie. 

As discovered by Mabuchi~\cite{Mab86}, cscK metrics correspond to critical points of the Mabuchi functional $\mab\colon\cH\to\R$, which further is (infinitesimally) convex with respect to the `tautological' Riemannian $L^2$-metric of $\cH$. Later, X.~Chen proved in~\cite{Che00b} that the associated length (pseudo)metric defines a metric $\dd_2$ on $\cH$, and this was generalized by Darvas~\cite{Dar15} to all `tautological' Finsler $L^p$-norms of $\cH$, yielding for each $p\in [1,\infty]$ a metric $\dd_p$ on $\cH$. By~\cite{Dar17}, the $\dd_p$-completion of $\cH$ can further be identified with the space $\cE^p$ of plurisubharmonic (psh) metrics of finite $L^p$-energy~\cite{GZ}. We will work with the space $\cE^1$, which corresponds, when $\dim X=1$, to the classical space of (quasi)subharmonic functions with gradient in $L^2$. 

By~\cite{BDL1}, the complete metric space $(\cE^1,\dd_1)$ is uniquely geodesic\footnote{All geodesics are assumed to be psh.}, and \emph{Busemann convex}, \ie the metric $\dd_1$ is convex along geodesics. The Mabuchi energy admits a maximal lower semicontinuous (lsc) extension $\mab\colon\cE^1\to\R\cup\{+\infty\}$, which satisfies the following deep and crucial properties:
\begin{itemize}
\item $\mab$ is geodesically convex~\cite{BerBer,BDL1};
\item $\mab$ is \emph{strongly lsc}, in the sense that the intersection of each sublevel set with a closed ball is compact~\cite{BBEGZ};
\item any minimizer of $\mab$ in $\cE^1$ lies in $\cH$, and hence defines a cscK metric~\cite{BDL2,CC2}; 
\item minimizers of $\mab$ are unique modulo the action of $G:=\Aut^0(X,L)$~\cite{BerBer}. 
\end{itemize}
Using these properties, a general argument originating in~\cite{DR,YTDold} (see Theorem~\ref{thm:crit} below) then leads to the following version of Tian's properness conjecture: a cscK metric exists iff $\mab$ is \emph{coercive modulo $G$}, \ie $\mab\ge\sigma\dd_{1,G}-C$ for some $\sigma,C>0$, where 
$$
\dd_{1,G}(\phi):=\inf_{g\in G}\dd_1(\phi,g\cdot\phi_\refe)
$$
denotes the distance of $\phi\in\cE^1$ to the (closed) $G$-orbit of some reference metric $\phi_\refe$. This equivalently means that $\mab$ descends to a function on $\cE^1/G$ growing at least linearly with respect to the quotient metric, and it further suffices to check this condition on each individual geodesic ray. 

Exploiting the Busemann convexity of $\cE^1$, this last point admits in turn a nice reformulation in term of the \emph{asymptotic cone} $\cE^1_\ra$, that is, the space of equivalence classes $\f$, called \emph{directions}, of geodesic rays $\{\phi_t\}$ in $\cE^1$ modulo parallelism; this is a complete metric cone for the radial metric 
$$
\dd_{1,\ra}(\f,\p):=\lim_{t\to\infty} t^{-1}\dd_1(\phi_t,\psi_t).
$$
By geodesic convexity, $\mab$ admits a \emph{radial transform} $\mab_\ra\colon\cE^1_\ra\to\R\cup\{+\infty\}$, similarly computing the slope at infinity\footnote{This corresponds to \textyen~in~\cite{Che09,CC2}.}
$$
\mab_\ra(\f):=\lim_{t\to\infty} t^{-1}\mab(\phi_t), 
$$
and we may now state that a cscK metric exists iff $(X,L)$ is \emph{geodesically polystable}, in the sense that $\mab_\ra(\f)\ge 0$ for each $\f\in\cE^1_\ra$, with equality iff $\f$ lies in the set $\cD_G\subset\cE^1_\ra$ of \emph{toric $G$-directions}, induced by real one-parameter subgroups of tori of $G$~\cite{Don99,CC2,BDL2}. 

\medskip

In a first step towards the proof of Theorem~A, we show that this is also equivalent to \emph{uniform} geodesic polystability: 

\begin{thmAp}
The following three conditions are equivalent:
\begin{itemize}
 \item[(i)]
    there exists a cscK metric in $c_1(L)$;
 \item[(ii')]
    $(X,L)$ is geodesically polystable; 
   \item[(iii')]
     $(X,L)$ is \emph{uniformly geodesically polystable}, \ie $G=\Aut^0(X,L)$ is reductive, and there exists $\sigma>0$ such that $\mab_\ra(\f)\ge\sigma\dd_{1,\ra}(\f,\cD_G)$ for all $\f\in\cE^1_\ra$.\end{itemize}
\end{thmAp}
Here (iii')$\Rightarrow$(ii') since we prove that $\cD_G\subset\cE^1_\ra$ is closed as soon as $G$ is reductive. The latter holds when a cscK metric exists, and the proof of (i)$\Rightarrow$(iii') then relies on the slope formula
$$
\lim_{t\to\infty} t^{-1}\dd_G(\phi_t)=\dd_{1,\ra}(\f,\cD_G)
$$
for any geodesic ray $\{\phi_t\}$ in $\cE^1$ with direction $\f\in\cE_\ra$. 

Our next task is to reformulate~(ii') and~(iii') in terms of algebraic, or rather, non-Archimedean geometry. Following work by Phong--Sturm~\cite{PS06,PS07,PS10}, Chen--Tang~\cite{ChenTang}, Ross--Witt Nystr\"om~\cite{RWN14} and others, Berman~\cite{Berm16} gave a simple construction of a geodesic ray in $\cE^1$ associated to any ample test configuration $(\cX,\cL)$ for $(X,L)$. This yields an embedding $\cH_\NA\hto\cE^1_\ra$, which was extended in~\cite{YTD,Reb} to an isometric embedding 
$$
(\cE^1_\NA,\dd_{1,\NA})\hto(\cE^1_\ra,\dd_{1,\ra}), 
$$
taking the set $\cP_\R$ of real test configurations onto the set $\cD_G$ of toric $G$-directions. 

The key result in the present cscK case is now:
\begin{thmC}
For any $\f\in\cE^1_\ra$ we have $\mab_\ra(\f)=\left\{
\begin{array}{ll}
\mab_\NA(\f) &\  \text{if } \f\in\cE^1_\NA \\
+\infty & \text{ otherwise}. 
\end{array}
\right. 
$
\end{thmC}
Indeed, this implies that (uniform) $\hK$-polystability is equivalent to (uniform) geodesic polystability, so that Theorem~A follows from Theorem~A'. 

The second point in Theorem~C is due to Chi Li~\cite{LiGeod}. To explain our proof of the first point (the toric case of which was established in~\cite[Theorem~10]{Berm19}) recall that the Chen--Tian formula~\cite{Che00a,Tia00} yields, for $\phi\in\cE^1$, a decomposition
\[
\mab(\phi)=\ent(\phi)+\rr(\phi)+\bar{S}\en(\phi), 
\]
where $\ent(\phi)$ is the (relative) entropy of the Monge--Amp\`ere measure $\MA(\phi)$, $\rr(\phi)$ is the Ricci energy, $\en(\phi)$ the Monge--Amp\`ere energy, and $\bar{S}$ a cohomological constant. These functionals also admit radial transforms 
$$
\ent_\ra\colon\cE^1_\ra\to\R\cup\{+\infty\},\quad\en_\ra\colon\cE^1_\ra\to\R\quad\text{and}\quad\rr_\ra\colon\cE^1_\ra\to\R.
$$
The non-Archimedean Mabuchi functional $\mab_\NA$ decomposes as $\ent_\ra+\rr_\ra+\bar S\en_\ra$, and it follows from~\cite{YTD,LiGeod} that $\en_\ra=\en_\NA$ and $\rr_\ra=\rr_\NA$ on $\cE^1_\NA\subset\cE^1_\ra$. 

The problem in Theorem~C is therefore to prove the entropy asymptotics $\ent_\ra=\ent_\NA$ on $\cE^1_\NA$. This equality was established on the subspace $\cH_\NA$ in~\cite{BHJ2}, but while $\cH_\NA$ is dense in $\cE^1_\NA$, the radial and \na entropy functionals are only lsc, so the equality does not immediately propagate to $\cE^1_\NA$. However, C.~Li was able to use the ideas of~\cite{BHJ2} to establish the inequality $\ent_\NA\le\ent_\ra$ on $\cE^1_\NA$. Here we prove (also) the reverse inequality, thus establishing Conjecture~1.6 in~\cite{LiGeod}. 

In order to prove $\ent_\ra(\f)=\ent_\NA(\f)$ for $\f\in\cE^1_\NA$, we first show that it suffices to assume that $\f$ is associated (through a psh envelope construction) to a (not necessarily ample) test configuration $(\cX,\cL)/\P^1$. This reduction, which is based on solving a non-Archimedean Monge--Amp\`ere equation~\cite{nama,trivval}, was also observed by C.~Li, who, as already mentioned, calls such test configurations models. The geodesic ray associated to such a $\f$ is represented as a psh envelope: the largest $S^1$-invariant metric $\tilde\Phi$ on $\cL|_\DD$ whose restriction to $\cL|_{\partial\DD}$ is given by $\phi_\refe$. We next show that this envelope can be replaced with a global (over $\P^1$) psh envelope, namely the largest psh metric $\Phi$ on $\cL$ dominated by some given smooth $S^1$-invariant metric on $\cL$, which induces an almost geodesic ray with direction $\f$. The curvature current $\ddc\Phi$ is not smooth, but as a key new ingredient we use a result by Berman--Demailly~\cite{BerDem} (which contained a gap later fixed by Di Nezza--Trapani~\cite{DNT}) that controls its singularities, which are in particular contained in the augmented base locus $\B_+(\cL)$. We also use the fact, observed in~\cite{LiFuj}, that the Monge--Amp\`ere measure $\MA_\NA(\f)$ is an atomic measure with support on divisorial valuations associated to irreducible components of $\cX_0$ that are \emph{not} contained in the augmented relative base locus $\B_+(\cL/\P^1)$. After adding a large enough constant to $\f$, we show that the two base loci become equal, and use the ingredients above to prove that $\ent_\ra(\f)=\ent_\NA(\f)$, thus establishing Theorem~C.

\medskip
The proof of Theorem~B follows the same pattern as that of Theorem~A, but requires a significant amount of additional work.

First, we need to develop the relevant \na weighted pluripotential theory, and make sure that the relevant functionals and operators are defined in a purely algebraic way. Some of this was done in~\cite{HLi23}, but here we also need the \na versions of the twisted weighted \ma energy and weighted \ma measure; these turn out to be rather delicate to handle (see also the very recent work~\cite{HLiuklt}).

Second, we need to show that the (Archimedean) energy functionals and operators, defined on the space $\cE^{1,T}$ of $T$-invariant metrics of finite energy, have a well-defined slope at infinity along  geodesic rays, and that this slope agrees with the corresponding non-Archimedean quantity when the ray in question is maximal. 

With the energy asymptotics under control, the proof of the weighted analogue of Theorem~C goes along the same lines as before, although the approximation arguments require us to solve a weighted non-Archimedean Monge--Amp\`ere equation. 

Finally, the variational approach also requires us to know that minimizers of the relative weighted Mabuchi functional $\mab^\rel$ on $\cE^{1,S}$ are smooth, for which we adapt the arguments of~\cite{CC2} using the recent work of Di Nezza, Jubert and Lahdili~\cite{DJL1,DJL2}, and Han--Liu~\cite{HLiu} extending~\cite{CC1,He19} to the weighted setting.  

\medskip\noindent\textbf{Questions and further remarks}.
While we prove the YTD conjecture in the form of Theorems~A and~B here, there are a number of questions left, of which we present a sample.

First, Theorem~A shows that $\hK$-polystability is equivalent to uniform $\hK$-polystability, but the proof goes through the existence of a cscK metric. Can we find a direct, algebraic proof?

To test that a smooth polarized variety $(X,L)$ is $\hK$-polystable is impractical with the current definition. The situation was a bit similar in the Fano case at the time the proof of the YTD conjecture was announced, but various subsequent developments, such as the valuative criterion by Fujita and Li~\cite{Fuj19,Li17}, the delta invariant of Fujita and Odaka~\cite{FO18} (see also~\cite{BlJ20,RTZ21,Zha21,Zha24,DJ}), and the work of Abban and Zhuang~\cite{AZ22}, has made it much more manageable to test K-stability. 

Specifically, we can consider uniform $\hK$-polystability. Here it suffices to test objects such as divisorial measures, that is, convex combinations of divisorial valuations. It is probably not sufficient to consider Dirac masses~\cite{DerLeg}, but one may wonder if it suffices to consider measures whose support has a cardinality bounded above, say in terms of the dimension of $X$.

Another question concerns equivariance. It follows from Theorems~A and~B that if $(X,L)$ is $S$-equivariantly $\hK$-polystable for some compact subgroup $S\subset\Aut^0(X,L)$, then $(X,L)$ is also $\hK$-polystable. Conversely, as the isometry group of a cscK metric complexifies to $\Aut^0(X,L)$, it follows that if $(X,L)$ is $\hK$-polystable, then it is $T$-equivariantly $\hK$-polystable for some (hence any) maximal compact torus $T$ of $\Aut^0(X,L)$ (see Corollary~\ref{cor:equivstab}). Can these statements be proved algebraically? See~\cite{Zhua21} for related results in the Fano case.

We could also consider the singular case. Chi Li~\cite{LiKE} proved the YTD conjecture for log Fano varieties. Extending our work to polarized klt pairs is probably challenging, but there has been some recent work on singular cscK metrics, see~\cite{BJT,PTT23,PT24,Sze24}.

Finally we could consider the case of a (non-projective) K\"ahler manifold $(X,\omega)$. In this context, Dervan--Ross~\cite{DR17} and Sj\"ostr\"om-Dyrefelt~\cite{SD18} proved that the existence of a cscK metric implies K-semistability, and even uniform K-stability when the automorphism group is finite. More recently, Mesquita-Piccione~\cite{MP} made a nontrivial adaptation of~\cite{LiGeod} and gave a sufficient stability condition for the existence of a cscK metric, which was further refined in very recent joint work with Witt Nystr\"om~\cite{MW}. 

In the Fano case $L=-K_X$, the YTD conjecture was proven in~\cite{CDS15}, so in this case, $\hK$-polystability and uniform $\hK$-polystability are both equivalent to K-polystability (as well as its uniform version, see Remark~\ref{rmk:uKps}). All of this can be proven algebraically, using the powerful techniques that have been developed over the last few years, see~\cite{XuBook} and the references therein.

\medskip\noindent\textbf{Comparison to~\cite{DZ25}}.
As this manuscript was being completed, we were informed by T.~Darvas and K.~Zhang that they had obtained a different version of the YTD conjecture. In~\cite{DZ25}, $X$ is assumed to have discrete automorphism group. The authors prove that the existence of a cscK metric in $c_1(L)$ is equivalent to $(X,L)$ being uniformly $K^\beta$-stable for some $\beta>0$. This notion of stability is defined on ample test configurations, but using a deformation $\mab^\beta_\NA$ of the functional $\mab_\NA$ depending on a parameter $\beta$. The algebraic description of $\mab^\beta_\NA$ is somewhat more involved, but it corresponds to the radial transform of a functional $\mab^\beta$, which is a deformation of $\mab$ with good properties. The proofs in~\cite{DZ25} and here are both rooted in some form of Theorem~A', but they diverge substantially after that.

%
%
\medskip\noindent\textbf{Organization}.
After recalling some relevant material in~\S\ref{sec:prelim}, we prove in~\S\ref{sec:asymgeod} an abstract result characterizing when a suitable functional on a Busemann convex space admits a minimum, the characterization being in terms of the slopes at infinity along geodesic rays. In~\S\ref{sec:ppt} and~\S\ref{sec:NAppt}, we review the surprisingly parallel words of global complex analytic and non-Archimedean pluripotential theory, respectively. The relation between the non-Archimedean and complex analytic worlds is explored in~\S\ref{sec:NAlimmet}, where we first prove some new results about maximal geodesic rays, and then set up a general framework for viewing non-Archimedean objects as limits of their complex analytic counterparts. This is used in~\S\ref{sec:wppt} to define and study non-Archimedean analogues of the functionals appearing in weighted pluripotential theory; extending some of these---notably the non-Archimedean weighted Mabuchi functional---to the space of finite energy metrics is quite subtle. In~\S\ref{sec:entasym}, we prove that the weighted non-Archimedean entropy is the slope at infinity of the weighted entropy along any maximal geodesic ray: this is the key to proving (a weighted generalization of) Theorem~C, which in turn is a key new ingredient needed to prove Theorem~B.
Finally, in~\S\ref{sec:wstab}, we put everything together, and prove Theorem~B, of which Theorem~A is a special case. The paper ends with two appendices: one collecting material about complex reductive groups, and one proving that minimizers of the weighted Mabuchi functional are smooth, following~\cite{DJL2} and~\cite{CC2}.
%
%

\medskip\noindent\textbf{Acknowledgement}. We thank V.~Apostolov, R.~Berman, T.~Darvas, R.~Dervan, E.~Di Nezza, S.~Finski, Y.~Hashimoto, S.~Jubert, A.~Lahdili, Y.~Odaka, P.~Piccione, R.~Reboulet, and T.D.~Tô for fruitful discussions. We are grateful to M.~Brion for his help with Lemma~\ref{lem:reduc}, and to T.~Darvas and K.~Zhang for sharing their manuscript~\cite{DZ25}. We further thank T.~Darvas, J.~Han, E.~Inoue,  Yaxiong Liu, P.~Piccione, J.~Ross, and V.~Tosatti for helpful comments on a first draft of this paper. The first author was partially supported by the ANR project AdAnAr (ANR-24-CE40-6184) and the KRIS project. The second author was partially supported by NSF grants DMS-2154380 and DMS-2452797, and award 1034361 from the Simons Foundation.

\section{Preliminaries}\label{sec:prelim}
Besides setting up some notation, this preliminary section gathers some basic facts on Euler--Lagrange functionals and Futaki invariants, and also provides a comparison result between absolute and relative base loci of big line bundles that plays an important role in the proof of the main result.   

\subsection{Notation and terminology}\label{sec:not}

\subsubsection{Hermitian metrics}
We use additive notation for (possibly singular) Hermitian metrics $\phi$ on a holomorphic $\Q$-line bundle $F$ over a complex space, \ie $m\phi$ denotes the induced metric on $mF=F^{\otimes m}$ for $m\in\Z$. We denote by $C^\infty(F)$ the affine space of smooth metrics $\phi$ on $F$, modeled on the real vector space $C^\infty(X)$ of smooth functions $f\colon X\to\R$. We write $\ddc\phi$ for the curvature form of any metric $\phi\in C^\infty(F)$. 

\subsubsection{Tori}\label{sec:lattices} Given a compact torus $T\simeq(S^1)^r$ we denote by $T_\C\simeq(\C^\times)^r$ its complexification, and by 
\[
M_\Z(T):=\Hom(T,S^1)=\Hom(T_\C,\C^\times)\simeq\Z^r,\quad N_\Z(T):=\Hom(S^1,T)=\Hom(\C^\times,T_\C)\simeq\Z^r
\]
the character and cocharacter lattices, respectively. We also set $M_\Q(T):=M_\Z(T)\otimes_\Z\Q$, and similarly define $N_\Q(T)$, $M_\R(T)$ and $N_\R(T)$. We then have a canonical isomorphism 
$\Lie T\simeq N_\R(T)$. 

\subsubsection{Centralizers}\label{sec:cent} For any topological group $G$ and subgroup $S\subset G$, we denote by $G^S$ the identity component of the centralizer of $S$ in $G$, \ie of the set of fixed points in $G$ under the action of $S$ by conjugation. In particular, we write $\Aut^S(X,L)$ for the identity component of the group automorphisms of a polarized variety $(X,L)$ that commute with a given subgroup $S$. 

\subsubsection{} In this paper, a \emph{quasi-metric} on a set $\cE$ is a symmetric function $\dd\colon\cE\times\cE\to\R_{\ge 0}$ such that for all $\phi_1,\phi_2,\phi_3\in\cE$ we have
\begin{itemize}
\item $\dd(\phi_1,\phi_2)=0$ iff $\phi_1=\phi_2$;
\item $\dd(\phi_1,\phi_3)\le C\left(\dd(\phi_1,\phi_2)+\dd(\phi_2,\phi_3)\right)$ for some uniform constant $C>0$. 
\end{itemize}
A quasi-metric space $(\cE,\dd)$ is equipped with a Hausdorff topology and a uniform structure, and admits a completion $(\hcE,\dd)$. 

We say that a map $F\colon(\cE,\dd)\to(\cF,\dd)$ between quasi-metric spaces is a \emph{quasi-isometry} if there exists $C>0$ such that $C^{-1}\dd(\phi_1,\phi_2)\le\dd(F(\phi_1),F(\phi_2))\le C\dd(\phi_1,\phi_2)$ for all $\phi_1,\phi_2\in\cE$. (We do not allow an additive constant, as is sometimes the case in the literature.)

The following notion will play a recurring role in this paper.

\begin{defi}\label{defi:sgh} We say that a map $F\colon(\cE,\dd)\to(\cF,\dd)$ between two quasi-metric spaces is \emph{\sgh} if there exists $C>0$ and $\a\in (0,1)$ such that 
\begin{equation}\label{equ:sgh}
\dd(F(\phi_1),F(\phi_2))\le C\dd(\phi_1,\phi_2)^\a\max_i\left(1+\dd(\phi_i)\right)^{1-\a} 
\end{equation}
for all $\phi_1,\phi_2\in\cE$. 
\end{defi}
Here we have set $\dd(\phi):=\dd(\phi,\phi_\refe)$ for some choice of basepoint $\phi_\refe\in\cE$, the definition being independent of that choice. Note that~\eqref{equ:sgh} implies the linear growth property
\begin{equation}\label{equ:sghbd}
\dd(F(\phi))\le A\dd(\phi)+B
\end{equation}
for uniform constants $A,B>0$. Any {\sgh} map $F\colon\cE\to\cF$ is uniformly continuous on balls, and hence uniquely extends to a {\sgh} map 
$F\colon\hcE\to\hcF$ between the completions of $\cE,\cF$. 

\subsubsection{Isometric group actions}

We say that an isometric action $G\times\cE\to\cE\quad(g,\phi)\mapsto g\cdot\phi$ of a locally compact group $G$ on a metric space $(\cE,\dd)$ is \emph{proper} if, for each ball $B$ in $\cE$, 
$\{g\in G\mid gB\cap B\ne\emptyset\}$ is relatively compact in $G$. Equivalently, for some (hence any) $\phi,\p\in\cE$ and all $C>0$, 
$\{g\in G\mid\dd(\phi,g\cdot\p)\le C\}$ is relatively compact in $G$. As a result, the infimum 
$$
\dd_G(\phi,\p):=\inf_{g\in G}\dd(\phi,g\cdot\p)
$$ 
is achieved, and descends to a \emph{quotient metric} on $\cE/G$.

\subsubsection{Relative entropy}\label{sec:ent}
Given a compact topological space $Z$, we denote by $\Cz(Z)^\vee$ the space of signed Radon measures endowed with the weak topology, and by 
$$
\cM(Z)\subset\Cz(Z)^\vee
$$
the subspace of positive measures. The \emph{relative entropy} of $\mu,\nu\in\cM(Z)$ is defined by setting
$$
\Ent(\mu|\nu):=\int (\rho\log \rho)\,\nu\in\R\cup\{+\infty\}
$$
if $\mu=\rho \nu$, $\rho\in L^1(\nu)$, and $\Ent(\mu|\nu):=+\infty$ otherwise. Up to normalization, $\Ent(\cdot|\nu)$ is the Legendre transform of the convex functional $\Cz(Z)\ni f\mapsto\log\int e^f\,\nu$. More specifically, 
\begin{equation}\label{equ:entleg}
\Ent(\mu|\nu)=\sup_{f\in \Cz(Z)}\left\{\int f\,d\mu-\mu(Z)\log\int e^f \nu\right\}+\mu(Z)\log\mu(Z). 
\end{equation}
This shows that $\Ent(\cdot|\nu)\colon\cM(Z)\to\R\cup\{+\infty\}$ is convex, lsc, and satisfies 
\begin{equation}\label{equ:entlow}
\Ent(\mu|\nu)\ge\log\left(\frac{\mu(Z)}{\nu(Z)}\right)\mu(Z). 
\end{equation}

\subsection{{\EL} functionals and Futaki invariants}\label{sec:EL} 
We introduce here some general terminology and useful observations that will apply to both the complex and the \na cases we shall be dealing with. 

\subsubsection{{\EL} functionals} 

Consider a real affine space $\cA$ of `metrics' $\phi$, directed by a real vector space of $\cF$ of `functions' $f$, a convex subset $\cC\subset\cA$, and a `measure-valued' map
$$
\Ga\colon \cC\to \cF^\vee\quad\phi\mapsto \Ga(\phi)
$$
assumed to be sufficiently regular for the various derivatives below to make sense. Here $ \cF^\vee$ denotes the (algebraic) dual of $\cF$, the duality pairing $\cF\times \cF^\vee\to\R$ being suggestively written as $(f,\mu)\mapsto\int f\,\mu$. 

\begin{defi} An \emph{{\EL} functional} for $\Ga$ is a map $F\colon \cC\to\R$ such that 
$$
\frac{d}{ds}\bigg|_{s=0}F(\phi+s f)=\int f\,\Ga(\phi)
$$
for all $\phi\in\cC$ and $f\in\cF$ such that $\phi+f\in\cC$. For brevity we also write $F'(\phi)=\Ga(\phi)$. 
\end{defi} 
For all $\phi,\p\in \cC$, we then have 
\begin{equation}\label{equ:EL}
F(\phi)-F(\p)=\int_0^1 ds\int (\phi-\p)\Ga(s\phi+(1-s)\p). 
\end{equation}

\begin{rmk} It is sometimes convenient to normalize $F$ by requiring $F(\phi)=0$ for a given $\phi\in\cC$. Denoting by $F(\phi,\cdot):=F-F(\phi)$ this choice of normalization, we then trivially have the \emph{cocycle formula} 
$$
F(\phi_1,\phi_3)=F(\phi_1,\phi_2)+F(\phi_2,\phi_3)
$$
for all $\phi_1,\phi_2,\phi_3\in\cC$.
\end{rmk}

For any $\phi\in\cC$ and $f\in\cF$ such that $\phi+f\in\cC$, we consider the directional derivative $\Ga'(\phi;f)\in \cF^\vee$ defined by
$$
\int g\,\Ga'(\phi;f):=\frac{d}{ds}\bigg|_{s=0}\int g\,\Ga(\phi+sf)
$$
for $g\in \cF$. One then easily checks: 

\begin{lem}\label{lem:EL} Pick a basepoint $\phi_\refe\in \cC$. The following properties are equivalent:
\begin{itemize}
\item[(i)] $\Ga\colon \cC\to \cF^\vee$ admits an {\EL} functional; 
\item[(ii)] the functional defined by 
$$
F(\phi):=\int_0^1 ds\int(\phi-\phi_\refe)\Ga(s\phi+(1-s)\phi_\refe)
$$ 
satisfies~\eqref{equ:EL} for all $\phi,\p\in \cC$; 
\item[(iii)] the directional derivatives of $\Ga$ satisfy the symmetry property
\begin{equation}\label{equ:sym}
\int f_1\Ga'(\phi;f_2)=\int f_2\Ga'(\phi;f_1)
\end{equation}
for all $\phi\in\cC$ and $f_1,f_2\in\cF$ such that $\phi+f_i\in\cC$. 
\end{itemize}
\end{lem}

\begin{lem}\label{lem:Fconc} Assume $F\colon\cE\to\R$ is an {\EL} functional for $\Ga\colon\cC\to\cF^\vee$. Then the following are equivalent:
\begin{itemize}
\item[(i)] $F$ is concave; 
\item[(ii)] for all $\phi,\p\in\cC$, we have 
\begin{equation}\label{equ:FJ}
\jj(\phi,\p):=\int(\phi-\p)\Ga(\p)-F(\phi)+F(\p)\ge 0; 
\end{equation}
\item[(iii)] for all $\phi,\p\in\cC$, we have 
\begin{equation}\label{equ:FI} 
\ii(\phi,\p):=\int(\phi-\p)\left(\Ga(\phi)-\Ga(\p)\right)\ge 0. 
\end{equation}
\end{itemize}
\end{lem}
\begin{proof} Given a line segment in $\cC$, (ii) expresses that $F$ lies below its affine interpolation, while (iii) expresses that the derivative of $F$ is monotone decreasing. They are thus both equivalent to (i). Note that (ii)$\Rightarrow$(iii) also follows from $\ii(\phi,\p)=\jj(\phi,\p)+\jj(\p,\phi)$.
\end{proof}

\subsubsection{Futaki invariants}

Assume next given a group action 
$$
G\times\cA\to \cA\quad (g,\phi)\mapsto g\cdot\phi
$$
by affine transformations preserving $\cC$, with respect to which the map $\Ga\colon \cC\to \cF^\vee$ is $G$-equivariant. Using for instance~\eqref{equ:EL}, one sees that any {\EL} functional $F$ is then \emph{quasi-invariant}, in the sense that 
\begin{equation}\label{equ:quasi}
F(g\cdot\phi)-F(g\cdot\p)=F(\phi)-F(\p)
\end{equation}
for all $\phi,\p\in \cC$ and $g\in G$. However, $F$ is not $G$-invariant in general. Indeed, as a consequence of~\eqref{equ:quasi}, 
\begin{equation}\label{equ:char}
\chi_\Ga(g):=F(g\cdot\phi)-F(\phi)
\end{equation}
is independent of $\phi\in \cC$, and defines a group character 
$$
\chi_\Ga\colon G\to\R,
$$
which vanishes iff $F$ is $G$-invariant (compare, for instance~\cite{Fut}). 

\begin{exam}\label{exam:transinv} Assume given an affine action of $G=\R$ on $\cA$, denoted by $(\phi,c)\mapsto\phi+c$. Set $1_\cF:=0_\cF+1\in \cF$, and assume that $\Ga$ is `translation-invariant', \ie $\Ga(\phi+c)=\Ga(\phi)$ for all $\phi\in \cC$ and $c\in\R$. By~\eqref{equ:sym}, the `total mass' $\bar\Ga:=\int 1_\cF\,\Ga(\phi)$ is independent of $\phi\in \cC$, and 
$$
F(\phi+c)=F(\phi)+c\bar\Ga
$$
for all $\phi\in \cC$ and $c\in\R$.
\end{exam}

Assume further that $G$ is a connected Lie group whose action on $\cA$ is sufficiently regular for the differential of $\chi_\Ga$ to make sense. The latter then defines a Lie algebra character 
$$
\Fut_\Ga\colon\Lie G\to\R,
$$
the \emph{Futaki character} of $\Ga$, which vanishes iff $F$ is $G$-invariant.  Differentiating~\eqref{equ:char} yields the more explicit formula
\begin{equation}\label{equ:FutLie}
\Fut_\Ga(\xi)=\int(\cL_\xi\phi)\,\Ga(\phi)
\end{equation}
for all $\xi\in\mathrm{Lie}\,G$ and $\phi\in \cC$, where 
$$
\cL_\xi\phi:=\frac{d}{dt}\bigg|_{t=0} e^{t\xi}\cdot\phi\in \cF.
$$

\begin{exam} If $F$ admits a critical point $\phi\in \cC$, \ie $\Ga(\phi)=0$, then $\Fut_\Ga\equiv 0$, \ie $F$ is $G$-invariant.
\end{exam}

\begin{exam} If $F$ is bounded below, then so is the group homomorphism $\chi_\Ga\colon G\to\R$; it thus vanishes, \ie $F$ is $G$-invariant. 
\end{exam}

%
\subsection{Absolute vs.~relative base loci of big line bundles}\label{sec:loci}
We use~\cite{ELMNP,BBP} as general references for what follows. Given a $\Q$-line bundle $L$ on a normal projective variety $X$, recall that the \emph{asymptotic},  \emph{diminished} and \emph{augmented base loci} of $L$ are respectively defined as
$$
\B(L):=\bigcap_m\Bs(mL),\quad\B_-(L):=\bigcup_{\e\in\Q_{>0}}\B(L+\e A),\quad\B_+(L):=\bigcap_{\e\in\Q_{>0}}\B(L-\e A).
$$
Here $A$ denotes any given ample line bundle on $X$ (the definitions being independent of the choice of $A$), $m$ ranges over all sufficiently divisible positive integers,  $\Bs(mL)$ denotes the base locus of $mL$, \ie the Zariski closed subset cut out by the \emph{base ideal}
$$
\fb_{mL}:=\im\left(\Hnot(X,mL)\otimes\cO_X(-mL)\to\cO_X\right). 
$$
We obviously have 
$$
\B_-(L)\subset\B(L)\subset\B_+(L);
$$
where $\B(L)$, $\B_+(L)$ are Zariski closed, but \emph{a priori} not $\B_-(L)$ (see~\cite{Les} for an example with $L$ an $\R$-divisor). Furthermore, 
\begin{align*}
\B_-(L)\ne X\Leftrightarrow L\text{ psef},& \quad \B_-(L)=\emptyset\Leftrightarrow L\text{ nef},\\
\B_+(L)\ne X\Leftrightarrow L\text{ big},& \quad\B_+(L)=\emptyset\Leftrightarrow L\text{ ample}.
\end{align*}
Note also that for any $\e\in\Q_{>0}$ small enough we have 
\begin{equation}\label{equ:Bpert}
\B_+(L)=\B_-(L-\e A). 
\end{equation}
More generally, given a morphism $\pi\colon X\to S$ to another projective variety, one defines relative versions 
$$
\Bs(mL/S)\subset\Bs(mL),\quad\B(L/S)\subset\B(L),\quad\B_-(L/S)\subset\B_-(L),\quad\B_+(L/S)\subset\B_+(L)
$$
of the above loci by replacing $\Hnot(X,mL)$ with the sheaf of relative sections $\pi_\star\cO_X(mL)$ (and allowing $A$ to be merely relatively ample).

The proof of~\cite[Proposition~2.3]{BBP} directly extends to the relative case and yields the following birational invariance property: 

\begin{lem}\label{lem:B+bir} Given a $\Q$-line bundle $L$ on $X$ and any birational morphism $\rho\colon Y\to X$ with $Y$ normal projective, we have 
$$
\B_+(\rho^\star L/S)=\rho^{-1}(\B_+(L/S))\cup\mathrm{Exc}(\rho).
$$
In particular, a prime divisor on $X$ lies in $\B_+(L/S)$ iff its strict transform on $Y$ lies in $\B_+(\rho^\star L/S)$.
\end{lem}

Similarly, \cite[Lemma~V.1.9]{Nak} or~\cite[Proposition 2.8]{ELMNP}, which are based on Nadel's vanishing theorem, adapt without change to the relative case to yield: 

\begin{lem}\label{lem:B-} Assume that $L$ is relatively big. Then the \emph{relative asymptotic vanishing order} 
\begin{equation}\label{equ:vanordrel}
v\|L/S\|:=\lim_{m\to\infty} m^{-1}\min\left\{v(s)\mid 0\ne s\in\pi_\star\cO_X(mL)\right\}\in [0,+\infty)
\end{equation} 
is well-defined for any valuation $v\colon\C(X)^\times\to\R$. If $X$ is further smooth, then $v\|L/S\|>0$ iff $\B_-(L/S)$ contains the center of $v$ on $X$.
\end{lem}

Using this we show: 

\begin{lem}\label{lem:B+} Assume $X$ is smooth. If $\pi(\B_+(L))$ is a finite set, then $\B_+(L)=\B_+(L/S)$.
\end{lem}
\begin{proof} Pick an ample line bundle $A$ on $X$. By noetherianity, we can choose $\e\in\Q_{>0}$ and then $m$ sufficiently divisible such that $L_\e:=L-\e A$ is big and satisfies 
$$
B:=\B_+(L)=\B_-(L_\e)=\B(L_\e)=\Bs(mL_\e), 
$$
$$
B_S:=\B_+(L/S)=\B_-(L_\e/S)=\B(L_\e/S)=\Bs(mL_\e/S). 
$$
Pick a birational morphism $\rho\colon Y\to X$ that dominates the blowup of both base ideals $\fb_{mL_\e},\fb_{mL_\e/S}$ and is an isomorphism outside their zero loci $B\supset B_S$. Then
$$
\rho^\star(mL_\e)=M+E=M_S+E_S
$$
where $M$ (resp.~$M_S$) is globally generated (resp.~relatively globally generated over $S$) and $E\ge E_S$ are effective Cartier divisors on $Y$ such that $\rho(\supp E)=B$, $\rho(\supp E_S)=B_S$. We claim that $M_S$ is (globally) nef on $Y$. To see this, consider any irreducible curve $C\subset Y$. If $C$ maps to a point of $S$, then $M_S\cdot C\ge0$ as $M_S$ is relatively nef over $S$. Assume now that $C$ maps to a curve on $S$. Then $C$ cannot be contained in the finite set 
$$
\pi(\B_+(L))=\pi(\rho(\supp E)).
$$
Thus $C$ cannot be contained in the support of the effective divisor $E-E_S$, and hence $(E-E_S)\cdot C\ge0$, which yields as desired $M_S\cdot C\ge M\cdot C\ge0$, since $M$ is nef. 

Since $L_\e$ is big and $M_S=\rho^\star(mL_\e)-E_S$ is nef, $B_S=\rho(\supp E_S)$ contains the non-nef locus of $L_\e$ (see~\cite[Lemma~2.4]{BBP}), which coincides with its diminished base locus $\B_-(L_\e)=B$ by Lemma~\ref{lem:B-}. Thus $B=B_S$, and the result follows. 
\end{proof}

%
\section{Asymptotic cones and geodesic stability}\label{sec:asymgeod}
This section provides a general study of the growth properties of geodesically convex functions on Busemann convex metric spaces, culminating with a characterization of coercivity modulo a reductive group action in terms of geodesic stability. Most results here are closely related to~\cite{DR,BDL2,DL,CC2}, but the slope formula for the quotient metric (Theorem~\ref{thm:slopedG}) appears to be new. In this whole section, $(\cE,\dd)$ denotes a complete metric space (which will be the space $\cE^1$ of finite energy metrics in our concrete application). 

\subsection{The asymptotic cone}\label{sec:asym}
Recall  that a \emph{(constant-speed) geodesic} in $\cE$ is a path $\{\phi_t\}_{t\in I}$ parametrized by a (possibly unbounded) interval $I\subset\R$, such that 
$$
\dd(\phi_s,\phi_t)=c|s-t|
$$
for $s,t\in I$ and a constant $c\in\R_{\ge 0}$, called the \emph{speed} of $\{\phi_t\}$. It is called a \emph{geodesic segment} if $I$ is a closed interval, and a \emph{geodesic ray} if  $I=\R_{\ge 0}$. The set of geodesics is preserved under affine reparametrization and pointwise limits. 

\begin{defi}\label{defi:disting} We shall say that a class $\cG$ of geodesics in $(\cE,\dd)$ is  \emph{distinguished} if: 
\begin{itemize}
\item $\cG$ is closed under affine reparametrization and pointwise limits; 
\item a geodesic lies in $\cG$ iff its restriction to each compact subinterval lies in $\cG$. 
\end{itemize}
\end{defi}
A distinguished class $\cG$ is uniquely determined by the set of geodesic segments it contains; conversely, any set of geodesic segments generates a minimal distinguished class of geodesics. From now on we fix a distinguished class $\cG$, and implicitly assume all geodesics to lie in $\cG$ unless otherwise specified. 

We say that $(\cE,\dd)$ is \emph{(uniquely) geodesic} (with respect to $\cG$) if any pair of points can be joined by a geodesic segment (unique up to reparametrization). Recall that $\cE$ is then locally compact iff its closed balls are compact, by the Hopf--Rinow theorem. 

\begin{exam} Any Banach space $(\cE,\dd)$ is uniquely geodesic with respect to the distinguished class of affine paths. 
\end{exam}

The metric space $(\cE,\dd)$ is said to be \emph{Busemann convex} (with respect to $\cG$) if it is geodesic, and its metric $\dd$ is geodesically convex, \ie $t\mapsto \dd(\phi_t,\p_t)$ is convex for any pair of geodesics $\{\phi_t\}$, $\{\p_t\}$. As one easily sees, it is enough to check this when $\phi_0=\p_0$. It then amounts to 
$$
\dd(\phi_t,\p_t)\le t \dd(\phi_1,\p_1),\quad t\in [0,1], 
$$
which says that geodesic triangles are `thinner' than Euclidean ones in the sense that Thales' theorem is an inequality. Busemann convexity is thus a \emph{nonpositive curvature} condition (see~\cite{Pap} as a general reference). 

\begin{exam} Any Banach space is Busemann convex with respect to the distinguished class of affine geodesics. 
\end{exam}

\begin{exam} If $(\cE,\dd)$ is $\CAT$, then it is Busemann convex with respect to the class of all geodesics, see for instance~\cite[II.2.2]{BrH}. For a Riemannian manifold $\cE$, the converse holds, and this is also equivalent to $\cE$ being of nonpositive sectional curvature and simply connected (and hence contractible, see Proposition~\ref{prop:Buse}~(iii)). 

In contrast, while any strictly convex Banach space is Busemann convex with respect to all geodesics (which are line segments in that case), it is $\CAT$ iff it satisfies the parallelogram law, and hence is a Hilbert space. 
\end{exam}

Busemann convex spaces enjoy the following general properties: 

\begin{prop}\label{prop:Buse} Assume that $(\cE,\dd)$ is Busemann convex. Then:
\begin{itemize}
\item[(i)] $(\cE,\dd)$ is uniquely geodesic; 
\item[(ii)] geodesic segments depend continuously on their end-points;
\item[(iii)] $\cE$ is contractible; 
\item[(iv)] every compact group of isometries of $(\cE,\dd)$ admits a fixed point. 
\end{itemize}
\end{prop}
\begin{proof} Properties (i), (ii)  are straightforward consequences of the convexity of $\dd$. Fixing a basepoint $\phi_\refe\in\cE$ and mapping each $\phi\in\cE$ to $H_t(\phi):=\phi_t$ for the unique geodesic segment $\{\phi_t\}_{t\in [0,1]}$ joining $\phi_0=\phi_\refe$ to $\phi_1=\phi$ then defines a homotopy $H_t\colon\cE\to\cE$ between $H_1=\id_\cE$ and $H_0\equiv\phi_\refe$, proving (iii). Finally, (iv) follows from~\cite[Theorem~1.1]{Bas}. 
\end{proof}

\begin{defi} The \emph{radial distance}\footnote{The term \emph{chordal metric} used in~\cite{DL} sounds more appropriate for the spherical version of $\cE_\ra$ using unit speed geodesic rays, as in~\cite[II.9]{BrH}.}  between two geodesic rays $\{\phi_t\}_{t\ge 0}$, $\{\p_t\}_{t\ge 0}$ is defined as
$$
\dd_\ra(\{\phi_t\},\{\p_t\}):=\lim_{t\to +\infty} t^{-1}  \dd(\phi_t,\p_t).
$$
\end{defi}
Note that the limit exists in $\R_{\ge 0}$, since the function $t\mapsto \dd(\phi_t,\p_t)$ is convex and has at most linear growth (by the triangle inequality). Convexity further yields
\begin{equation}\label{equ:drays}
\dd(\phi_t,\p_t)\le\dd(\phi_0,\p_0)+t\,\dd_\ra(\{\phi_t\},\{\p_t\}), 
\end{equation}
so that $\dd_\ra(\{\phi_t\},\{\p_t\})=0$ iff $\{\phi_t\}$ and $\{\p_t\}$ are \emph{parallel}, \ie $\dd(\phi_t,\p_t)$ is bounded. In line with the usual definition for $\CAT$ spaces, we introduce: 

\begin{defi} The \emph{asymptotic cone} of $(\cE,\dd)$ (with respect to $\cG$) is defined as the metric space $(\cE_\ra,\dd_\ra)$ where $\cE_\ra$ denotes the set of (distinguished) geodesic rays modulo parallelism, equipped with the radial metric $\dd_\ra$. 
\end{defi}
The equivalence class of a geodesic ray $\{\phi_t\}$ will be denoted by $\phi_\infty\in\cE_\ra$ and called its \emph{direction}; we also say that $\{\phi_t\}$ is \emph{directed} by $\f\in\cE_\ra$ if $\f=\phi_\infty$. 

The asymptotic cone $(\cE_\ra,\dd_\ra)$ is indeed a metric cone with respect to the scaling action
$$
\R_{>0}\times \cE_\ra\to \cE_\ra\quad (c,\f)\mapsto c\f
$$
induced by $c\cdot\{\phi_t\}:=\{\phi_{ct}\}$, in that 
$$
\dd_\ra(c\f,c\p)=c\,\dd_\ra(\f,\p)
$$
for all $\f,\p\in\cE_\ra$ and $c>0$. The vertex $0\in\cE_\ra$ corresponds to constant rays, and unit-speed geodesic rays yield a section of the cone. 

\begin{prop}\label{prop:Busecone} The metric space $(\cE_\ra,\dd_\ra)$ is complete. Given any basepoint $\phi_\refe\in \cE$, $\cE_\ra$ further is in 1--1 correspondence with the set of geodesic rays emanating from $\phi_\refe$. 
\end{prop}
Note that $\cE_\ra$ is generally not locally compact even when $\cE$ is (see Example~\ref{exam:building} below). 

\begin{proof} We first prove the second point. Pick a basepoint $\phi_\refe\in\cE$. By~\eqref{equ:drays}, geodesic rays emanating from $\phi_\refe$ are uniquely determined by their image in $\cE_\ra$. Conversely, pick a unit-speed geodesic ray $\{\p_t\}$ emanating from some other point $\p_0\in\cE$. We need to construct a unit-speed geodesic ray $\{\phi_t\}$ emanating from $\phi_\refe$ and parallel to $\{\p_t\}$. For each $T>0$, denote by $\{\phi_{T,t}\}_{0\le t\le T}$ the geodesic joining $\phi_\refe$ to $\p^T$, and note that its speed 
$$
c^T=\frac{\dd(\phi_\refe,\p^T)}{\dd(\p_0,\p^T)}
$$
tends to $1$ as $T\to\infty$. For $0\le t\le T\le T'$, convexity yields 
$$
\dd(\phi_{T,t},\phi_{T',t})\le\frac{t}{T} \dd(\p^T,\phi_{T',T})\text{ and } \dd(\p^T,\phi_{T',T})\le\frac{T'-T}{T}\dd(\phi_\refe,\p_0)
$$
(draw a picture!). For each $t$ fixed, it follows that $T\mapsto\phi_{T,t}$ is Cauchy. Since $\{\phi_{T,t}\}_t$ has speed $c^T\to 1$ as $T\to+\infty$, it thus converges to a unit-speed geodesic ray $\{\phi_t\}$ emanating from $\phi_\refe$, which is easily seen to be parallel to $\{\p_t\}$. This proves the second point. 

Now pick a Cauchy sequence $(\f_j)$ in $\cE_\ra$, and for each $j$ denote by $\{\phi_{j,t}\}_{t\ge 0}$ the geodesic ray emanating from $\phi_\refe$ and directed by $\f_j$. Then~\eqref{equ:drays} shows that, for each $t$, the sequence $(\phi_{j,t})_j$ is Cauchy. Thus $\{\phi_{j,t}\}$ converges pointwise to a geodesic ray $\{\phi_t\}$, whose direction $\f$ is easily seen to be the $\dd_\ra$-limit of $(\f_j)$, using again~\eqref{equ:drays}. This shows that $(\cE_\ra,\dd_\ra)$ is complete. 
\end{proof}

\begin{rmk}\label{rmk:radgeod} While this will not be needed in this paper, similar elementary arguments show that $(\cE_\ra,\dd_\ra)$ is geodesic and Busemann convex with respect to a natural distinguished class $\cG_\ra$ of geodesics induced by $\cG$ (compare~\cite[Theorem~4.7]{DL}). 
\end{rmk}

\ 

On top of geodesic rays, the following class of rays also defines a direction in the asymptotic cone:

\begin{defi}\label{defi:almost} We shall say that an arbitrary ray $\{\phi_t\}_{t\ge 0}$ in $\cE$ is an \emph{almost geodesic ray} if $\dd(\phi_t,\p_t)=o(t)$ as $t\to\infty$ for some geodesic ray $\{\p_t\}$. 
\end{defi}
The direction of $\{\p_t\}$ is then uniquely determined; it will be denoted by $\phi_\infty\in\cE_\ra$ and called the \emph{direction of} $\{\phi_t\}$. In fact, the geodesic ray $\{\p_t\}$ emanating from $\phi_\refe$ and directed by $\phi_\infty$ is necessarily equal to the pointwise limit as $T\to\infty$ of the geodesic segments $\{\p^T_t\}_{t\in [0,T]}$ joining $\phi_\refe$ to $\phi^T$, since $\dd(\p^T_t,\p_t)\le\frac{t}{T} \dd(\phi^T,\p^T)$ by Busemann convexity. 

As a simple consequence of the triangle equality, we note: 

\begin{lem}\label{lem:almostd} For any two almost geodesic rays $\{\phi_{1,t}\}$, $\{\phi_{2,t}\}$ we have
$$
\lim_{t\to\infty}t^{-1}\dd(\phi_{1,t},\phi_{2,t})=\dd_\ra(\phi_{1,\infty},\phi_{2,\infty}).
$$
\end{lem} 

%
\subsection{Radial transforms and coercivity thresholds}\label{sec:radtransf}
Assume as above that $(\cE,d)$ is Busemann convex, with asymptotic cone $(\cE_\ra,\dd_\ra)$, and consider a function
$$
F\colon\cE\to\R\cup\{+\infty\}. 
$$
Following the terminology of~\cite{DL}, we introduce: 

\begin{defi}\label{defi:radial} We say that a function $F_\ra\colon\cE_\ra\to\R\cup\{+\infty\}$ is the \emph{radial transform} of $F$ if, for every geodesic ray $\{\phi_t\}$ in $\cE$ not entirely contained in $\{F=+\infty\}$, we have 
$$
F_\ra(\phi_\infty)=\lim_{t\to\infty} t^{-1}F(\phi_t).
$$
If $F\equiv+\infty$ we set $F_\ra\equiv+\infty$. 
\end{defi}
Since every $\f\in\cE_\ra$ is the direction of a unique geodesic ray emanating from a given point in $\{F<\infty\}$ (assuming such a point exists) (see Proposition~\ref{prop:Busecone}), the radial transform $F_\ra$, when it exists, is uniquely determined. It is further trivially homogeneous (of degree $1$) with respect to the scaling action of $\R_{>0}$. Note also that $F_\ra$ takes values in $[0,+\infty]$ as soon as $F$ is bounded below. 

\begin{prop}\label{prop:radial} Assume $F\colon\cE\to\R\cup\{+\infty\}$ is lsc and geodesically convex, \ie $t\mapsto F(\phi_t)$ is convex for each geodesic $\{\phi_t\}$. Then $F$ admits a radial transform $F_\ra\colon\cE_\ra\to\R\cup\{+\infty\}$, that is further lsc. 
\end{prop}

\begin{exam} If $(\cE,\dd)$ is a Banach space, then $(\cE_\ra,\dd_\ra)$ can be canonically identified with $(\cE,\dd)$, and the radial transform
$F_\ra(\f)=\lim_{t\to\infty} t^{-1}  F(t\f)$ of a convex lsc function $F$ coincides with its  \emph{homogenization}, characterized as the smallest convex homogeneous function above $F$. 
\end{exam}

\begin{proof}[Proof of Proposition~\ref{prop:radial}] For any geodesic ray $\{\phi_t\}$, $t\mapsto F(\phi_t)$ is a convex function on $\R_{\ge 0}$, and we denote by 
$$
F^\slo(\{\phi_t\}):=\lim_{t\to\infty} t^{-1} F(\phi_t)
$$
its slope at infinity. Note that 
\begin{equation}\label{equ:Fslope}
F(\phi_t)\le F(\phi_0)+t F^\slo(\{\phi_t\})\quad\text{ for all}\quad t\ge 0,
\end{equation}
by convexity. We first claim that if $\{\phi_t\}$, $\{\p_t\}$ are two parallel geodesic rays not entirely contained in the locus $\{F=+\infty\}$, then $F^\slo(\{\phi_t\})=F^\slo(\{\p_t\})$.

To see this, pick $t_0\ge 0$ such that $F(\p_{t_0})<\infty$. Pick $T>t_0$, and consider as above the geodesic segment $\{\p^T_t\}_{t\in [t_0,T]}$ joining $\p_{t_0}$ to $\phi^T$. By convexity of $d$, for $t\in [t_0,T]$ we have 
$$
\dd(\p^T_t,\p_t)\le\frac{t-t_0}{T-t_0}\dd(\phi^T,\p^T),
$$
and hence $\lim_{T\to\infty} \p^T_t=\p_t$ for each $t$, since $\dd(\phi^T,\p^T)=o(T)$. On the other hand, by convexity of $F$ we have 
$$
F(\p^T_t)\le\left(1-\frac{t-t_0}{T-t_0}\right)F(\p_{t_0})+\frac{t-t_0}{T-t_0} F(\phi^T).
$$
Since $F$ is lsc, letting $T\to \infty$ yields $F(\p_t)\le F(\p_{t_0})+tF^\slo(\{\phi_t\})$, and hence $F^\slo(\{\p_t\})\le F^\slo(\{\phi_t\})$. The claim follows, by symmetry. 

Now fix a basepoint $\phi_\refe\in\cE$, chosen so that $F(\phi_\refe)<\infty$. Since the set of $\f\in\cE_\ra$ is in 1--1 correspondence with the set of geodesic rays $(\phi_t)$ emanating from $\phi_\refe$ (see Proposition~\ref{prop:Busecone}~(i)), setting $F_\ra(\f):=F^\slo(\{\phi_t\})$ yields the desired radial transform
$$
F_\ra\colon\cE_\ra\to\R\cup\{+\infty\}. 
$$

Next pick a convergent sequence $\f_j\to\f$ in $\cE_\ra$, and denote by $\{\phi_{j,t}\}_t$, $\{\phi_t\}_t$ the corresponding geodesic rays emanating from $\phi_\refe$. By convexity of $\dd$, for all $j$ and $t\ge 0$ we have $\dd(\phi_{j,t},\phi_t)\le t\dd_\ra(\f_j,\f)$. In particular, $\phi_{j,t}\to \phi_t$ for each $t$. By~\eqref{equ:Fslope} we further have $t^{-1} F(\phi_{j,t})\le F_\ra(\f_j)$. Since $F$ is lsc, we infer $t^{-1} F(\phi_t)\le\liminf_j F_\ra(\f_j)$. Letting $t\to\infty$ yields $F_\ra(\f)\le\liminf_jF_\ra(\f_j)$, which proves that $F_\ra$ is lsc. 
\end{proof}

\begin{rmk} Similar arguments show that $F_\ra$ further is geodesically convex with respect to the distinguished class $\cG_\ra$, see Remark~\ref{rmk:radgeod}. 
\end{rmk}

The following useful observation originates from~\cite[Proposition~5.1]{Xia21}. 

\begin{lem}\label{lem:slopemono} Assume $F\colon\cE\to\R\cup\{+\infty\}$ is convex and lsc. For any almost geodesic ray $\{\phi_t\}$ we then have 
$$
F_\ra(\phi_\infty)\le\liminf_{t\to\infty} t^{-1} F(\phi_t). 
$$
\end{lem}
\begin{proof} We may assume that $F$ is finite at $\phi_\refe:=\phi_{t_0}$ for some $t_0\ge 0$, as the result is otherwise trivial. The unique geodesic ray $\{\p_t\}$ emanating from $\phi_\refe$ and directed by $\phi_\infty$ is then pointwise limit of the geodesic segments $\{\p^T_t\}_{t\in [0,T]}$ joining $\phi_\refe$ to $\phi^T$. For all $0<t\le T$ we then have 
$$
t^{-1}\left(F(\p^T_t)-F(\phi_\refe)\right)\le T^{-1} \left(F(\phi^T)-F(\phi_\refe)\right),
$$
by convexity. Since $F$ is lsc and $\lim_{T\to\infty} \p^T_t=\p_t$, we infer  
$$
t^{-1}\left(F(\p_t)-F(\phi_\refe)\right)\le\liminf_{T\to\infty} T^{-1} F(\phi^T), 
$$
and letting $t\to\infty$ concludes the proof. 
\end{proof}

\begin{prop}\label{prop:sgh} Assume that $F\colon\cE\to\R$ is {\sgh} (see Definition~\ref{defi:sgh}), and that it admits a radial transform $F_\ra$. Then:
\begin{itemize}
\item[(i)] $F_\ra\colon\cE_\ra\to\R$ is finite valued and {\sgh};
\item[(ii)] $F_\ra(\phi_\infty)=\lim_{t\to\infty} t^{-1}  F(\phi_t)$ for any almost geodesic ray $\{\phi_t\}$ in $\cE$. 
\end{itemize}
\end{prop}
\begin{proof} (i) is straightforward by taking the slopes at infinity in~\eqref{equ:sgh}. To see (ii), pick a geodesic ray $\{\p_t\}$ directed by $\phi_\infty$. Since $\dd(\p_t,\p_0)=O(t)$ and $\dd(\phi_t,\p_t)=o(t)$, we have $\dd(\phi_t)=O(t)$. By~\eqref{equ:sgh} we infer $|F(\phi_t)-F(\p_t)|=o(t)$, and (ii) follows. 
\end{proof}

\begin{defi} The \emph{coercivity threshold} of a function $F\colon \cE\to\R\cup\{+\infty\}$ is defined as 
$$
\sigma(F):=\sup\{\sigma\in\R\mid F\ge\sigma \dd(\cdot,\phi_\refe)-C\text{ for some }C\in\R\}\in [-\infty,+\infty]. 
$$
We say that $F$ is \emph{coercive} if $\sigma(F)>0$, \ie $F\ge\sigma \dd-C$ for some positive constants $\sigma,C>0$. 
\end{defi}
The coercivity threshold is easily seen to be independent of the choice of basepoint $\phi_\refe\in\cE$. 

\begin{exam}\label{exam:convcoer} If $F\colon\R_{\ge 0}\to\R\cup\{+\infty\}$ is convex, then $\sigma(F)$ coincides with the slope at infinity $F'(\infty)=\lim_{t\to+\infty}t^{-1}  F(t)\in\R\cup\{+\infty\}$. In particular, $F$ is either coercive or nonincreasing. 
\end{exam}

As we shall see, basically following~\cite{DR,YTDold,CC2}, a similar dichotomy holds in far greater generality. To state this we introduce: 
 
\begin{defi}\label{defi:stronglsc} We say that $F\colon \cE\to\R\cup\{+\infty\}$ is \emph{strongly lsc} if $B\cap\{F\le c\}$ is compact for each closed ball $B$ and $c\in\R$. 
\end{defi} 

This definition is motivated by the compactness properties of the entropy functional, see Proposition~\ref{prop:hent} below. Note that $F$ strongly lsc implies $F$ lsc, the converse being true when $\cE$ is locally compact. Further, if $F$ is strongly lsc then it is bounded below on each bounded subset of $\cE$. The following perturbation result is straighforward to check: 

\begin{lem}\label{lem:holdpert} If $F\colon\cE\to\R\cup\{+\infty\}$ is strongly lsc and $G\colon\cE\to\R$ is {\sgh}  (see Definition~\ref{defi:sgh}), then $F+G$ is strongly lsc as well.
\end{lem}
%
\subsection{The equivariant coercivity threshold}
Suppose next given an isometric action 
$$
G\times\cE\to\cE\quad(g,\phi)\mapsto g\cdot\phi
$$
of a locally compact group $G$, preserving the distinguished class $\cG$ of geodesics. We assume that the $G$-action is proper, so that the infimum
$$
\dd_G(\phi,\p)=\inf_{g\in G}\dd(\phi,g\cdot\p)
$$ 
is achieved, and defines the quotient metric on $\cE/G$. 

\begin{defi} We say that a subset $\cM\subset\cE$ is \emph{$G$-minimal} if $\dd(\phi,\p)=\dd_G(\phi,\p)$ for all $\phi,\p\in\cM$. 
\end{defi}
By properness of the $G$-action, each pair of points in $\cE/G$ is represented by some $G$-minimal pair in $\cE$. 
\begin{lem}\label{lem:minclosed} Pick two convergent sequences $\phi_j\to\phi$, $\p_j\to\p$ in $\cE$. If $\{\phi_j,\p_j\}$ is $G$-minimal for all $j$, then $\{\phi,\p\}$ is $G$-minimal as well. 
\end{lem}
\begin{proof} Pick $g\in G$. For each $j$ we have $\dd(\phi_j,\p_j)\le\dd(g\cdot\phi_j,\p_j)$. Thus $\dd(\phi,\p)\le\dd(g\cdot\phi,\p)$, and the result follows.
\end{proof}

The next observation is certainly well-known (cf.~\cite[Lemma~7.4]{CC2}).

\begin{lem}\label{lem:Gmin} A geodesic segment $\{\phi_t\}_{t\in [0,1]}$ in $\cE$ is $G$-minimal iff $\{\phi_0,\phi_1\}$ is $G$-minimal. 
\end{lem}
\begin{proof} Assume $\{\phi_0,\phi_1\}$ is $G$-minimal, and pick $0\le s\le t\le 1$. Since $G$ acts by isometries, the triangle inequality yields, for all $g\in G$, 
\begin{multline*}
\dd(\phi_0,\phi_1)\le \dd(\phi_0,g\cdot\phi_1)  \\
\le \dd(\phi_0,\phi_s)+\dd(\phi_s,g\cdot\phi_t)+\dd(\phi_t,\phi_1)
 =\dd(\phi_0,\phi_1)-\dd(\phi_s,\phi_t)+\dd(\phi_s,g\cdot\phi_t).
\end{multline*}
This implies $\dd(\phi_s,g\cdot\phi_t)\ge \dd(\phi_s,\phi_t)$, and the result follows. 
\end{proof}

The image in $\cE/G$ of any $G$-minimal geodesic in $\cE$ is a geodesic for the quotient metric on $\cE/G$, which is thus a geodesic metric space with respect to a natural distinguished class of geodesics induced by $\cG$. It is, however, not uniquely geodesic in general (think of $S^1=\R/\Z)$. In particular, $\cE/G$ is generally not Busemann convex. 

Pick a basepoint $\phi_\refe\in\cE$, and set 
\[
\dd_G(\phi):=\dd_G(\phi,\phi_\refe),
\]
which is thus the distance of $\phi$ to the (closed) $G$-orbit of $\phi_\refe$. 

\begin{defi} We define the \emph{$G$-equivariant coercivity threshold} of a function $F\colon\cE\to\R\cup\{+\infty\}$ as
$$
\sigma_G(F)=\sup\{\sigma\in\R\mid F\ge\sigma \dd_G-C\text{ on }\cE\text{ for some }C\in\R\}\in[-\infty,+\infty]. 
$$
\end{defi}
Trivially, $\sigma(F)\le\sigma_G(F)$. In practice, $F$ will be quasi $G$-invariant (see~\eqref{equ:quasi}), and hence $G$-invariant as soon as it is bounded below; then $\sigma_G(F)$ coincides with the coercivity threshold of the induced function on the quotient space $(\cE/G,\dd_G)$. 

As in~\cite[Lemma~1.11]{BJT} we note:

\begin{lem}\label{lem:coersub} Assume $F\colon\cE\to\R\cup\{+\infty\}$ is quasi $G$-invariant, and let $H\subset G$ be a closed subgroup. Then the following are equivalent:
\begin{itemize}
    \item $F$ is coercive modulo $H$; 
    \item $F$ is coercive modulo $G$, and $G/H$ is compact. 
\end{itemize}
\end{lem}

We may now generalize Example~\ref{exam:convcoer} as follows: 

\begin{thm}\label{thm:crit} Suppose $F\colon \cE\to\R\cup\{+\infty\}$ is strongly lsc, geodesically convex and $G$-invariant. Then
$$
\sigma_G(F)=\inf_\f F_\ra(\f)
$$
where $\f\in\cE_\ra$ ranges over the directions of all $G$-minimal, unit-speed geodesic rays in $\cE$ emanating from $\phi_\refe$. Furthermore, the infimum is achieved. 
\end{thm} 

As a consequence, we get the following result, which can be traced back to~\cite{YTD,ICM,CC2}. 

\begin{cor}\label{cor:crit} Under the same assumptions, exactly one of the following holds:
\begin{itemize}
\item[(i)] $F$ is coercive modulo $G$;   
\item[(ii)] there exists a $G$-minimal, unit-speed geodesic ray, emanating from any given basepoint $\phi_\refe\in \cE$, along which $F$ is nonincreasing. 
\end{itemize}
Moreover, (i) implies that $F$ admits a minimizer, and the converse holds if the minimizer is further unique modulo $G$.
\end{cor}

\begin{proof}[Proof of Theorem~\ref{thm:crit}] Assume $F\ge\sigma\dd_G-C$ with $\sigma,C\in\R$, and pick a $G$-invariant, unit-speed geodesic ray $\{\phi_t\}$ emanating from $\phi_\refe$ and with direction $\f\in\cE_\ra$. Then $F(\phi_t)\ge\sigma\dd(\phi_t,\phi_\refe)-C$, and hence 
$F_\ra(\f)\ge\sigma$. This show $\sigma_G(F)\le\inf_\f F_\ra(\f)$. 

To get the reverse inequality, we may assume $\sigma_G(F)<+\infty$. Pick a decreasing sequence $\sigma_j\in\R$ such that $\sigma_j>\sigma_G(F)$ and $\sigma_j\to\sigma_G(F)$, and pick any sequence $C_j\to+\infty$. By definition of $\sigma_G(F)$, for each $j$ we can find $\phi_j\in \cE$ such that 
\begin{equation}\label{equ:Fnonp}
F(\phi_j)\le\sigma_j T_j-C_j
\end{equation}
with $T_j:=\dd_G(\phi_j)$. By $G$-invariance of $F$, we may further assume that $\{\phi_\refe,\phi_j\}$ is $G$-minimal, \ie $T_j=\dd(\phi_j,\phi_\refe)$. 

If $T_j\le R$ for some uniform constant $R>0$, then $F(\phi_j)\le\sigma_0 R-C_j\to-\infty$, which contradicts the fact that $F$ is bounded below on $B(0,R)$, being strongly lsc. After passing to a subsequence, we may thus assume $T_j\to+\infty$. 

Consider the unit-speed geodesic $\{\phi_{j,t}\}_{0\le t\le T_j}$ joining $\phi_\refe$ to $\phi_j$; this is $G$-minimal by Lemma~\ref{lem:Gmin}. By convexity of $F$ and~\eqref{equ:Fnonp}, we have 
\begin{equation}\label{equ:fup}
F(\phi_{j,t})-F(\phi_\refe)\le \frac{t}{T_j}(F(\phi_j)-F(\phi_\refe))\le \sigma_j t
\end{equation}
for $j$ large enough and $t\in [0,T_j]$. For any $T>0$, this implies that $\{\phi_{j,t}\}_{0\le t\le T}$ stays in a compact set of $\cE$. By the Arzel\`a--Ascoli theorem, we may thus assume, after passing to a subsequence, that $\{\phi_{j,t}\}$ converges on compact sets to a unit-speed geodesic ray $\{\phi_t\}_{t\ge 0}$ as $j\to\infty$, which remains $G$-minimal (see Lemma~\ref{lem:minclosed}). Since $F$ is lsc and $\sigma_j\to\sigma_G(F)$, \eqref{equ:fup} yields $F(\phi_t)-F(\phi_\refe)\le \sigma_G(F) t$, and $\f:=\phi_\infty$ therefore satisfies, as desired, $F_\ra(\f)\le\sigma_G(F)$. 
\end{proof}

\begin{proof}[Proof of Corollary~\ref{cor:crit}] The dichotomy is a direct consequence of Theorem~\ref{thm:crit}. If (i) holds, then any minimizing sequence of $F$ is bounded, and hence lies in the compact set $B\cap\{F\le c\}$ for some closed ball $B$ and $c\in\R$, and any limit point yields a minimizer of $F$.

Conversely, assume $F$ admits a unique minimizer modulo $G$, and pick the basepoint $\phi_\refe$ as a minimizer of $F$. If $F$ is not coercive modulo $G$, then by Theorem~\ref{thm:crit} there exists a $G$-minimal unit-speed geodesic ray $\{\phi_t\}$ emanating from $\phi_\refe$ along which $F$ is nonincreasing. But then $\phi_t$ is a minimizer of $F$ for all $t$, and hence lies in $H\phi_0$, which contradicts $G$-minimality. 
\end{proof}

%
\subsection{Admissible $G$-directions and the slope formula}\label{sec:slformula}
Consider first a proper isometric action on $\cE$ of an algebraic torus $T_\C\simeq(\C^\times)^r$, with maximal compact torus $T\simeq(S^1)^r$. By Proposition~\ref{prop:Buse}~(iv), the set $\cE^T$ of $T$-fixed points is non-empty. We further assume that the following holds:
\begin{equation}\label{equ:assum}
\emph{For any }\phi\in\cE^T\emph{ and }\xi\in\Lie T,\,\{e^{it\xi}\cdot\phi\}_{t\ge 0}\emph{ is a (distinguished) geodesic ray in }\cE.
\end{equation}
The direction $\f_\xi\in\cE_\ra$ of this geodesic ray is independent of the choice of $\phi$. Pick a norm on $\Lie T$. 

\begin{lem}\label{lem:qi1} The map $\xi\mapsto\f_\xi$ defines a quasi-isometric embedding $\Lie T\hto\cE_\ra$. 
\end{lem}

\begin{proof} Pick $\phi\in\cE^T$. The function $\xi\mapsto\dd(e^{i\xi}\phi,\phi)$ is continuous on $\Lie T$, positive on the unit sphere, and positively homogeneous of degree $1$ as a consequence of~\eqref{equ:assum}. Thus there exists $C>0$ such that 
$$
C^{-1}\|\xi-\xi'\|\le\dd(e^{i\xi}\phi,e^{i\xi'}\phi)\le C\|\xi-\xi'\|
$$ 
for all $\xi,\xi'\in\Lie T$. Passing to slopes at infinity, we infer
$$
C^{-1}\|\xi-\xi'\|\le\dd_\ra(\f_\xi,\f_{\xi'})\le C\|\xi-\xi'\|, 
$$
which proves the result. 
\end{proof}

By completeness of $\Lie T$, the image of the quasi-isometric embedding $\Lie T\hto\cE_\ra$ is closed; we denote it by
$$
\cD_T\subset\cE_\ra
$$
and call it the set of \emph{$T$-directions}. 

\medskip

Consider, more generally, a proper isometric action on $\cE$ of a complex linear algebraic group $G$ (not necessarily reductive), and assume~\eqref{equ:assum} holds for any compact torus $T\subset G$. 

\begin{defi} We define the set of \emph{toric $G$-directions} as 
$$
\cD_G:=\bigcup_T\cD_T\subset\cE_\ra,
$$
where $T\subset G$ ranges over all (maximal) compact tori. 
\end{defi}
Since all maximal compact tori of $G$ are conjugate, we equivalently have 
$\cD_G=\bigcup_{g\in G} g\cdot\cD_T$ for any given maximal compact torus $T\subset G$. 
\smallskip

Assume further that $G$ is \emph{reductive} (see Appendix~\ref{sec:reductive}). Pick a maximal compact subgroup $G_c\subset G$ and a $G_c$-invariant $\phi\in\cE$ (see Proposition~\ref{prop:Buse}~(iv)). Pick also a maximal compact torus $T\subset G_c$. The asymptotic cone of the symmetric space $\Sigma(G)=G/G_c$ then satisfies 
\begin{equation}\label{equ:coneconj}
\Sigma_\ra(G)\simeq\Lie G_c=\bigcup_{g\in G_c} g\cdot\Lie T,
\end{equation}
and~\eqref{equ:assum} thus implies that $\{e^{it\xi}\cdot\phi\}_{t\ge 0}$ is a geodesic ray in $\cE$ for any $\xi\in\Sigma_\ra(G)$. Denote by $\f_\xi\in\cE_\ra$ its direction. 

\begin{prop}\label{prop:Gdir} The map $\xi\mapsto\f_\xi$ defines a $G\times\R_{>0}$-equivariant quasi-isometric embedding $\Sigma_\ra(G)\hto\cE_\ra$ with image $\cD_G$. 
\end{prop}
\begin{proof} The first point is a direct consequence of~\eqref{equ:coneconj} and Lemma~\ref{lem:qi1}. The second one follows from the Iwasawa decomposition $G=G_c T G_c$, which implies $\cD_G=\bigcup_{g\in G_c} g\cdot\cD_T$.  
\end{proof}

Since $\Sigma_\ra(G)$ is complete, we infer: 

\begin{cor}\label{cor:dirclosed} If $G$ is reductive, then the set $\cD_G\subset\cE_\ra$ of toric $G$-directions is closed. 
\end{cor}

We may now state the following general slope formula  (compare~\cite[Theorem~B]{His} and~\cite[Proposition~5.15]{HLi23} for related results when $G$ is a torus).

\begin{thm}\label{thm:slopedG} Assume $G$ is reductive. For each almost geodesic ray $\{\phi_t\}$ in $\cE$ with direction $\f\in\cE_\ra$, we then have
$$
\lim_{t\to\infty} t^{-1} \dd_G(\phi_t)=\dd_\ra(\f,\cD_G)=\inf_{\xi\in\Sigma_\ra(G)}\dd_\ra(\f,\f_\xi). 
$$
Furthermore, the right-hand infimum is achieved. 
\end{thm}
\begin{proof} Pick $\xi\in\Sigma_\ra(G)\simeq\Lie G_c$. For each $t\ge 0$ we have 
$\dd(\phi_t,e^{it\xi}\phi_\refe)\ge \dd_G(\phi_t)$. Lemma~\ref{lem:almostd} thus yields 
$$
\dd_\ra(\f,\f_\xi)=\lim_{t\to\infty} t^{-1}\dd(\phi_t,e^{it\xi}\phi_\refe)\ge\limsup_{t\to\infty} t^{-1}\dd_G(\phi_t), 
$$
and hence $\dd_\ra(\f,\cD_G)\ge\limsup_{t\to\infty} t^{-1}  \dd_G(\phi_t)$. Conversely, pick a sequence $t_j\to+\infty$ such that 
$$
t_j^{-1} \dd_G(\phi_{t_j})\to L:=\liminf_{t\to\infty} t^{-1}  \dd_G(\phi_t). 
$$
For each $j$, pick $h_j\in G$ such that $ \dd_G(\phi_{t_j})=\dd(\phi_{t_j},h_j\phi_\refe)$, and consider its polar decomposition $h_j=e^{t_j\xi_j} h'_j$ with $\xi_j\in i\Lie G_c$ and $h'_j\in G_c$. By $G_c$-invariance of $\phi_\refe$ we get
$$
 \dd_G(\phi_{t_j})=\dd(\phi_{t_j},e^{it_j\xi_j}\phi_\refe). 
$$ 
Pick $j$ large enough, so that $t_j\ge 1$. The convexity of $t\mapsto\dd(\phi_t,e^{it\xi_j} \phi_\refe)$ yields 
$$
\dd(\phi_1,e^{i\xi_j} \phi_\refe)\le t_j^{-1} \dd(\phi_{t_j},e^{i t_j\xi_j} \phi_\refe)=t_j^{-1} \dd_G(\phi_{t_j}), 
$$
which converges to $L$, and hence is bounded. By properness of the $G$-action on $\cE$, $(\xi_j)$ thus stays in a fixed compact subset of $\Lie G_c$, and we may hence assume, after passing to a subsequence, that $\xi_j\to\xi\in\Lie G_c$ (in the vector space topology). Using again Busemann convexity, we have 
$$
t^{-1}  \dd(\phi_t,e^{it\xi_j} \phi_\refe)\le t_j^{-1}\dd(\phi_{t_j},e^{it_j\xi_j} \phi_\refe)=t_j^{-1} \dd_G(\phi_{t_j})
$$
for $t\in (0,t_j]$. Letting $j\to\infty$, this yields $t^{-1} \dd(\phi_t,e^{it\xi} \phi_\refe)\le L$ for any $t>0$. We conclude
\begin{multline*}
\dd_\ra(\f,\cD_G)\le\dd_\ra(\f,\f_\xi)=\lim_{t\to\infty} t^{-1}\dd(\phi_t,e^{it\xi} \phi_\refe)\\
\le L=\liminf_{t\to\infty} t^{-1}  \dd_G(\phi_t)\le\limsup_{t\to\infty} t^{-1}  \dd_G(\phi_t)\le\dd_\ra(\f,\cD_G), 
\end{multline*}
which proves the result, with $\xi$ achieving the infimum in the statement. 
\end{proof}

%
\subsection{Geodesic stability}
Here we combine the above results to state a refined version of~\cite[Theorem~3.4]{DR}. As above, $\cE$ is a Busemann convex, complete metric space endowed with a proper isometric action of a complex linear algebraic group $G$, assumed to satisfy condition~\eqref{equ:assum} of the previous section. Consider also a $G$-invariant function 
$$
F\colon\cE\to\R\cup\{+\infty\},
$$
assumed to be lsc and geodesically convex, with radial transform $F_\ra\colon\cE_\ra\to\R\cup\{+\infty\}$. By $G$-invariance of $F$,  $F_\ra$ automatically vanishes on the set $\cD_G\subset\cE_\ra$ of toric $G$-directions.  

\begin{defi}\label{defi:geodstab} We say that $F$ is: 
\begin{itemize}
    \item[(a)] \emph{geodesically stable modulo $G$} if $F_\ra(\f)\ge 0$ for each $\f\in\cE_\ra$, with equality iff $\f$ lies in $\cD_G$; 
    \item[(b)] \emph{uniformly geodesically stable modulo $G$} if $G$ is reductive and there exists $\sigma>0$ such that $F_\ra(\f)\ge\sigma\dd_\ra(\f,\cD_G)$ on $\cE_\ra$. 
\end{itemize}
\end{defi} 
As the terminology suggests, (b) indeed implies (a), since $\cD_G$ is closed when $G$ is reductive (see Corollary~\ref{cor:dirclosed}). We may now state: 

\begin{thm}\label{thm:mincrit} Assume that the $G$-invariant function $F\colon\cE\to\R\cup\{+\infty\}$ is strongly lsc and geodesically convex. Assume further that its minimizers are unique mod $G$, and that the existence of a minimizer implies that $G$ is reductive. Then the following are equivalent: 
\begin{itemize}
    \item[(i)] $F$ admits a minimizer; 
    \item[(ii)] $F$ is geodesically stable modulo $G$; 
    \item[(iii)] $F$ is uniformly geodesically stable modulo $G$. 
    \end{itemize}
\end{thm}
\begin{proof} By Corollary~\ref{cor:crit}, (i) implies $F\ge\sigma\dd_G-C$ for some $\sigma,C>0$. Passing to slopes at infinity yields $F_\ra(\f)\ge\sigma\dd_\ra(\f,\cD_G)$, by Theorem~\ref{thm:slopedG}. This proves (i)$\Rightarrow$(iii), and we already observed (iii)$\Rightarrow$(ii). Finally, assume (ii). If (i) fails, then Corollary~\ref{cor:crit} yields a $G$-minimal unit-speed geodesic ray $\{\phi_t\}$ along which $F$ is nonincreasing. Thus $F_\ra(\phi_\infty)\le 0$, and (ii) yields $\phi_\infty=\f_\xi$ for with $\xi\in \Lie T$ for some algebraic torus $T_\C\subset G$. The two geodesic rays $\{\phi_t\}$ and \{$e^{it\xi}\phi_0\}$ thus share the same origin and direction, and hence are equal. In particular, $\{\phi_t\}$ lies in a single $G$-orbit, and hence cannot be both $G$-minimal and non-constant, a contradiction. This shows (ii)$\Rightarrow$(i), and concludes the proof. 
\end{proof}

%
\section{Complex pluripotential theory}\label{sec:ppt}
In the remainder of this paper, $(X,L)$ will denote a smooth complex projective variety of dimension $n=\dim X$, equipped with an $\Q$-ample line bundle $L$ of volume $V=V_L:=(L^n)$. We write
$$
\cH=\cH(L)\subset C^\infty(L)
$$
for the (open, convex) set of smooth strictly psh metrics, and fix a reference metric $\phi_\refe\in\cH$ whenever convenient. 

As in~\cite{trivval,synthetic}, we write $x\lesssim y$ if $x\le C_n y$ for a constant $C_n>0$ only depending on $n$, and $x\approx y$ if $x\lesssim y$ and $y\lesssim x$.  We use~\cite{GZ,BBGZ,synthetic} as general references. 
%
\subsection{Monge--Amp\`ere operator and energy}
The (normalized) \emph{\ma operator}
$$
\MA\colon C^\infty(L)\to \Cz(X)^\vee
$$
is defined by 
$$
\MA(\phi):=V^{-1}(\ddc\phi)^{n}
$$
for $\phi\in C^\infty(L)$. Its restriction to $\cH$ takes values in the space of smooth positive probability measures. 

The \ma operator admits an {\EL} functional (see~\S\ref{sec:EL}) 
$$
\en\colon C^\infty(L)\to\R,
$$
called the \emph{Monge--Amp\`ere energy}, which we normalize by $\en(\phi_\refe)=0$. It satisfies the translation equivariance property
$$
\en(\phi+c)=\en(\phi)+c
$$
for $\phi\in C^\infty(L)$ and $c\in\R$ (see Example~\ref{exam:transinv}). 

The restriction $\en\colon\cH\to\R$ is further monotone increasing, \ie 
$\phi\le\p\Longrightarrow\en(\phi)\le\en(\p)$ for $\phi,\p\in\cH$, and concave, which by Lemma~\ref{lem:Fconc} amounts to 
\begin{equation}\label{equ:Econc}
\int(\phi-\p)\MA(\phi)\le\en(\phi)-\en(\p)\le\int(\phi-\p)\MA(\p). 
\end{equation}
%

\subsection{The Darvas metric}\label{sec:Darvas}
There exists a unique metric $\dd_1$ on $\cH$, called the \emph{Darvas metric}, such that for all $\phi,\p\in\cH$ we have 
\begin{itemize}
\item[(D1)] $\phi\ge\p\Longrightarrow\dd_1(\phi,\p)=\en(\phi)-\en(\p)$; 
\item[(D2)] $\dd_1(\phi,\p)=\inf\left\{\dd_1(\phi,\tau)+\dd_1(\tau,\p)\mid\tau\in\cH,\,\tau\le\min\{\phi,\p\}\right\}$,
\end{itemize}
Equivalently, setting
\begin{equation}\label{equ:Darvas}
\dd_1(\phi,\p):=\inf\left\{\en(\phi)+\en(\p)-2\en(\tau)\mid \tau\in\cH,\,\tau\le\min\{\phi,\p\}\right\}
\end{equation}
defines a metric on $\cH$, the most difficult part being to check the triangle inequality (see for instance~\cite{DDL}). 

Furthermore, by~\cite{Dar15}, the completion of $(\cH,\dd_1)$ can be identified with the space 
$$
\cE^1=\cE^1(L)\subset\PSH(L)
$$
of psh metrics on $L$ of \emph{finite energy}. We refer to the metric space topology of $\cE^1$ as the \emph{strong topology}, and its subspace topology induced by $\PSH(L)$ as the \emph{weak topology}. 

The \ma operator and energy admit unique continuous extensions 
$$
\MA\colon\cE^1\to \Cz(X)^\vee,\quad\en\colon\cE^1\to\R, 
$$
the latter still being an {\EL} of the former. The strong topology of $\cE^1$ can then be characterized as the coarsest refinement of the weak topology in which $\en\colon\cE^1\to\R$ becomes continuous. 

For all $\phi,\p\in\cE^1$ we have
\begin{equation}\label{equ:d1I1}
\dd_1(\phi,\p)\approx\ii_1(\phi,\p):=\int|\phi-\p|\left(\MA(\phi)+\MA(\p)\right), 
\end{equation}
see~\cite[Theorem~3.7]{DDL}. Moreover, 
$$
\ii_1(\phi,\p)=\ii_1(\phi,\max\{\phi,\p\})+\ii_1(\max\{\phi,\p\},\p),
$$
which implies
\begin{equation}\label{equ:dmax}
\dd_1(\max\{\phi,\p\},\tau)\lesssim\max\{\dd_1(\phi,\tau),\dd_1(\p,\tau)\}
\end{equation}
for all $\phi,\p,\tau\in\cE^1$.

The \emph{psh envelope} 
\begin{equation}\label{equ:pshenv}
\env(\p):=\sup\left\{\phi\mid \PSH(L)\ni\phi\le\p\right\}
\end{equation}
of any continuous metric $\p$ on $L$ is continuous and psh, and hence lies in $\cE^1$.  Trivially, 
\begin{equation}\label{equ:envlip}
\sup_X|\env(\p_1)-\env(\p_2)|\le\sup_X|\p_1-\p_2|. 
\end{equation}
We further have the \emph{orthogonality relation}
\begin{equation}\label{equ:ortho}
\int(\p-\env(\p))\MA(\env(\p))=0.
\end{equation}
In other words, the \ma measure $\MA(\env(\p))$ is supported in the \emph{contact locus} $\{\p=\env(\p)\}$.

%
\subsection{Measures of finite energy}

\begin{defi}\label{defi:finen} We say that a signed measure $\mu\in \Cz(X)^\vee$ has \emph{finite energy} if each function $\phi-\phi_\refe$ with $\phi\in\cE^1$ is $|\mu|$-integrable. 
\end{defi}
A probability measure $\mu$ has finite energy iff its \emph{pluricomplex energy}\footnote{The present notation follows~\cite{synthetic}.}
$$
\jj(\mu)=\jj(\mu_\refe,\mu):=\sup_{\phi\in\cE^1}\left(\en(\phi)-\int(\phi-\phi_\refe)\,\mu\right)\in [0,+\infty]
$$
is finite. We denote by 
$$
\cM^1\subset \Cz(X)^\vee
$$
the (convex) set of probability measures of finite energy. 
\begin{exam} For any $\phi\in\cE^1$, $\mu=\MA(\phi)$ has finite energy, with 
$\jj(\mu)=\jj(\phi_\refe,\phi)$. 
\end{exam} 
For any $\mu\in\cM^1$, the map
$$
\cE^1\to L^1(\mu)\quad\phi\mapsto\phi-\phi_\refe
$$
is {\sgh} (see Definition~\ref{defi:sgh}). More specifically:

\begin{lem}\label{lem:holdmu} For all $\mu\in\cM^1$ and $\phi_1,\phi_2\in\cE^1$ we have 
\begin{equation}\label{equ:holdmu}
\|\phi_1-\phi_2\|_{L^1(\mu)}\lesssim\dd_1(\phi_1,\phi_2)^{\a_n}\left(\jj(\mu)+\max_i\dd_1(\phi_i)\right)^{1-\a_n}
\end{equation}
with $\a_n:=2^{-n}$. 
\end{lem}

\begin{proof} Writing $|\phi_1-\phi_2|=2\left(\max\{\phi_1,\phi_2\}-\phi_2\right)+(\phi_2-\phi_1)$ and using~\eqref{equ:dmax}, it is enough to show
$$
\left|\int(\phi_1-\phi_2)\,\mu\right|\lesssim\dd_1(\phi_1,\phi_2)^{\a_n}\left(\jj(\mu)+\max_i\dd_1(\phi_i)\right)^{1-\a_n}. 
$$
This follows from
$$
\left|\int(\phi_1-\phi_2)\left(\mu-\MA(\phi_1)\right)\right|\lesssim\ii(\phi_1,\phi_2)^{\a_n}\left(\jj(\mu)+\max_i\dd_1(\phi_i)\right)^{1-\a_n},  
$$
see~\cite[(2.25)]{synthetic}, combined with~\eqref{equ:d1I1}. 
\end{proof}

Following~\cite[\S 2.7]{synthetic}, we next introduce a metric on $\cM^1$ by setting for $\mu,\nu\in\cM^1$ 
\begin{equation}\label{equ:dM1}
\done(\mu,\nu):=\sup_{\phi\in\cE^1,\,\ii(\phi)\le 1}\left|\int(\phi-\phi_\refe)(\mu-\nu)\right|. 
\end{equation}
By~\cite[(2.31)]{synthetic}, for all $\phi_1,\phi_2\in\cE^1$ we have
\begin{equation}\label{equ:estMA}
\left|\int(\phi_1-\phi_2)(\mu-\nu)\right|\lesssim \done(\mu,\nu)\max_i\left(1+\ii(\phi_i)\right)^{1/2}. 
\end{equation}
By~\cite[Proposition~2.28]{synthetic}, the metric space $(\cM^1,\done)$ is complete, and 
$$
\jj(\mu)\approx\done(\mu):=\done(\mu,\mu_\refe)\quad\text{with}\quad\mu_\refe:=\MA(\phi_\refe).
$$
The metric space topology of $\cM^1$ is called the \emph{strong topology}; it is characterized as the coarsest refinement of the weak topology in which $\jj\colon\cM^1\to [0,+\infty)$ becomes continuous. 

While the metric $\done$ depends on the choice of $L$ and $\phi_\refe$, the set $\cM^1$, its strong topology, and even the Lipschitz equivalence class of $(\cM^1,\done)$, are all independent of any choice, and thus intrinsically attached to $X$~\cite[Theorem~3.1]{synthetic}. 

By~\cite{GZ,BBEGZ}, the \ma operator induces a homeomorphism 
$$
\MA\colon\cE^1/\R\simto\cM^1
$$
for the strong topologies on both sides (see also Theorem~\ref{thm:wMA} below), a finite-energy version of~\cite{Yau78}. Using this we show: 

\begin{lem}\label{lem:weakcont} For any $\mu\in\cM^1$ and $\p\in\cE^1$, the restriction of $\phi\mapsto\int(\phi-\p)\,\mu$ to any bounded subset of $\cE^1$ is continuous in the weak topology.
\end{lem}

\begin{proof} Pick a sequence $(\tau_j)$ in $\cH$ such that $\mu_j:=\MA(\tau_j)\to\mu$ strongly. For each $j$, $\mu_j$ is smooth, and $\phi\mapsto\int(\phi-\p)\,\mu_j$ is thus continuous in the weak topology of $\cE^1$, which coincides with the $L^1(X)$ topology. By~\eqref{equ:estMA}, it converges to $\phi\mapsto\int(\phi-\p)\,\mu$ uniformly on bounded subsets of $\cE^1$, and the result follows.
\end{proof}

%
\subsection{Quasi-metrics and semi-globally H\"older operators}\label{sec:comppt}
By Lemma~\ref{lem:Fconc}, the concavity of $\en\colon\cE^1\to\R$ means that 
\begin{equation}\label{equ:J}
\jj(\phi,\p):=\int(\phi-\p)\MA(\p)+\en(\p)-\en(\phi)
\end{equation}
and
\begin{equation}\label{equ:I}
\ii(\phi,\p):=\int(\phi-\p)(\MA(\p)-\MA(\phi))
\end{equation}
are both non-negative. They are further translation-invariant, and satisfy 
\begin{equation}\label{equ:JI}
\jj(\phi,\p)\approx\ii(\phi,\p)\lesssim \ii_1(\phi,\p)\approx\dd_1(\phi,\p). 
\end{equation}
By convexity of $\jj(\cdot,\p)$, we thus have
\begin{equation}\label{equ:Iconv}
\ii((1-s)\phi_0+s\phi_1,\p)\lesssim\max_i\ii(\phi_i,\p)
\end{equation}
for all $\phi_0,\phi_1,\p\in\cE^1$ and $s\in [0,1]$. The symmetric functional $\ii$ further satisfies the quasi-triangle inequality 
$$
\ii(\phi_1,\phi_3)\lesssim\ii(\phi_1,\phi_2)+\ii(\phi_2,\phi_3), 
$$
and it descends to a quasi-metric on $\cE^1/\R$ that determines the quotient (strong) topology (see~\cite{BBEGZ,synthetic}).

On several occasions we shall consider translation-invariant measure-valued operators $\Ga\colon\cS\to\cM^1$, where $\cS\subset\cE^1$ is a translation-invariant subset, with the property that the induced map of quasi-metric spaces $(\cS/\R,\ii)\to(\cM^1,\done)$ is {\sgh} (see Definition~\ref{defi:sgh}). Let us spell out what that means: 

\begin{lem}\label{lem:sghop} With the above notation, $\Ga\colon(\cS/\R,\ii)\to(\cM^1,\done)$ is {\sgh} iff there exist $C>0$ and $\a\in (0,1)$ such that 
\begin{equation}\label{equ:sgh1}
\left|\int(\phi_1-\phi_2)\left(\Ga(\phi_3)-\Ga(\phi_4)\right)\right|\le C\ii(\phi_3,\phi_4)^\a\max_i\left(1+\ii(\phi_i)\right)^{1-\a} 
\end{equation}
for all $\phi_1,\phi_2\in\cH$ (or, equivalently, in $\cE^1$) and $\phi_3,\phi_4\in\cS$. When this holds we further have
\begin{equation}\label{equ:sgh2}
\|\phi_1-\phi_2\|_{L^1\left(\Ga(\phi_3)\right)}\le C\dd_1(\phi_1,\phi_2)^{\a_n}\max_i\left(1+\dd_1(\phi_i)\right)^{1-\a_n}
\end{equation}
for all $\phi_1,\phi_2\in\cE^1$ and $\phi_3\in\cS$, with $\a_n=2^{-n}$ and a possibly larger constant $C>0$. 
\end{lem}

\begin{proof} The first equivalence is an easy consequence of~\eqref{equ:estMA}, while~\eqref{equ:sgh2} follows from Lemma~\ref{lem:holdmu} together with~\eqref{equ:sghbd}. 
\end{proof}

\begin{exam} The \ma operator $\MA\colon(\cE^1/\R,\ii)\to(\cM^1,\done)$ is {\sgh}. 
\end{exam}

More generally, the (normalized) \emph{mixed Monge--Amp\`ere operator}
$\cH^n\to\cM^1$
$$
(\phi_1,\dots,\phi_n)\mapsto V^{-1}\ddc\phi_1\winter\ddc\phi_n
$$
also admits a unique (strongly) continuous extension $(\cE^1)^n\to\cM^1$, which satisfies 
\begin{equation}\label{equ:enmixed}
\done\left(V^{-1}\ddc\phi_1\winter\ddc\phi_n\right)\le C\max_i\left(1+\ii(\phi_i)\right)
\end{equation}
for all $\phi_i\in\cE^1$ and a constant $C>0$ only depending on $\phi_\refe$ (see~\cite[Theorem~3.4]{synthetic}). 

Similarly, given a smooth psh metric $\tau\in C^\infty(F)$ on another (necessarily nef) $\Q$-line bundle $F$, the operator $\cH^{n-1}\to\cM^1$
$$
(\phi_1,\dots,\phi_n)\mapsto V_G^{-1} \ddc\phi_1\winter\ddc\phi_{n-1}\wedge\ddc\tau 
$$
with $V_G:=(L^{n-1}\cdot G)$ extends by continuity to a map $(\cE^1)^{n-1}\to\cM^1$, which satisfies 
\begin{equation}\label{equ:enmixed2}
\done\left(V_G^{-1} \ddc\phi_1\winter\ddc\phi_{n-1}\wedge\ddc\tau\right)\le C\max_i\left(1+\ii(\phi_i)\right)
\end{equation}
for a constant $C>0$ only depending on $\phi_\refe$ and $\tau$. As a consequence we get: 
\begin{lem}\label{lem:quaden}
For all $\phi_0,\dots,\phi_n\in\cH$ we have
\begin{equation}\label{equ:quaden1}
0\le\int d(\phi_0-\phi_1)\wedge d^c(\phi_0-\phi_1)\wedge\ddc\phi_2\winter\ddc\phi_n\le C_1\ii(\phi_0,\phi_1)^{\a_n}\max_i\left(1+\ii(\phi_i)\right)^{1-\a_n}
\end{equation}
and
\begin{equation}\label{equ:quaden2}
0\le\int d(\phi_0-\phi_1)\wedge d^c(\phi_0-\phi_1)\wedge\ddc\phi_2\winter\ddc\phi_{n-1}\wedge\ddc\tau\le C_2\ii(\phi_0,\phi_1)^{\a_n}\max_i\left(1+\ii(\phi_i)\right)^{1-\a_n}
\end{equation}
for any smooth psh metric $\tau\in C^\infty(F)$ on a $\Q$-line bundle $F$, where $C_1$ (resp.~$C_2$) is a positive constant only depending on $\phi_\refe$ (resp.~$\phi_\refe$ and $\tau$). 
\end{lem}

\begin{proof} Integration-by-parts yields
$$
\int d(\phi_0-\phi_1)\wedge d^c(\phi_0-\phi_1)\wedge\ddc\phi_2\winter\ddc\phi_n=-\int (\phi_0-\phi_1)\ddc(\phi_0-\phi_1)\wedge\ddc\phi_2\winter\ddc\phi_n
$$
$$
=\int(\phi_0-\phi_1) \ddc\phi_1\winter\ddc\phi_n+\int(\phi_1-\phi_0)\ddc\phi_0\wedge\ddc\phi_2\winter\ddc\phi_n,
$$
The first estimate now follows from~\eqref{equ:enmixed} and~\eqref{equ:holdmu}, while~\eqref{equ:enmixed2} similarly yields
the second one.
\end{proof}

%
\subsection{Psh geodesics and reductive groups actions}\label{sec:pshgeod}
Next we show that the complete metric space $(\cE^1,\dd_1)$ fits into the framework of~\S\ref{sec:slformula}.

There is a 1--1 correspondence between paths $\{\phi_t\}_{t\in I}$ of (possibly singular) metrics on $L$ parametrized by an interval $I\subset\R$ and $S^1$-invariant metrics $\Phi$ on $L\times\DD_I$, where 
\begin{equation}\label{equ:DI}
\DD_I:=\{z\in\C^\times\mid -\log|z|\in I\},\quad\Phi(v,z)=\phi_{-\log|z|}(v),
\end{equation}
A path $\{\phi_t\}_{t\in I}$ of psh metrics on $L$ is called \emph{psh}\footnote{This is also known as a \emph{subgeodesic} in the literature.} if the corresponding metric $\Phi$ is psh over $X\times\DD_{\mathring{I}}$. A psh path $\{\phi_t\}$ such that $\phi_t\in\cE^1$ for all $t\in I$ is a geodesic in $\cE^1$ iff $\Phi$ is maximal locally over $\DD_{\mathring{I}}$. We then say that $\{\phi_t\}$ is a \emph{psh geodesic} in $\cE^1$.

By~\cite{Dar15}, the complete metric space $(\cE^1,\dd_1)$ is geodesic with respect to the distinguished class of (metric space) geodesics that are psh in the above sense. In fact,  $(\cE^1,\dd_1)$ is Busemann convex, as follows from~\cite[Proposition~5.1]{BDL1}. 

The group $\Aut(X,L)$ acts on $\cE^1$ by isometries, and by~\cite[Proposition~1.8]{BJT}, this action is proper. 

\begin{lem}\label{lem:mabgeod} For any algebraic torus $T_\C\subset\Aut(X,L)$, condition~\eqref{equ:assum} holds, \ie $\{e^{it\xi}\cdot\phi\}_{t\ge 0}$ is a psh geodesic ray for any $\xi\in\Lie T$ and $\phi\in\cE^{1,T}$. 
\end{lem}
\begin{proof}
    When $\phi\in\cH$ the result is a direct consequence of~\cite[Theorem~3.5]{Mab87}. In general, let $\tilde\Phi$ be the metric on $L\times\DD^\times$ defined as the pull-back of $\phi$. Then the $S^1$-invariant metric $\Phi$ on $L\times\C$ associated to $\{e^{it\xi}\cdot\phi\}_{t\ge 0}$ is the image of $\tilde\Phi$ under the local line bundle isomorphism of $L\times\DD^\times$ given by 
    \[
    (v,z)\mapsto(e^{(-\log z)\xi}\cdot v,z)
    \]
    for some local choice of holomorphic logarithm, the resulting metric being independent of that choice, by $T$-invariance of $\phi_\refe$. It follows that  $\Phi$ is psh and locally maximal, and the result follows. 
\end{proof}
As a result, \S\ref{sec:slformula} applies: to any algebraic subgroup $G\subset\Aut(X,L)$ is associated a set $\cD_G\subset\cE^1_\ra$ of toric $G$-directions, which is closed whenever $G$ is reductive. 

%
\section{Non-Archimedean pluripotential theory}\label{sec:NAppt}

We use~\cite{trivval,nakstab1,synthetic} as references for what follows. 

%
\subsection{The Berkovich space}\label{sec:Berkspace}
We denote by $X_\NA$ the Berkovich analytification of $X$ with respect to the trivial absolute value on the ground field $\C$~\cite{BerkBook}. Its elements can be viewed as \emph{semivaluations} on $X$, \ie valuations $v\colon\C(Z)^\times\to\R$ on the function field of a subvariety $Z\subset X$ (or, alternatively, as `tropical characters' on the semiring of coherent ideal sheaves of $X$, a description that still applies in the non-algebraic case, see~\cite{MP}). 

The space $X_\NA$ is compact Hausdorff, and contains as a dense subset the space $X_\div\subset X_\NA$ of \emph{divisorial valuations}, of the form $v=c\ord_D$ with $c\in\Q_{>0}$ and $D\subset Y$ a prime divisor on a smooth birational model of $X$. We include in $X_\div$ the \emph{trivial valuation} $v_\triv$, given by $v_\triv\equiv 0$ on $\C(X)^\times$; this is a fixed point for the scaling action 
$$
\R_{>0}\times X_\NA\to X_\NA\quad (c,v)\mapsto c v.
$$
 For any section $s$ of a line bundle $F$ on $X$ and $v\in X_\NA$, we can make sense of $v(s)\in [0,+\infty]$ by trivializing $F$ at the center of $v$ in $X$. Setting $|s|(v):=e^{-v(s)}$ determines a continuous function $|s|\colon X_\NA\to\R_{\ge 0}$
 (which coincides with the pointwise length of $s$ in the \emph{trivial metric} on $F_\NA$).        

We denote by 
$$
\cH_\NA=\cH_\NA(L)\subset \Cz(X_\NA)
$$ 
the set of \emph{Fubini--Study potentials} $\f\colon X_\NA\to\R$ (with respect to the trivial metric of $L_\NA$), of the form 
\begin{equation}\label{equ:RFS}
    \f:=m^{-1}\max_i\{\log|s_i|+\la_i\}
\end{equation}
for $m\in\Z_{>0}$, a finite set $(s_i)$ of sections of $mL$ without common zeroes, and $\la_i\in\Q$. The set $\cH_\NA$ spans a dense $\Q$-linear subspace
$$
\PL(X_\NA)\subset \Cz(X_\NA), 
$$
which is independent of $L$, and whose elements are called \emph{piecewise linear functions} (PL functions for short). The scaling action of $\R_{>0}$ on $X_\NA$ induces an action 
$$
\Q_{>0}\times\PL(X_\NA)\to\PL(X_\NA)\quad (c,\f)\mapsto c\cdot\f:=c\f(c^{-1}\bullet)
$$
that preserves $\cH_\NA$.

%
\subsection{Test configurations}

In this paper, a \emph{test configuration} $\cX$ for $X$ is defined as a normal projective variety equipped with a $\C^\times$-action together with a $\C^\times$-equivariant surjective morphism $\pi\colon\cX\to\P^1$ and an equivariant isomorphism 
\begin{equation}\label{equ:tc}
\pi^{-1}(\P^1\setminus\{0\})\simeq X\times(\P^1\setminus\{0\})
\end{equation}
over $\P^1$. We say that a test configuration $\cX$ is \emph{snc} if the variety $\cX$ is smooth and the central fiber
$$
\pi^{-1}(0)=\cX_0=\sum_E b_E E
$$ 
has simple normal crossing support $\cX_{0,\redu}=\sum_E E$. The set of test configurations for $X$ naturally forms a poset, in which snc test configurations are cofinal, by resolution of singularities.

To each irreducible component $E$ of the central fiber of a test configuration $\cX$ is attached a divisorial valuation 
$v_{E}\in X_\div$, defined as the pullback of $b_E^{-1}\ord_{E}$ by the function field inclusion $\C(X)\hto\C(\cX)$ induced by~\eqref{equ:tc}. Conversely, every element of $X_\div$ arises in this way. 

More generally, the \emph{dual complex} $\D_\cX$ of any snc test configuration $\cX$, defined as the dual intersection complex of the snc divisor $\cX_{0,\redu}$, admits a canonical embedding $\D_\cX\hto X_\NA$ realized by monomial valuations, the vertices of $\D_\cX$ being mapped to the valuations $v_E$. We further have a canonical retraction map $p_\cX\colon X_\NA\to\D_\cX$, and this induces a homeomorphism 
$$
X_\NA\simto\varprojlim_\cX\D_\cX
$$
compatible with the PL structures, in the sense that 
\begin{equation}\label{equ:PLdual}
\PL(X_\NA)=\bigcup_\cX p_\cX^\star\PL(\D_\cX).
\end{equation}
A \emph{test configuration} $\cL$ for $L$ is defined as a $\C^\times$-equivariant $\Q$-line bundle $\cL$ on a test configuration $\cX$ for $X$ together an equivariant $\Q$-line bundle isomorphism 
$$
\cL|_{\cX\setminus\cX_0}\simeq p_1^\star L|_{X\times(\P^1\setminus\{0\})}
$$
compatible with~\eqref{equ:tc}, where $p_1\colon X\times\P^1\to X$ denotes the first projection. Denoting by $\tau$ the coordinate on $\C\subset\P^1$, this induces for any $m\in\Z_{>0}$ sufficiently divisible an inclusion of $\C[\tau]$-modules
\begin{equation}\label{equ:sectc}
\Hnot(\cX,m\cL)\hto\Hnot(X,mL)\otimes\C[\tau^{\pm}]. 
\end{equation}
Assume that $\cX$ dominates $X\times\P^1$, and denote by $L_\cX$ the pullback of $L$ by the composition
$$
\cX\to X\times\P^1\to X. 
$$
We can then write $\cL=L_\cX+D$ for unique $\Q$-Cartier divisor $D$ on $\cX$ that is \emph{vertical}, \ie supported in $\cX_0$. Setting for each component $E$ of $\cX_0$ 
$$
\f_\cL(v_E):=b_E^{-1}\ord_E(D)
$$
defines a PL function $\f_\cL\in\PL(X_\NA)$, and every PL function arises in this way. 

If $\cL$ is relatively semiample (with respect to $\pi\colon\cX\to\P^1$), then $\f_\cL$ lies in $\cH_\NA$. Indeed, choose $m\in\Z_{>0}$ sufficiently divisible, a basis $(s_i)$ of $\Hnot(X,mL)$ and $\la_i\in\Z$ such that $s_i\tau^{\la_i}\in\Hnot(X,mL)\otimes\C[\tau^\pm]$ yields a $\C[\tau]$-basis of $\Hnot(\cX,m\cL)\otimes\C[\tau]$ (see~\eqref{equ:sectc}). Then
\begin{equation}\label{equ:FS}
\f_\cL=\frac 1m\max_i\{\log|s_i|+\la_i\}.
\end{equation}
Moreover, $(\cX,\cL)\mapsto\f_\cL$ restricts to a bijection between the set of (relatively) ample test configurations $(\cX,\cL)$ for $(X,L)$ and the set $\cH_\NA$ of Fubini--Study potentials. 

\begin{exam}\label{exam:prodtc} Any cocharacter $\rho\colon\C^\times\to\Aut(X,L)$ induces a \emph{product test configuration} for which $(\cX,\cL)|_{\C}$ is $\C^\times$-equivariantly isomorphic to $(X,L)\times\C$ via $\rho$. The associated Fubini--Study potential is given by~\eqref{equ:FS} for any basis of $\C^\times$-eigensections $s_i\in\Hnot(X,mL)$ of weight $\la_i\in\Z$. 
\end{exam}

The next observation plays a key role in the proof of this paper's main result.

\begin{lem}\label{lem:B+tc} For any test configuration $(\cX,\cL)$ for $(X,L)$ such that $\f_\cL>0$ on $X_\NA$, the relative and absolute augmented base loci of $\cL$ satisfy
$$
\B_+(\cL/\P^1)=\B_+(\cL)\subset\cX_0.
$$
\end{lem}
\begin{proof} Pick a test configuration $\cX'$ that dominates both $\cX$ and $X\times\P^1$, with corresponding morphisms of test configurations $\rho\colon\cX'\to\cX$, $\mu\colon\cX'\to X\times\P^1$. Write $\rho^\star\cL=L_{\cX'}+D$, and set $c:=\inf\f_\cL\in\Q_{>0}$. Then $E:=D-c\cX'_0$ satisfies $\f_E=\f_\cL-c\ge 0$, so that $E$ is an effective $\Q$-divisor (see~\cite[Lemma~2.8]{trivval}). On the other hand, $\rho^\star(L_\cX+c\cX_0)=L_{\cX'}+c\cX'_0=\mu^\star\cA$ where $\cA:=p_1^\star L+c p_2^\star\cO_{\P^1}(1)$ is ample on $X\times\P^1$. By Lemma~\ref{lem:B+bir} we infer
$$
\rho^{-1}(\B_+(\cL))=\B_+(\rho^\star\cL)\cup\Exc(\rho),\quad\B_+(\rho^\star\cL)=\B_+(\mu^\star\cA+E)\subset\Exc(\mu)\cup\supp E,
$$
and hence 
$$
\rho^{-1}(\B_+(\cL))\subset\Exc(\mu)\cup\supp E\cup\Exc(\rho)\subset\cX'_0. 
$$
Thus $\B_+(\cL)\subset\cX_0$, which implies $\B_+(\cL/\P^1)=\B_+(\cL)$ thanks to Lemma~\ref{lem:B+}. 
\end{proof}

%
\subsection{\ma operator and energy}
The (normalized) \emph{\na \ma operator} 
$$
\MA_\NA=\MA_{L,\NA}\colon\PL(X_\NA)\to \Cz(X_\NA)^\vee
$$
associates to each PL function $\f=\f_\cL$ attached to a test configuration $(\cX,\cL)$ the signed measure
$$
\MA_\NA(\f_\cL):=V^{-1}\sum_E b_E(\cL^n\cdot E)\d_{v_E}, 
$$
which has finite support in $X_\div$ and total mass $1$ (see Proposition--Definition~\ref{prop:wNAMA} below for a justification of this definition). In particular, 
$$
\MA_\NA(0)=\d_{v_\triv}.
$$
From there on, as illustrated by~\cite{synthetic}, the story can be developed in a quite parallel way to the complex case; we only emphasize here the aspects of relevance to this paper.  

The operator $\MA_\NA$ admits an {\EL} functional 
$$
\en_\NA=\en_{L,\NA}\colon\PL(X_\NA)\to\R,
$$
the \emph{\ma energy}, normalized by $\en_\NA(0)=0$, and simply given by 
$$
\en_\NA(\f_\cL)=\frac{(\cL^{n+1})}{(n+1)V}
$$
for any test configuration $\cL$ for $L$. For any $\f\in\cH_\NA$, $\MA_\NA(\f)$ is a probability measure, and $\en_\NA\colon\cH_\NA\to\R$ is monotone increasing and concave.

%
\subsection{The \na Darvas metric}
By~\cite[\S 5.3]{nakstab1} (see also~\cite{Rebgeod}), there exists a unique metric $\dd_{1,\NA}$ on $\cH_\NA$ that satisfies the analogue of (D1), (D2) in~\S\ref{sec:Darvas}, \ie
\begin{equation}\label{equ:DarvasNA}
\dd_{1,\NA}(\f,\p)=\inf\left\{\en_\NA(\f)+\en_\NA(\p)-2\en_\NA(\tau)\mid \tau\in\cH_\NA,\,\tau\le\min\{\f,\p\}\right\}
\end{equation}
for all $\f,\p\in\cH_\NA$. Further, the completion of $(\cH_\NA,\dd_{1,\NA})$ can be identified
with the space 
$$
\cE^1_\NA=\cE^1_\NA(L)\subset\PSH_\NA(L)
$$
of \emph{$L$-psh potentials of finite energy}. Here $\PSH_\NA(L)$ denotes the set of \emph{$L$-psh potentials} $\f\colon X_\NA\to\R\cup\{-\infty\}$, defined as pointwise limits of an decreasing net $(\f_i)$ in $\cH_\NA$ such that $\f\not\equiv-\infty$. Then $\f$ lies in $\cE^1_\NA$ iff $\en_\NA(\f)=\lim_i\en(\f_i)>-\infty$. 

Each $\f\in\PSH_\NA(L)$ is uniquely determined by its restriction to $X_\div$, which is finite valued, and $\PSH_\NA(L)$ is endowed with the topology of pointwise convergence on $X_\div$. 


\begin{rmk}\label{rmk:NABusemann} As in the complex case, one can show that the complete metric space $(\cE^1_\NA,\dd_{1,\NA})$ is geodesic with respect to a distinguished class of psh geodesics, and that 
$(\cE^1_\NA,\dd_{1,\NA})$ is Busemann convex with respect to this class. See~\cite{Reb25}.
\end{rmk} 

\bigskip
The next result makes explicit a canonical approximation procedure for functions in $\cE^1_\NA$ from~\cite{YTD}, based on multiplier ideals. 

\begin{prop}\label{prop:canapp} For any $\f\in\cE^1_\NA$ such that $\f\le 0$, consider the ideal sheaf $\cJ(\f)$ on $X\times\P^1$ consisting of all germs $f\in\cO_{X\times\P^1}$ for which there exists $\e>0$ such that 
$$
\sigma(v)(f)\ge -(1+\e)\f(v)-A_X(v)-1
$$
for all $v\in X_\div$, where $\sigma(v)\in (X\times\P^1)^\div$ denotes the Gauss extension of $v$ (see~\cite[\S1.3]{trivval}). Then:

\begin{itemize}
\item[(i)] $\cJ(\f)$ is a $\C^\times$-invariant coherent ideal sheaf cosupported on $X\times\{0\}$;
\item[(ii)] there exists $m_0\gg 1$ such that setting
$$
\f_m(v):=-(m+m_0)^{-1}\sigma(v)(\cJ(m\f));
$$
for $v\in X_\NA$ defines $\f_m\in\cH_\NA$ for each $m$, and $\f_m\to\f$ in $\cE^1_\NA$.
\end{itemize}
\end{prop}
\begin{proof} By~\cite[Theorem~6.6]{YTD}, we can find an $S^1$-invariant psh metric $\Phi$ on $p_1^\star L|_{\DD}$ with zero Lelong numbers away from $X\times\{0\}$, and such that $\sigma(v)(\Phi)=-\f(v)$ for all $v\in X_\div$. By~\cite[Theorem~B.5]{YTD}, $\cJ(\f)$ coincides with the multiplier ideal sheaf of $\Phi$, and is thus coherent, and (i), (ii) now follow from~\cite[Lemma~5.7]{YTD}. 
\end{proof}

If $\f\in\cE^1_\NA$ is invariant under a subgroup of $\Aut(X,L)$, then so is $\cJ(\f)$. We infer: 

\begin{cor}\label{cor:invdense} For any subgroup $S\subset\Aut(X,L)$, $\cH^S_\NA$ is dense in $\cE^{1,S}_\NA$ in the strong topology. 
\end{cor}

The \na \ma operator and energy both admit unique continuous extensions 
$$
\MA_\NA\colon\cE^1_\NA\to \Cz(X_\NA)^\vee,\quad\en_\NA\colon\cE^1_\NA\to\R, 
$$
the latter still being monotone increasing, concave, and an {\EL} functional of the former. The Darvas metric satisfies 
\begin{equation}\label{equ:d1I1NA}
\dd_{1,\NA}(\f,\p)\approx\ii_{1,\NA}(\f,\p):=\int|\f-\p|\left(\MA_\NA(\f)+\MA_\NA(\p)\right)
\end{equation}
for all $\f,\p\in\cE^1_\NA$, while the translation-invariant functional
\begin{equation}\label{equ:INA}
\ii_\NA(\f,\p):=\int(\f-\p)(\MA_\NA(\p)-\MA_\NA(\f))
\end{equation}
associated to the concave Euler--Lagrange functional $\en_\NA$ (see Lemma~\ref{lem:Fconc}) descends to a quasi-metric on $\cE^1_\NA/\R$ that determines the quotient topology. We use $0\in\cE^1_\NA$ as a basepoint and set
$$
\dd_{1,\NA}(\f):=\dd_{1,\NA}(\f,0),\quad\ii_\NA(\f):=\ii_\NA(\f,0). 
$$
 
%
\subsection{Measures of finite energy}

We say that a signed measure $\mu\in \Cz(X_\NA)^\vee$ has \emph{finite energy} if each $\f\in\cE^1_\NA$ is $|\mu|$-integrable. A probability measure $\mu$ lies in the set $\cM^1_\NA$ of probability measures of finite energy iff
\begin{equation}\label{equ:NAenmeas}
\jj_\NA(\mu)=\jj_{L,\NA}(\mu):=\sup_{\f\in\cE^1_\NA}\left(\en_\NA(\f)-\int\f\,\mu\right)\in [0,+\infty]
\end{equation}
is finite. The identity map $\cE^1_\NA\to L^1(\mu)$ then satisfies the {\sgh} estimate
$$
\|\f_1-\f_2\|_{L^1(\mu)}\lesssim\dd_{1,\NA}(\f_1,\f_2)^{\a_n}\left(\jj_\NA(\mu)+\max_i\dd_{1,\NA}(\f_i)\right)^{1-\a_n}
$$
for all $\f_1,\f_2\in\cE^1_\NA$, with $\a_n:=2^{-n}$, and $\f\mapsto\int\f\,\mu$ is further weakly continuous on bounded subsets of $\cE^1_\NA$.
`

Setting for $\mu,\nu\in\cM^1_\NA$ 
\begin{equation}\label{equ:dM1NA}
\d_{1,\NA}(\mu,\nu):=\sup_{\f\in\cE^1_\NA,\,\ii_\NA(\f)\le 1}\left|\int\f\,(\mu-\nu)\right|
\end{equation}
defines a complete metric on $\cM^1_\NA$, such that 
\begin{equation}\label{equ:estMANA}
\left|\int(\f_1-\f_2)(\mu-\nu)\right|\lesssim \d_{1,\NA}(\mu,\nu)\max_i\left(1+\ii_\NA(\f_i)\right)^{1/2}
\end{equation}
for all $\f_1,\f_2\in\cE^1_\NA$ and $\mu,\nu\in\cM^1_\NA$. Furthermore, 
$$
\d_{1,\NA}(\mu,\d_{v_\triv})\approx\jj_\NA(\mu), 
$$
and the \na \ma operator
$$
\MA_\NA\colon(\cE^1_\NA/\R,\ii_\NA)\to(\cM^1_\NA,\d_{1,\NA})
$$
is {\sgh} (see Lemma~\ref{lem:sghop}). 

\begin{exam} The set $\cM_\div$ of \emph{divisorial measures}, \ie probability measures $\mu:=\sum_i m_i\d_{v_i}$ with finite support in $X_\div$, is contained in $\cM^1_\NA$. 
\end{exam}

For such measures, the energy~\eqref{equ:NAenmeas} can be more concretely computed as the (finite dimensional) Legendre transform of an $S$-invariant type function, of common use in the `concrete K-stability' literature. We are grateful to Pietro Piccione for discussions on this result, which generalizes~\cite[Theorem~2.18]{nakstab2} (see also~\cite[\S 10.1]{MW}). 

\begin{prop}\label{prop:endivmeas} Pick a finite subset $\{v_1,\dots,v_r\}$ in $X_\div$, write $v_i:=d_i^{-1}\ord_{D_i}$ with $d_i\in\Q_{>0}$ and $D_i$ a prime divisor on a projective birational model $\rho\colon Y\to X$. Consider the simplex $\sigma:=\{a\in\R_{\ge 0}^r\mid \sum_i a_i=1\}$ and the divisorial measure $\mu_a=\sum_{i=1}^r a_i\d_{v_i}$ with $a\in\sigma$. Then $a\mapsto\jj_\NA(\mu_a)$ is convex on $\sigma$, and equal to the Legendre transform
$$
\jj_\NA(\mu_a)=\sup_{t\in\R_{\ge 0}^r} \left(S(t)-a\cdot t\right)
$$
of the concave function 
$$
S(t):=V^{-1}\int_0^\infty\vol\left(\rho^\star L-\sum_i d_i(\la-t_i)_+D_i\right)d\la.
$$
\end{prop}

\begin{proof} We use freely the notation and terminology of~\cite{nakstab1}. For any $t\in\R_{\ge 0}^r$ consider the divisorial norm 
$\chi_t:=\min_i\{\chi_{v_i}+t_i\}$. The associated filtration is given by 
$$
F^\la \Hnot(X,mL)\simeq\Hnot(Y,m\rho^\star L-\sum_i c_i^{-1}(\la-t_i)_+ D_i)
$$
and~\cite[Theorem~3.6~(v)]{nakstab1} thus yields $\vol(\chi_t)=\int_0^\infty\vol(\chi_t\ge\la)\,d\la=S(t)$. The function $\f_t:=\FS(\chi_t)\in\cE^1_\NA$ now satisfies 
\begin{equation}\label{equ:envt}
\f_t=\sup\{\f\in\cE^1_\NA\mid\f(v_i)\le t_i\}. 
\end{equation}
Then $\int\f_t\,\mu_a\le a\cdot t$, while~\cite[Theorem~5.1]{nakstab1} yields $\en(\f_t)=\vol(\chi_t)=S(t)$. By~\eqref{equ:NAenmeas} we infer  
$$
\jj_\NA(\mu_a)\ge\sup_{t\in\R_{\ge 0}^r}\left(S(t)-a\cdot t\right). 
$$
Conversely, pick $\f\in\cE^1_\NA$ such that $\MA_\NA(\f)=\mu_a$, and hence $\jj_\NA(\mu_a)=\en(\f)-\int\f\,\mu_a$. After adding a constant, we may assume $\f\ge 0$. Set $t_i:=\f(v_i)$. By~\eqref{equ:envt} we have $\f\le\f_t$, with equality on the support of $\MA_\NA(\f)$, and hence everywhere by the domination principle~\cite[Corollary~10.6]{trivval}. Thus $\en(\f)=S(t)$, $\int\f\,\mu_a=a\cdot t$, and hence $\jj_\NA(\mu_a)=S(t)-a\cdot t$. The result follows. 
\end{proof}

%
\subsection{Psh envelopes}\label{sec:env} 

Since $X$ is smooth, the \emph{envelope property} holds, \ie for any bounded-above family $\{\f_\a\}_\a$ in $\PSH_\NA(L)$, the usc upper envelope $\supstar_\a\f_\a$ also lies in $\PSH_\NA(L)$. 

As a consequence, for any $f\in \Cz(X_\NA)$ the \emph{psh envelope} 
$$
\env(f):=\sup\left\{\f\in\PSH_\NA(L),\,\f\le f\right\}
$$
is a continuous $L$-psh function, which thus lies in $\cE^1_\NA$. The psh envelope satisfies the \emph{orthogonality relation}
\begin{equation}\label{equ:orthoNA}
\int(f-\env(f))\MA_\NA(\env(f))=0,
\end{equation}
\ie $\MA_\NA(\env(f))$ is supported in the contact locus $\{f=\env(f)\}$. 

As a key ingredient in the proof of the main theorem, we further show:  
 
\begin{lem}\label{lem:suppenv} Pick $f\in\PL(X_\NA)$ with psh envelope $\env(f)$, and write $f=\f_\cL$ for an snc test configuration $(\cX,\cL)$. Then $\supp\MA(\env(f))$ is contained in the (finite) set of divisorial valuations $v_E\in X_\div$ attached to the irreducible components $E$ of $\cX_0$ that are not contained in the augmented base locus $\B_+(\cL/\P^1)$ (relative to the morphism $\cX\to\P^1$, see~\S\ref{sec:loci}). 
\end{lem}
While this will not be needed here, one can show as in~\cite[\S 3]{LiFuj} that $\supp\MA(\env(f))$ is in fact equal to the set of divisorial valuations in question. 

\begin{proof} Set $\mu:=\MA(\env(f))$. The PL function $f=\f_\cL$ satisfies $f=f\circ p_\cX$, and $f|_{\D_\cX}$ is linear on each face of $\D_\cX$. On the other hand, any $L$-psh function $\f$ satisfies $\f\le\f\circ p_\cX$, and $\f$ is convex on any face of $\D_\cX$ (see~\cite[Appendix~A]{trivval} and~\cite[\S7]{siminag}). 
This implies
$$
\env(f)=\sup\left\{\f\in\PSH(L_\NA)\mid \f(v_E)\le f(v_E)\text{ for all irreducible components }E\text{ of }\cX_0\right\}, 
$$
and the orthogonality relation can then be used to show that $\supp\mu$ is contained in the set of all $v_E$ (see~\cite[Lemma~8.5]{nama} or~\cite[\S 6.1]{NAGreen}). 

It remains to show that each component $E$ of $\cX_0$ contained in $\B_+(\cL/\P^1)$ satisfies $\mu(\{v_E\})=0$. Assume first $E\subset\B_-(\cL/\P^1)\subset\B_+(\cL/\P^1)$, \ie $\ord_E\|\cL/\P^1\|>0$ (see Lemma~\ref{lem:B-}). Since 
$$
(f-\env(f))(v_E)=b_E^{-1}\ord_E\|\cL/\P^1\|
$$
by~\cite[Lemma~5.22]{trivval}, it follows that $v_E$ does not belong to the contact locus $\{f=\env(f)\}$, and hence $\mu(\{v_E\})=0$ by~\eqref{equ:orthoNA}. 

Suppose now that $E$ is merely contained in $\B_+(\cL/\P^1)$. Thanks to Lemma~\ref{lem:B+bir} we may assume, after perhaps passing to a higher test configuration, that $\cX$ dominates $X\times\P^1$ and admits a vertical $\Q$-Cartier divisor $H$ such that $\cA:=L_\cX+H$ is relatively ample (see~\cite[Lemma~3.11]{trivval}). For any $\e>0$, 
$$
\cL_\e:=(1+\e)^{-1}(\cL-\e\cA)
$$
is a test configuration for $L$ defined on $\cX$. Since $\cA$ is relatively ample, the relative version of \eqref{equ:Bpert} yields 
$$
\B_+(\cL/\P^1)=\B_-((\cL-\e\cA)/\P^1)=\B_-(\cL_\e/\P^1) 
$$
for all $\e\in\Q_{>0}$ small enough. Set $f_\e:=\f_{\cL_\e}$, $\mu_\e:=\MA(\env(f_\e))$. Since $E$ is contained in $\B_-(\cL_\e/\P^1)$, the previous step yields $\mu_\e(\{v_E\})=0$.  Now 
$$
f_\e=(1+\e)^{-1}(f-\e\f_\cA)
$$
converges uniformly to $f$ as $\e\to 0$, so $\env(f_\e)\to\env(f)$ uniformly (see~\eqref{equ:envlip}), and hence $\mu_\e\to\mu$ weakly. Since these measures are all supported in a fixed finite set, it follows that $\mu_\e(\{v_E\})\to\mu(\{v_E\})$, which thus vanishes as desired. 

\end{proof}

Crucially, the envelope property implies that $\PSH_\NA(L)/\R$ is compact, and that the \na \ma operator induces a homeomorphism 
$$
\MA_\NA\colon\cE^1_\NA/\R\simto\cM^1_\NA
$$
for the strong topologies on both sides (see~\cite[Theorems~5.11, 12.8]{trivval} or Theorem~\ref{thm:NAwMA} below). 

\begin{exam} Pick a probability measure $\mu$ supported in the dual complex $\D_\cX\subset X_\NA$ of some snc test configuration $\cX$. Then $\mu$ lies in $\cM^1_\NA$, and the solution $\f\in\cE^1_\NA$ to $\MA_\NA(\f)=\mu$ belongs to $\Cz(X_\NA)$, and hence is a uniform limit of functions in $\cH_\NA$ (see~\cite[Theorems~A.4, 12.12]{trivval}).
\end{exam}

%
\subsection{Entropy}
The \emph{log discrepancy function} $A_X\colon X_\div\to\Q_{\ge 0}$ admits a maximal lsc extension 
$$
A_X\colon X_\NA\to [0,+\infty], 
$$
which is homogeneous with respect to the scaling action of $\R_{>0}$. It can be written as the pointwise limit of the increasing net of PL functions 
$$
A_\cX:=\f_{K^{\log}_{\cX/\P^1}}\in\PL(X_\NA), 
$$
indexed by the poset of all snc models $\cX$, where the relative log canonical bundle
$$
K^{\log}_{\cX/\P^1}:=K_{\cX/\P^1}+\cX_{0,\redu}-\cX_0
$$
is viewed as a test configuration for $K_X$. (See~\cite[Appendix~A]{nakstab2} for details.)

The \emph{\na entropy} of a positive measure $\mu$ on $X_\NA$ is defined as
\begin{equation}\label{equ:NAentsup}
\Ent_\NA(\mu):=\int_{X_\NA} A_X\,\mu=\sup_\cX\int A_\cX\,\mu\in [0,+\infty]. 
\end{equation} 
This yields an additive, lsc functional on the space of positive measures. Projection onto dual complexes further yields the following regularization result (see~\cite[Lemma~3.14]{nakstab2}). 

\begin{lem}\label{lem:cvent} Pick $\mu\in\cM^1_\NA$, and for each snc test configuration $\cX$ with associated dual complex $\D_\cX\hto X_\NA$ and retraction map $p_\cX\colon X_\NA\to\D_\cX$ set $\mu_\cX:=(p_\cX)_\star\mu$. Then $\mu_\cX\in\cM^1_\NA$ converges strongly to $\mu$, and $\Ent_\NA(\mu_\cX)\to\Ent_\NA(\mu)$. 
\end{lem}
%
\subsection{Real product test configurations}
Consider first an algebraic torus $T_\C\subset\Aut(X,L)$ and, for each $m\in\N$, the corresponding weight decomposition
$$
\Hnot(X,mL)=:R_m=\bigoplus_{\a\in M_\Z(T)} R_{m,\a}.
$$
To any $\xi\in N_\R(T)$ we associate the potential 
\begin{equation}\label{equ:FSprod}
\f_\xi:=\frac 1m\max_{\a\in M_\Z(T)}\max_{s\in R_{m,\a}\setminus\{0\}}\left(\log|s|+\langle\a,\xi\rangle\right),
\end{equation}
which is easily see to be independent of $m$ sufficiently divisible (by finite generation of $\bigoplus_m R_m$) and lies in $\cE^1_\NA\cap C^0(X_\NA)$. Fix a norm on $N_\R(T)$. 

\begin{prop}\label{prop:qi4} The map $\xi\mapsto\f_\xi$ defines a quasi-isometric embedding $N_\R(T)\hto\cE^1_\NA$. 
\end{prop}
\begin{proof} In the terminology of~\cite{nakstab1}, $\f_\xi=\FS(\chi_\xi)$ is the Fubini--Study potential associated to the finitely generated norm $\chi_\xi$ on $R(X,L)$ defined in terms of the weight decomposition by 
\begin{equation}\label{equ:norm}
\chi_\xi(s):=\min\{\langle\a,\xi\rangle\mid s_\a\ne 0\}
\end{equation}
for $R_m\ni s=\sum_\a s_\a$ with $s_\a\in R_{m,\a}$. For any two $\xi,\xi'\in N_\R(T)$, \cite[\S 5.2]{nakstab1} thus yields 
$$
\dd_{1,\NA}(\f_\xi,\f_{\xi'})=\dd_1(\chi_\xi,\chi_{\xi'}):=\int_\R|\la|\,\sigma(d\la)
$$
with $\sigma=\sigma(\chi_\xi,\chi_{\xi'})$ the relative spectral measure. We claim that $\sigma$ is the pushforward under the map $\langle\cdot,\xi-\xi'\rangle$ of the Duistermmaat--Heckman measure $\mu:=\DH^T(X,L)$ on the moment polytope $P\subset M_\R(T)$ (see~\S\ref{sec:DH}). This will imply
$$
\dd_{1,\NA}(\f_\xi,\f_{\xi'})=\int_P|\langle\xi-\xi',\a\rangle|\,\mu(d\a),
$$
and the result follows by equivalence of all norms on $N_\R(T)$. 
On the one hand, $\mu$ is the weak limit of the measures 
$$
\mu_m:=\frac{1}{\dim R_m}\sum_{\a\in M_\Z(T)}\dim R_{m,\a}\,\d_{m^{-1}\a}. 
$$
On the other hand, the decomposition $R_m=\bigoplus_\a R_{m,\a}$ is orthogonal for both $\chi$ and $\chi'$, with $\chi=\langle\a,\xi\rangle$ and $\chi'=\langle\a,\xi'\rangle$ on $R_{m,\a}\setminus\{0\}$, so $\sigma$ is the limit of the measures 
    \[
    \sigma_m:=\frac{1}{\dim R_m}\sum_{\a\in M_\Z(T)}\dim R_{m,\a}\,\delta_{m^{-1}\langle\a,\xi-\xi'\rangle},
    \]
    that is, the pushforward of $\mu_m$ under $\langle\cdot,\xi-\xi'\rangle$. The result follows.
\end{proof}

\begin{defi} We define the set $\cP_\R\subset\cE^1_\NA$ of \emph{real product test configurations} as the union of the images of all embeddings $N_\R(T)\hto\cE^1_\NA$ where $T_\C\subset\Aut(X,L)$ ranges over all (maximal) tori. 
\end{defi}
By Example~\ref{exam:prodtc}, using $N_\Z(T)$ in place of $N_\R(T)$ recovers the set 
$\cP_\Z\subset\cP_\R$ of (Fubini--Study potentials of) product test configurations. 

We similarly define the set $\cP_\Q\subset\cP_\R$ of \emph{rational product test configurations} by using $N_\Q(T)$ in the definition. As one easily checks, 
$$
\cP_\Q=\cP_\R\cap\cH_\NA
$$
can be identified with the set of ample test configurations $(\cX,\cL)$ that are also real product test configurations, which happens iff the normalization of the base change of $(\cX,\cL)$ by $z\mapsto z^d$ is a (usual) product test configuration for any sufficiently divisible $d$. 

Given a subgroup $S\subset\Aut(X,L)$, we slightly abusively call the elements of 
$$
\cP_\R^S=\cP_\R\cap\cE^{1,S}_\NA
$$
\emph{$S$-equivariant real product test configurations}. We similarly define the set $\cP_\Q^S=\cP_\R^S\cap\cH_\NA$ of \emph{$S$-equivariant rational product test configurations}.

\medskip

Next consider any complex reductive subgroup $G\subset\Aut(X,L)$. The choice of a maximal compact torus $T\subset G$ and a norm on $N_\R=N_\R(T)$ invariant under the Weyl group $W$ endows the conical Tits building $\Sigma_\NA(G)=\bigcup_{g\in G} g\cdot N_\R$ with a $G$-invariant complete metric (see Appendix~\ref{sec:reductive}). 

\begin{prop}\label{prop:qi3} The quasi-isometric embedding $N_\R\hto\cE^1_\NA$ $\xi\mapsto\f_\xi$ uniquely extends to a $G\times\R_{>0}$-equivariant quasi-isometric embedding $\Sigma_\NA(G)\hto\cE^1_\NA$. 
\end{prop}

\begin{proof} According to Lemma~\ref{lem:descend}, we need to show that $w\cdot\f_\xi=\f_{w\cdot\xi}$ for $w\in W=N_G(T)/T$, and that $\f_\xi$ is invariant under the unipotent subgroup $U_\g\simeq(\C,+)$ associated to a root $\g\in\Phi$ such that $\langle\g,\xi\rangle\ge0$ (since $\f_\xi$ is already $T$-invariant). As $\xi\mapsto\f_\xi$ factors through $\xi\mapsto\chi_\xi$ (see~\eqref{equ:norm}), it suffices to show these properties hold for $\chi_\xi$. 

Since $R_m$ is a $G$-module, its weight decomposition $R_m=\bigoplus_\a R_{m,\a}$ with respect to $T$ satisfies 
$$
R_{m,w\cdot\a}=w\cdot R_{m,\a},\quad U_\g\cdot R_{m,\a}\subset\bigoplus_{k\in\N} R_{m,\a+k\g}
$$
for all $\a$ (see~\eqref{equ:unipact}). This implies $\chi_{w\cdot\xi}=w\cdot\chi_\xi$, while~\eqref{equ:norm} shows that $\chi_\xi$ is $U_\g$-invariant, since $\langle\g,\xi\rangle\ge0$. 
\end{proof}

\begin{cor}\label{cor:prodclosed} Pick a subgroup $S\subset\Aut(X,L)$, and assume that $\Aut^S(X,L)$ is reductive. Then the set $\cP_\R^S\subset\cE^{1,S}_\NA$ of $S$-equivariant real product test configurations is closed. 
\end{cor}
\begin{proof} By construction, $\cP_\R^S$ coincides with the image of the quasi-isometric embedding $\Sigma_\NA(G)\hto\cE^1_\NA$ associated to the reductive group $G:=\Aut^S(X,L)$. The result follows since $\Sigma_\NA(G)$ is complete. 
\end{proof}

%
%
%
%
\section{Non-Archimedean limits of metrics and measures}\label{sec:NAlimmet}
We now look at situations where we can relate objects on $X_\NA$ and rays of objects on $X$.
\subsection{Maximal geodesic rays}\label{sec:geodrays} 
Of fundamental importance to our study is the relation between $\cE^1_\NA$ and psh geodesic rays in $\cE^1$.
Any psh geodesic ray $\{\phi_t\}$ in $\cE^1$ induces a unique function $\phi_\NA\in\cE^1_\NA$, whose restriction $\phi_\NA\colon X_\div\to\R$ encodes the Lelong numbers of the corresponding $S^1$-invariant psh metric $\Phi$ (see Theorem~6.2 and Proposition~4.1 in~\cite{YTD}.) This gives rise to an $\R_{>0}$-equivariant map
\[
\cE^1_\ra:=(\cE^1)_\ra\to\cE^1_\NA,
\]
which furthermore has a canonical section. Namely, given any $\f\in\cE^1_\NA$ and any reference metric $\phi_\refe\in\cH$, there exists a unique psh geodesic ray $\{\phi_t\}$ emanating from a $\phi_\refe$, such that $\phi_\NA=\f$ and $\psi_t\le\phi_t$ for any psh geodesic ray $\{\psi_t\}$ in $\cE^1$ with $\psi_0\le\phi_\refe$ and $\psi_\NA\le\f$. We say that $\{\phi_t\}$ is the \emph{maximal} (psh) geodesic ray directed by $\f$ and emanating from $\phi_\refe$. This leads to an isometric $\R_{>0}$-equivariant embedding
$$
(\cE^1_\NA,\dd_{1,\NA})\hto(\cE^1_\ra,\dd_{1,\ra}), 
$$
see~\cite[Theorem~6.6]{YTD} and~\cite[Theorem~B]{Reb}. Often we shall view $\cE^1_\NA$ as a subset of $\cE^1_\ra$. This subset is closed, since $\cE^1_\NA$ is complete. 

The above embedding is equivariant under the action of $\Aut(X,L)$. Given any compact group $S\subset\Aut(X,L)$, it restricts to an isometric embedding 
$$
\cE^{1,S}_\NA\hto\cE^{1,S}_\ra
$$
that is equivariant with respect to $\Aut^S(X,L)$, as one sees by choosing $\phi_\refe\in\cH^S$.  

We will later need the following concrete description of the maximal geodesic rays associated to the psh envelope of a (not necessarily ample) test configuration. It generalizes~\cite[Proposition~2.7]{Berm16}.

\begin{prop}\label{prop:tcgray} Pick a test configuration $(\cX,\cL)$ for $(X,L)$, with associated function $\f_\cL\in\PL(X_\NA)$ and psh envelope $\env(\f_\cL)\in\cE^1_\NA$. Let $\Phi$ be the metric on $p_1^\star L|_{\DD^\times}\simeq\cL|_{\DD^\times}$ corresponding to the maximal geodesic ray $\{\phi_t\}$ emanating from $\phi_\refe$ and directed by $\env(\f_\cL)$. Then $\Phi$ extends to a psh metric on $\cL|_{\DD}$, characterized as the largest psh metric whose restriction to $\cL|_{\partial\DD}$ coincides with the metric induced by $\phi_\refe$. 
\end{prop}
The proof relies on the following consequence of Proposition~4.1 and Lemma~4.4 of~\cite{YTD}.
\begin{lem}\label{lem:extpsh}Let $\cL$ be a test configuration for $L$, and $\{\phi_t\}$ a psh geodesic ray in $\cE^1$. The following are equivalent: 
\begin{itemize}
\item[(i)] the $S^1$-invariant psh metric $\Phi$ on $p_1^\star L|_{\DD^\times}\simeq\cL|_{\DD^\times}$ induced by $\{\phi_t\}$ extends to a psh metric on $\cL|_{\DD}$; 
\item[(ii)] 
$\phi_\NA\le\f_\cL$.
\end{itemize}
\end{lem}
\begin{proof}[Proof of Proposition~\ref{prop:tcgray}] 
As $\phi_\NA=\env(\f_\cL)\le\f_\cL$, it follows from Lemma~\ref{lem:extpsh} that $\Phi$ extends to $\cL$. Now let $\Psi$ be the largest $S^1$-invariant psh metric on $\cL|_\DD$ whose restriction to $\cL|_{\partial\DD}$ is given by $\phi_\refe$. Then $\Psi$ induces a psh path $\{\psi_t\}$ of metrics on $L$, and as $\Psi\ge\Phi$ it follows that $\{\psi_t\}$ is a psh geodesic ray in $\cE^1$. On the one hand, $\Phi\le\Psi$ implies $\phi_t\le\psi_t$ for all $t$. On the other hand, Lemma~\ref{lem:extpsh} gives $\psi_\NA\le\f_L$, and hence $\psi_\NA\le\env(\f_\cL)=\phi_\NA$, which by maximality of $\{\phi_t\}$ yields $\psi_t\le\phi_t$, and we are done.
\end{proof}

\begin{cor}\label{cor:charmray}
Let $\{\phi_t\}_{t\ge0}$ be a psh geodesic ray in $\cE^1$, emanating from a locally bounded reference metric $\phi_\refe$, and let $(\cX,\cL)$ be a semiample test configuration for $(X,L)$. Then the following are equivalent:
\begin{itemize}
\item[(i)] $\{\phi_t\}$ is a maximal geodesic ray directed by $\f_\cL$;
\item[(ii)] the $S^1$-invariant metric $\Phi$ on $L\times\DD^\times$ induced by $\{\phi_t\}$ extends to a locally bounded metric on $\cL|_{\DD}$.
\end{itemize}
\end{cor}
\begin{proof}
    Let $\Psi$ be the largest $S^1$-invariant psh metric on $\cL|_\DD$ whose restriction to $\cL|_{\partial\DD}$ coincides with $\phi_\refe$. As $\cL$ is semiample and $\phi_\refe$ locally bounded, $\Psi$ is locally bounded. 

    If~(i) holds, then $\Phi=\Psi$ by Proposition~\ref{prop:tcgray}, so $\Phi$ is locally bounded and~(ii) holds. Conversely, if~(ii) holds, then $\Phi$ is locally bounded, and $(\ddc\Phi)^{n+1}=0$ on $\cX|_{\DD^\times}$. Now $\Phi$ being locally bounded implies that the Monge--Amp\`ere measure $(\ddc\Phi)^{n+1}$ does not put mass on the pluripolar set $\cX_0$, so $(\ddc\Phi)^{n+1}=0$ on $\cX|_\DD$. As we also have $(\ddc\Phi)^{n+1}=0$ on $\cX|_\DD$ and $\Phi=\Psi$ on $\cL|_{\partial\DD}$, we must have $\Phi=\Psi$, so that~(i) holds.
\end{proof}

\subsection{Real product test configurations vs.~toric directions}\label{sec:Titsrays} 

Fix a compact subgroup $S\subset\Aut(X,L)$, and set $G:=\Aut^S(X,L)$. 

\begin{prop}\label{prop:qi5} The isometric embedding $\cE^{1,S}_\NA\hto\cE^{1,S}_\ra$ maps the set $\cP_\R^S$ of $S$-equivariant real test configurations onto the set $\cD_G$ of toric $G$-directions. 
\end{prop}
Note that this is consistent with Corollaries~\ref{cor:prodclosed} and~\ref{cor:dirclosed}, which respectively say that the subsets $\cP_\R^S\subset\cE^{1,S}_\NA$ and $\cD_G\subset\cE^{1,S}_\ra$ are closed when $G$ is reductive. 
\begin{proof} Given a compact torus $T\subset\Aut^S(X,L)$, we need to show that the map $N_\R(T)\hto\cE^1_\NA$ $\xi\mapsto\f_\xi$ of Proposition~\ref{prop:qi4} corresponds to the map
$N_\R(T)\simeq\Lie T\hto\cE^1_\ra$ in Lemma~\ref{lem:qi1} under the embedding $\cE^1_\NA\hto\cE^1_\ra$. As all the maps are continuous and $\R_{>0}$-equivariant, it suffices to consider an integral point $\xi\in N_\Z(T)$, corresponding to a cocharacter $\rho\colon\C^\times\to\Aut(X,L)$. 
    
Pick $m$ sufficiently divisible, and a basis of $\rho$-eigensections $s_j\in\Hnot(X,mL)$ with weights $\la_j\in\Z$. Then the image of $\xi\in N_\Z(T)$ in $\cE^1_\NA$ is the potential 
\[
  \f_\xi=\frac1m\max_j(\log|s_j|+\la_j)\in\cP_\Z^S, 
\] 
see Example~\ref{exam:prodtc}. If we assume, as we may, that the $s_j$ are $S$-invariant, then $\phi:=\frac1m\max_j\log|s_j|$ lies in $\cE^{1,S}$. We claim that the image of $\f_\xi$ in $\cE^{1,S}_\ra$ is represented by the ray
    \[
    \phi_t:=\frac1m\max_j(\log|s_j|+\la_jt).
    \]
To see this, note that $\phi_0=\phi$, and that the $S^1$-invariant metric $\Phi$ on $L\times\DD^\times$ defined by $\{\phi_t\}$ is given by $\Phi(\cdot,z)=\frac1m\max\{\log|s_j|+|z|^{-\la_j}\}$. Clearly, $\Phi$ is well-defined and psh on $L\times\C^\times$, and it is also locally maximal, since it is the image of the constant (and hence locally maximal) metric $\Phi_{\refe}(v,z)=\phi_\refe(v)$ by the automorphism of $L\times\C^\times$ given by $(v,z)\mapsto(\rho(z)\cdot v,z)$. 

    Thus $\{\phi_t\}$ is a psh geodesic ray, and we claim that it is directed by $\f_\xi$. 
    To this end, write $\P:=\P^{N_m-1}$, consider the embedding $X\times\C^\times\hto\P\times\C^\times$ given by 
    $$
    (x,z)\mapsto([s_j(x)z^{\la_j}]_j,z),
    $$
    and pick a test configuration $\cX$ such that $\cX$ dominates $X\times\P^1$, and the embedding extends to a morphism $\kappa\colon\cX\to\P\times\C$, and set $\cL:=m^{-1}\kappa^\star\cO(1)$. Then $\cL$ is semiample and $\f=\f_\cL$. Further, $\Phi$ is the restriction to $\cL|_{\C^\times}\simeq L\times\C^\times$ of the pullback of the Hermitian metric $\max_j\log|w_j|$ on $\cO_{\P}(1)$. In particular, $\Phi$ extends as a locally bounded psh metric on $\cL$, and Corollary~\ref{cor:charmray} now shows that $\{\phi_t\}$ defines a maximal geodesic ray directed by $\f_\cL=\f$.

    \smallskip
    On the other hand, the image of $\xi\in\Lie T\simeq N_\R(T)$ in $\cE^1_\ra$ is the direction of the geodesic ray $\{e^{it\xi}\cdot\phi\}$. But $e^{it\xi}\cdot\log|s_j|=\log|s_j|+t\la_j$, so $e^{it\xi}\cdot\phi=\frac1m\max_j(\log|s_j|+t\la_j)=\phi_t$, and we are done.
\end{proof}

\begin{rmk}
    One can view the above pair $(\{\phi_t\}_t,\f_\xi)$ as a continuous psh metric on the analytification of $L$ with respect to the hybrid norm on $\C$, see~\cite[\S2.4]{PS23}.
\end{rmk}
%
%

%
\subsection{Non-Archimedean limits of metrics}
Beyond geodesic rays, there are other situations when rays of metrics give rise to non-Archimedean metrics/potentials. Following~\cite[\S 3]{BHJ2}, we introduce: 

\begin{defi}\label{defi:anal} Given any $\Q$-line bundle $F$ on $X$, we say that a ray $\{\phi_t\}_{t\ge 0}$ of smooth metrics on $F$ has \emph{analytic singularities} if there exists a test configuration $\cF$ for $F$ such that the induced $S^1$-invariant metric $\Phi$ on $p_1^\star F|_{\DD^\times}\simeq\cF|_{\DD^\times}$ extends to a smooth metric on $\cF|_\DD$. 
\end{defi} 
The PL function 
$$
\phi_\NA:=\f_\cF\in\PL(X_\NA)
$$
is then uniquely determined by $\{\phi_t\}$ (see~\cite[Lemma~3.2]{BHJ2}); we call it the \emph{\na limit} of $\{\phi_t\}$. 

\begin{exam}\label{exam:raycst} For any $c\in\Q$, the ray of constant functions $f_t\equiv ct$ in $C^\infty(X)$ has analytic singularities, with \na limit $f_\NA\equiv c$. 
\end{exam}

\begin{exam}\label{exam:raytriv} Any ray $\{f_t\}$ with analytic singularities  in $C^\infty(X)$ satisfies 
\begin{equation}\label{equ:anbd}
\sup_X|f_t|=O(t). 
\end{equation}
In fact, one can show $t^{-1}\sup_X|f_t|\to\sup_{X_\NA}|f_\NA|$. 
\end{exam}

\begin{prop}\label{prop:anal} The following properties hold: 
\begin{itemize}
\item[(i)] if $\{\phi_{i,t}\}$ is a ray with analytic singularities in $C^\infty(F_i)$, $i=1,2$, then $\{\phi_{1,t}\pm\phi_{2,t}\}$ is a ray with analytic singularities in $C^\infty(F_1\pm F_2)$, with non-Archimedean limit $\phi_{1,\NA}\pm\phi_{2,\NA}$;
\item[(ii)] conversely, for any ray with analytic singularities $\{\phi_t\}$ in $C^\infty(F)$, we can find ample $\Q$-line bundles $L_1,L_2$ and psh rays with analytic singularities $\{\phi_{i,t}\}$ in $\cH(L_i)$, $i=1,2$ such that $F=L_1-L_2$ and $\phi_t=\phi_{1,t}-\phi_{2,t}$;
\item[(iii)] any $\f\in\cH_\NA$ can be realized as the \na limit of a psh ray with analytic singularities $\{\phi_t\}$ in $\cH$;
\item[(iv)] any $\f\in\PL(X_\NA)$ can be realized as the \na limit of a ray with analytic singularities $\{f_t\}$ in $C^\infty(X)$.
\end{itemize} 
If $\f$ in (iii), (iv) is invariant under a compact subgroup $S\subset\Aut(X,L)$, then the rays can be chosen to be $S$-invariant.  
\end{prop}

\begin{proof} The proof of (i) is straightforward, and (ii) follows from the fact that any smooth $S^1$-invariant on a test configuration $\cF$ for $F$ can be written as a difference of smoot strictly psh $S^1$-invariant metrics on ample test configurations. 

To prove (iii) pick $\f\in\cH_\NA$. Then $\f=\f_\cL$ for an ample test configuration $\cL$ for $L$, and the ray $\{\phi_t\}$ in $\cH$ associated to an $S^1$-invariant smooth strictly psh metric on $\cL$ similarly does the job. If $\f$ is $S$-invariant, then $\cL$ can be chosen $S$-invariant, and choosing an $S^1\times S$-metric thereon yields an $S$-invariant ray. Finally (iv) follows from (ii) and (iii). 
\end{proof}

\begin{exam}\label{exam:bdext} Assume $\{\phi_t\}_{t\ge 0}$ is a ray in $(\cE^1,\dd_1)$ that induces a locally bounded $S^1$-invariant metric on $\cL|_\DD$ for some semiample test configuration $\cL$ for $L$. Then $\{\phi_t\}$ is an almost geodesic ray directed by $\phi_\infty=\f_\cL\in\cH_\NA\subset\cE^1_\NA\subset\cE^1_\ra$. Indeed, the maximal geodesic ray $\{\p_t\}$ directed by $\f_\cL$ corresponds to a locally bounded metric on $\cL|_\DD$ as well (see Proposition~\ref{prop:tcgray}), and hence satisfies $\dd_1(\phi_t,\p_t)\le\sup_X|\phi_t-\p_t|=O(1)$. 
\end{exam}

This implies in turn: 
\begin{exam}\label{exam:analmost} Let $\{\phi_t\}$ be a ray in $\cH\subset C^\infty(L)$ with analytic singularities, and assume $\phi_\NA\in\cH_\NA$. Then $\{\phi_t\}$ is an almost geodesic ray in $\cE^1$ directed by $\phi_\infty=\phi_\NA\in\cE^1_\NA\subset\cE^1_\ra$. 
\end{exam}

%
\subsection{Non-Archimedean limits of measures}\label{sec:NAlimmeas}
Next we turn our attention from metrics to measures.
\begin{defi}\label{defi:NAmeas} We say that a ray $\{\mu_t\}_{t\ge 0}$ of signed measures on $X$ admits a \emph{\na limit} $\mu_\NA\in \Cz(X_\NA)^\vee$ if, for every ray $\{f_t\}$ in $C^\infty(X)$ with analytic singularities and \na limit $f_\NA\in\PL(X_\NA)$, we have
$$
\lim_{t\to\infty}t^{-1}\int_X f_t\,\mu_t=\int_{X_\NA}f_\NA\,\mu_\NA.
$$
\end{defi}
Since every element of $\PL(X_\NA)$ can be realized as the \na limit of a ray with analytic singularities in $C^\infty(X)$ (see Proposition~\ref{prop:anal}), and since $\PL(X_\NA)$ is dense in $\Cz(X_\NA)$ (see~\S\ref{sec:NAppt}), the \na limit $\mu_\NA$ is unique. By Example~\ref{exam:raycst}, we further have 
\begin{equation}\label{equ:masscv}
\lim_{t\to\infty} \int_X\mu_t=\int_{X_\NA} \mu_\NA. 
\end{equation}
As a consequence of~\cite[Lemma~3.9]{BHJ2}, we get: 

\begin{exam}\label{exam:limMA} For any ray $\{\phi_t\}$ in $C^\infty(L)$ with analytic singularities, the signed measure $\MA_\NA(\phi_\NA)$ is the \na limit of the ray of signed measures $\{\MA(\phi_t)\}_{t\ge 0}$. 
\end{exam}

The above concept of \na limit of measures is in fact a reformulation of the notion of weak convergence of measures in the \emph{hybrid space} 
$$
X_\hyb\simeq\left(X\times(0,1]\right)\coprod \left(X_\NA\times\{0\}\right), 
$$
which is a compact Hausdorff space for the hydrid topology (see for instance~\cite{konsoib,PS22,Hybrid}). 
More specifically, $\mu_\NA\in \Cz(X_\NA)^\vee$ is the \na limit of a ray $\{\mu_t\}$ in $\Cz(X)^\vee$ iff $\mu_t$, viewed as a measure on $X\times\{e^{-t}\}\subset X_\hyb$, converges weakly in $X_\hyb$ to $\mu_\NA$ viewed as a measure on $X_\NA\times\{0\}\subset X_\hyb$. Indeed, each ray $\{f_t\}$ with analytic singularities in $C^\infty(X)$ induces a continuous function $f\colon X_\hyb\to\R$ by setting 
$$
f|_{X\times\{e^{-t}\}}:=t^{-1} f_t,\quad f|_{X_\NA\times\{0\}}=f_\NA,
$$
and one shows using the Stone--Weierstrass theorem that such functions form a dense subset of $\Cz(X_\hyb)$ (see~\cite[\S 6]{Hybrid}). 

As an illustration, \cite[Lemma~3.11]{BHJ2} (or~\cite{konsoib}) yield: 

\begin{exam}\label{exam:adapted} Let $\{\p_t\}_{t\ge0}$ be a ray in $C^\infty(K_X)$ with analytic singularities, and denote by $\nu_t:=e^{2\p_t}$ the associated ray of smooth volume forms on $X$. Then there exists $d\in\{0,\dots,n\}$ such that $t^{-d}\nu_t$ admits as \na limit a Lebesgue-type measure on a subcomplex of the dual complex of some snc test configuration. 
\end{exam}

The next result follows from Pille--Schneider's handy criterion for hybrid convergence~\cite[Theorem~B]{PS22} (see also~\cite[Proposition~6.4]{Hybrid}).

\begin{prop}\label{prop:PS} Let $\{\mu_t\}_{t\ge 0}$ be a ray of positive measures on $X$. Denote by $\{\tmu_z\}_{z\in\DD^\times}$ the corresponding $S^1$-invariant family of measures on $X\times\DD^\times$, \ie 
$$
\tmu_z:=\mu_{-\log|z|}\quad\text{on}\quad X\times\{z\},
$$
and assume given an snc test configuration $\cX$ with central fiber $\cX_0=\sum_E b_E E$ such that:
\begin{itemize}
\item[(i)] as $z\in\DD^\times$ tends to $0$, $\tmu_z$ converges weakly in $\cX$ to a positive measure $\tmu_0$ on $\cX_0$;
\item[(ii)] $\tmu_0$ puts no mass on nowhere dense Zariski closed subsets of $\cX_0$.
\end{itemize}
Then $\mu_\NA:=\sum_E\tmu_0(E)\d_{v_{E}}$ is the \na limit of $\{\mu_t\}$. 
\end{prop}

%
\subsection{Non-Archimedean transforms}\label{sec:NAtrans}
Recall from~\S\ref{sec:pshgeod} that the space $(\cE^1,\dd_1)$ of finite energy metrics is Busemann convex, and that its \na counterpart $(\cE^1_\NA,\dd_{1,\NA})$ naturally sits in the asymptotic cone $(\cE^1_\ra,\dd_{1,\ra})$ (see~\S\ref{sec:geodrays}). Relying on this, our goal is to introduce a general mechanism producing \na analogues of measure-valued operators and energy functionals. 

In what follows we fix a compact subgroup $S\subset\Aut(X,L)$, and denote by $\cH^S\subset\cH$, $\cH_\NA^S=\cH_\NA^{S_\C}\subset\cH_\NA$ the subspaces of $S$-fixed points.

\begin{defi}\label{defi:NAtrfunc} Given two functionals 
$$
F\colon\cH^S\to\R\cup\{+\infty\},\quad F_\NA\colon\cH^S_\NA\to\R\cup\{+\infty\},
$$
we say that $F_\NA$ is the \emph{\na transform} of $F$ if 
$$
F_\NA(\phi_\NA)=\lim_{t\to\infty} t^{-1} F(\phi_t)
$$
for any ray with analytic singularities $\{\phi_t\}$ in $\cH^S$ such that $\phi_\NA\in\cH^S_\NA$. 
\end{defi}
Of course, a similar definition makes sense for functionals defined on $\cH^S\times\cH^S$ etc...

Since any $\f\in\cH^S_\NA$ can be realized as the \na limit $\f=\phi_\NA$ of some ray with analytic singularities $\{\phi_t\}$ lying in $\cH^S$ (see Propostion~\ref{prop:anal}), $F_\NA$ is uniquely determined by $F$. It is further necessarily homogeneous with respect to the scaling action of $\Q_{>0}$ on $\cH^S_\NA$ (see~\cite[Proposition~3.5]{BHJ2}). 
\begin{defi}\label{defi:NAtrmeas} Given two operators
$$
\Ga\colon\cH^S\to \Cz(X)^\vee,\quad \Ga_\NA\colon\cH^S_\NA\to \Cz(X_\NA)^\vee,
$$
we say that $\Ga_\NA$ is the \emph{\na transform} of $\Ga$ if, for any ray $\{\phi_t\}_{t\ge 0}$ in $\cH^S$ with analytic singularities and such that $\phi_\NA\in\cH^S_\NA$, the measure $\Ga_\NA(\phi_\NA)$ on $X_\NA$ is the \na limit of the ray of measures $\{\Ga(\phi_t)\}_{t\ge 0}$ on $X$ (see Definition~\ref{defi:NAmeas}). 
\end{defi}
Again, $\Ga_\NA$ is uniquely determined by $\Ga$. 

\begin{exam} The \na \ma operator $\MA_\NA\colon\cH_\NA\to\cM^1_\NA$ is the \na transform of $\MA\colon\cH\to\cM^1$ (see Example~\ref{exam:limMA}). 
\end{exam}
In view of~\eqref{equ:INA}, this implies: 

\begin{exam}\label{exam:NAfunc} The functional $\ii_\NA\colon\cH_\NA\times\cH_\NA\to\R_{\ge 0}$ is the \na transform of $\ii\colon\cH\times\cH\to\R_{\ge 0}$. Similarly,  the \na \ma energy $\en_\NA$, Ricci energy $\rr_\NA$ and entropy $\ent_\NA$ on $\cH_\NA$ are the \na transforms of their complex counterparts $\en$, $\enR$, $\ent$ on $\cH$ (see~\cite[Theorem~3.6]{BHJ2}). 
\end{exam}

We next provide a general extension result for measure-valued operators and their \na transforms. We consider a translation-invariant operator 
$$
\Ga\colon\cH^S\to\cM^1,
$$
assumed to admit a \na transform $\Ga_\NA\colon\cH^S_\NA\to\cM^1_\NA$.

\begin{lem}\label{lem:NAtrans1} Assume $\Ga\colon(\cH^S/\R,\ii)\to(\cM^1,\done)$ is {\sgh} (see Lemma~\ref{lem:sghop}). Then:
\begin{itemize}
\item[(i)] both $\Ga$ and $\Ga_\NA$ uniquely extend to translation-invariant operators
$$
\Ga\colon\cE^{1,S}\to\cM^1,\quad\Ga_\NA\colon\cE^{1,S}_\NA\to\cM^1_\NA
$$
such that 
$\Ga\colon(\cE^{1,S}/\R,\ii)\to(\cM^1,\done)$ and $\Ga_\NA\colon(\cE^{1,S}_\NA/\R,\ii_\NA)\to(\cM^1_\NA,\d_{1,\NA})$ are both {\sgh}; 
\item[(ii)] for all almost geodesic rays $\{\phi_{1,t}\}$, $\{\phi_{2,t}\}$ in $\cE^1$, $\{\phi_{3,t}\}$ in $\cE^{1,S}$, such that $\phi_{1,\infty},\phi_{2,\infty}\in\cE^1_\NA$, $\phi_{3,\infty}\in\cE^{1,S}_\NA$, we further have 
\begin{equation}\label{equ:slopeint}
\lim_{t\to\infty} t^{-1}\int_X(\phi_{1,t}-\phi_{2,t})\Ga(\phi_{3,t})=\int_{X_\NA}(\phi_{1,\infty}-\phi_{2,\infty})\Ga_\NA(\phi_{3,\infty}). 
\end{equation}
In particular, the measure $\Ga_\NA(\phi_{3,\infty})$ on $X_\NA$ is the \na limit of the ray of measures $\{\Ga(\phi_{3,t})\}$ on $X$. 
\end{itemize}
\end{lem}

\begin{proof} Since $(\cM^1,\done)$ is complete and $\Ga$ is {\sgh}, $\Ga$ admits a unique {\sgh} extension $\Ga\colon(\cE^{1,S}/\R,\ii)\to(\cM^1,\done)$. 

In the \na case, it will similarly suffice to check that $\Ga_\NA$ is {\sgh} on $\cH^S_\NA$ with respect to $\ii_\NA$ and $\d_{1,\NA}$, which amounts to 
the existence of $C>0$, $\a\in(0,1)$ such that 
\begin{equation}\label{equ:sghNA}
\left|\int(\f_1-\f_2)\left(\Ga_\NA(\f_3)-\Ga_\NA(\f_4)\right)\right|\le C\ii_\NA(\f_3,\f_4)^\a\max_i\left(1+\ii_\NA(\f_i)\right)^{1-\a}
\end{equation}
for all $\f_1,\f_2\in\cH_\NA$, $\f_3,\f_4\in\cH^S_\NA$ (see Lemma~\ref{lem:sghop}, which applies without change to the \na case). By Proposition~\ref{prop:anal}, we can choose a ray with analytic singularities $\{\phi_{i,t}\}_t$ in $\cH$ (resp.~$\cH^S$) such that $\phi_{i,\NA}=\f_i$. Since $\Ga_\NA(\f_i)$ is the \na limit of $\{\Ga(\phi_{i,t})\}_t$ and $\{\phi_{1,t}-\phi_{2,t}\}$ is a ray in $C^\infty(X)$ with analytic singularities and \na limit $\f_1-\f_2$, we have
$$
\lim_{t\to\infty} t^{-1}\int(\phi_{1,t}-\phi_{2,t})\Ga(\phi_{i,t})=\int(\f_1-\f_2)\Ga_\NA(\f_i). 
$$
By Example~\ref{exam:NAfunc}, we further have 
$$
t^{-1}\ii(\phi_{i,t},\phi_{j,t})\to\ii_\NA(\f_i,\f_j),\quad t^{-1}\ii(\phi_{i,t})\to\ii_\NA(\f_i).
$$
Since $\Ga$ is {\sgh}, \eqref{equ:sgh1} yields 
$$
t^{-1}\left|\int(\phi_{1,t}-\phi_{2,t})\left(\Ga(\phi_{3,t})-\Ga(\phi_{4,t})\right)\right|\le C\left(t^{-1}\ii(\phi_{3,t},\phi_{4,t})\right)^\a\left(t^{-1}+\max_i t^{-1}\ii(\phi_{i,t})\right)^{1-\a}, 
$$
Letting $t\to\infty$ yields \eqref{equ:sghNA}\footnote{In fact without the additive constant $1$ on the right, but this is anyway automatic by homogeneity wrt the action of $\R_{>0}$.}, and concludes the proof of (i). 

To prove (ii), for $i=1,2$ (resp.~$i=3$) pick a sequence $(\f_i^j)_j$ in $\cH_\NA$ (resp~$\cH^S_\NA$) such that $\lim_j\dd_\ra(\f_i^j,\phi_{i,\infty})=0$, and for each $j$ choose a ray with analytic singularities $\{\phi_{i,t}^j\}$ in $\cH$ (resp.~$\cH^S$) with \na limit $\f_i^j$. Writing 
$$
\int(\phi_{1,t}-\phi_{2,t})\Ga(\phi_{3,t})-\int(\phi_{1,t}^j-\phi_{2,t}^j)\Ga(\phi_{3,t}^j)
$$
$$
=\int(\phi_{1,t}-\phi_{2,t})\left(\Ga(\phi_{3,t})-\Ga(\phi_{3,t}^j)\right)+\sum_{i=1,2}\int(\phi_{i,t}-\phi_{i,t}^j)\Ga(\phi_{3,t}^j),
$$
and using~\eqref{equ:sgh1}, \eqref{equ:sgh2}, we get
\begin{equation}\label{equ:bihold}
\left|\int(\phi_{1,t}-\phi_{2,t})\Ga(\phi_{3,t})-\int(\phi_{1,t}^j-\phi_{2,t}^j)\Ga(\phi_{3,t}^j)\right|\le C\left[\e_{j,t}^\a M_{j,t}^{1-\a}+\e_{j,t}^{\a_n} M_{j,t}^{1-\a_n}\right]
\end{equation}
with
$$
\e_{j,t}:=\max_i\dd_1(\phi_{i,t},\phi_{i,t}^j),\quad M_{j,t}:=1+\max_i\max\left\{\dd_1(\phi_{i,t}),\dd_1(\phi_{i,t}^j)\right\}.
$$
Since $\Ga_\NA$ is the \na transform of $\Ga$, we get for each $j$ 
$$
\lim_t t^{-1}\int(\phi_{1,t}^j-\phi_{2,t}^j)\Ga(\phi_{3,t}^j)=\int(\f_1^j-\f_2^j)\Ga_\NA(\f_3^j). 
$$
On the other hand, since $\{\phi_{i,t}^j\}$ is an almost geodesic ray directed by $\f_i^j$ (see Example~\ref{exam:analmost}), Lemma~\ref{lem:almostd} yields
$$
\lim_t t^{-1}\e_{j,t}=\e_j:=\max_i\dd_{1,\ra}(\f_i,\f_i^j),\quad\lim_t t^{-1} M_{j,t}=M_j:=\max_i\max\left\{\dd_{1,\ra}(\f_i),\dd_{1,\ra}(\f_i^j)\right\}.
$$
By~\eqref{equ:bihold} we thus get
$$
\int(\f_1^j-\f_2^j)\Ga_\NA(\f_3^j)-C\left[\e_j^\a M_j^{1-\a}+\e_j^{\a_n} M_j^{1-\a_n}\right]
$$
$$
\le\liminf_t t^{-1}\int(\phi_{1,t}-\phi_{2,t})\Ga(\phi_{3,t})\le\limsup_t t^{-1}\int(\phi_{1,t}-\phi_{2,t})\Ga(\phi_{3,t})
$$
$$
\le \int(\f_1^j-\f_2^j)\Ga_\NA(\f_3^j)+C\left[\e_j^\a M_j^{1-\a}+\e_j^{\a_n} M_j^{1-\a_n}\right].
$$
Since we have already proved that $\Ga_\NA$ is {\sgh}, and hence continuous, we have $\lim_j \int(\f_1^j-\f_2^j)\Ga_\NA(\f_3^j)=\int(\phi_{1,\infty}-\phi_{2,\infty})\Ga_\NA(\phi_{3,\infty})$. Further, $M_j$ converges to $\max_i\dd_{1,\ra}(\f_i)$, hence remains bounded, while $\e_j\to 0$. Letting $j\to\infty$ thus yields~\eqref{equ:slopeint}, which concludes the proof. 
\end{proof}

\begin{lem}\label{lem:NAtrans2} Assume as in Lemma~\ref{lem:NAtrans1} that $\Ga\colon(\cH^S/\R,\ii)\to(\cM^1,\done)$ is {\sgh}, and suppose further that it admits an {\EL} functional $F\colon\cH^S\to\R$. Then:
\begin{itemize}
\item[(i)] $F$ admits a unique {\sgh} extension to $\cE^{1,S}$, which further is an {\EL} functional for $\Ga\colon\cE^{1,S}\to\cM^1$; 
\item[(ii)] $F$ admits a \na transform $F_\NA\colon\cH^s\to\R$ with a unique {\sgh} extension to $\cE^{1,S}_\NA$, which further is an {\EL} functional for $\Gamma_\NA:\cE^{1,S}_\NA\to\R$;
\item[(iii)] for any almost geodesic ray $\{\phi_t\}$ in $\cE^{1,S}$ whose direction $\phi_\infty$ lies in $\cE^{1,S}_\NA$, we further have 
\begin{equation}\label{equ:slopeF}
F_\NA(\phi_\infty)=\lim_{t\to\infty} t^{-1} F(\phi_t);
\end{equation}
\item[(iv)] if $F$ is concave (with respect to linear segments) on $\cE^{1,S}$, then $F_\NA$ is concave on $\cE^{1,S}_\NA$. 
\end{itemize}
\end{lem}

\begin{proof} Define $F_\NA\colon\cE^{1,S}_\NA\to\R$ by setting 
$$
F_\NA(\f):=\int_0^1 ds\int\f\,\Ga_\NA(s\f)
$$
for $\f\in\cE^{1,S}_\NA$. Since $\Ga_\NA\colon\cE^{1,S}_\NA\to\cM^1_\NA$ is {\sgh}, it is straightforward to see that $F_\NA$ is {\sgh} as well. To show that $F_\NA$ is an {\EL} functional for $\Ga_\NA$ we need to show
\begin{equation}\label{equ:ELNA}
F_\NA(\f)-F_\NA(\p)=\int_0^1ds \int(\f-\p)\Ga_\NA(s\f+(1-s)\p)
\end{equation}
for all $\f,\p\in\cE^{1,S}_\NA$ (see Lemma~\ref{lem:EL}). By semi-global H\"older continuity, it is enough to show this when $\f,\p\in\cH^S_\NA$. Pick rays $\{\phi_t\}$, $\{\p_t\}$ in $\cH^S$ with analytic singularities and \na limits $\f,\p$. Since $F\colon\cE^{1,S}\to\R$ is an {\EL} functional for $\Ga$, we have for each $t\ge 0$ 
$$
F(\phi_t)=\int_0^1ds\int(\phi_t-\phi_\refe)\Ga(s\phi_t+(1-s)\phi_\refe),\quad F(\p_t)=\int_0^1ds\int(\p_t-\phi_\refe)\Ga(s\p_t+(1-s)\phi_\refe),
$$
$$
F(\phi_t)-F(\p_t)=\int_0^1 ds\int(\phi_t-\p_t)\Ga(s\phi_t+(1-s)\p_t). 
$$
By~\eqref{equ:sgh1} and~\eqref{equ:Iconv},  
$$
t^{-1}\int(\phi_t-\phi_\refe)\Ga(s\phi_t+(1-s)\phi_\refe),\quad t^{-1}\int(\p_t-\phi_\refe)\Ga(s\p_t+(1-s)\phi_\refe),
$$
$$
t^{-1}\int(\phi_t-\p_t)\Ga(s\phi_t+(1-s)\p_t)
$$
are uniformly bounded, and they respectively converge for each $s\in [0,1]$ to 
$$
\int\f\,\Ga_\NA(s\f),\quad\int\p\,\Ga_\NA(s\p),\quad\int(\f-\p)\Ga_\NA(s\f+(1-s)\p)
$$
as $t\to\infty$, by~\eqref{equ:slopeint}. By dominated convergence, we thus get
$$
t^{-1} F(\phi_t)\to F_\NA(\f),\quad t^{-1} F(\p_t)\to F_\NA(\p), 
$$
and
$$
t^{-1}\int_0^1 ds\int(\phi_t-\p_t)\Ga(s\phi_t+(1-s)\p_t)\to\int_0^1 ds\int (\f-\p)\Ga_\NA(s\f+(1-s)\p)
$$
as $t\to\infty$. This proves~\eqref{equ:ELNA}, and hence (i). Now assume $\{\phi_t\}$ is an almost geodesic ray in $\cE^{1,S}$ whose direction $\phi_\infty$ lies in $\cE^{1,S}_\NA$. Then 
$$
F(\phi_t)=\int_0^1 ds\int(\phi_t-\phi_\refe)\Ga(s\phi_t+(1-s)\phi_\refe),\quad F_\NA(\phi_\infty)=\int_0^1 ds\int\phi_\infty\,\Ga_\NA(s\phi_\infty).
$$
For any $s\in [0,1]$, \eqref{equ:slopeint} yields 
$$
t^{-1}\int(\phi_t-\phi_\refe)\Ga(s\phi_t+(1-s)\phi_\refe)\to\int\phi_\infty\,\Ga_\NA(s\phi_\infty),
$$
and (ii) follows by dominated convergence. 

Finally, assume $F$ is concave on $\cE^{1,S}$. Then $\int(\phi-\p)(\Ga(\p)-\Ga(\phi))\ge 0$ for all $\phi,\p\in\cH^S$ (see Lemma~\ref{lem:Fconc}). Passing to the slope at infinity along rays with analytic singularities, this yields $\int(\f-\p)\left(\Ga_\NA(\p)-\Ga_\NA(\f)\right)\ge 0$ for all $\f,\p$ in $\cH^S_\NA$, hence also in $\cE^{1,S}_\NA$. Using again Lemma~\ref{lem:Fconc}, we conclude $F_\NA$ is concave. 
\end{proof}

%
%
\section{Weighted pluripotential theory}\label{sec:wppt}

In this section, we first review the weighted pluripotential formalism in the complex setting following~\cite{BWN,Lah19}, and then introduce \na versions of the weighted Monge--Amp\`ere operators and energy functionals, extending~\cite{HLi23}. 

In what follows, $T_\C$ denotes a complex algebraic torus with maximal compact torus $T$ and dual lattices $M_\Z=M_\Z(T)$, $N_\Z=N_\Z(T)$ (see~\S\ref{sec:lattices}). 

%
%
\subsection{Duistermaat--Heckman measures and weighted degrees}\label{sec:DH}

Consider for the moment an arbitrary complex projective scheme $Y$ of pure dimension $n$, equipped with a $T_\C$-action. Denote by $\{Y_\a\}_\a$ its set of irreducible components, and by $m_\a$ the multiplicity of $Y$ along $Y_\a$. 

To each (not necessarily ample) $T_\C$-equivariant $\Q$-line bundle $F$ on $Y$ is associated its \emph{Duistermaat--Heckman measure} $\DH_F$, a signed measure with compact support on $M_\R$. This measure is uniquely characterized in terms of equivariant intersection numbers by
\begin{equation}\label{equ:DHmom}
\int_{M_\R} \frac{\xi^d}{d!}\,\DH_F=n!\frac{(F^{n+d})_{T}}{(n+d)!}(\xi)
\end{equation}
for all $\xi\in N_\R$ and $d\in\N$. Here
$$
(F^{n+d})_{T}:=\pi_\star\left(c_1^{T}(F)^{n+d}\cap[Y]_{T}\right)\in\CH_{-d}^{T}(\Spec\C)_\Q\simeq S^d M_\Q
$$
is viewed as a rational homogeneous polynomial of degree $d$ on $N_\R$, where $\pi\colon Y\to\Spec\C$ denotes the structure morphism, $c_1^{T}(F)\in\CH^1_{T}(Y)$ the equivariant first Chern class, and $[Y]_{T}=\sum_\a m_\a[Y_\a]_{T}\in\CH_n^{T}(Y)$ the equivariant fundamental class (see~\cite{EG}). Note that 
\begin{equation}\label{equ:DHdec}
\DH_F=\sum_\a m_\a\DH_{F|_{Y_\a}}. 
\end{equation}

\begin{rmk} When $\xi$ is rational, by construction of equivariant Chow (co)homology~\cite[\S 2.2]{EG}, the right-hand side of~\eqref{equ:DHmom} can be realized as a usual top self-intersection number on the total space of a fiber bundle with typical fiber $Y$, compare~\cite[\S 3.1]{BHJ1} and~\cite[\S 2.1]{HLi23}. 
\end{rmk}

The Duistermaat--Heckman measure also admits the following differential geometric description. Each smooth $T$-invariant metric $\phi\in C^\infty(F)^T$ induces a moment map $m_\phi\colon Y\to M_\R$ for the closed $T$-invariant $(1,1)$-form $\ddc\phi$ (the curvature form), 
given by 
\begin{equation}\label{equ:moment}
-m_\phi^\xi:=\langle m_\phi,\xi\rangle=\cL_{J\xi}\phi:=\frac{d}{dt}\bigg|_{t=0} \exp(Jt\xi)^\star\phi
\end{equation}
for $\xi\in N_\R$, which thus satisfies the moment map equation
\begin{equation}\label{equ:mom}
-dm_\phi^\xi=\ddc\phi(\xi,\cdot), 
\end{equation}
see for instance~\cite[\S 3.1]{BJT}. Then 
\begin{equation}\label{equ:DH2}
\DH_F=(m_\phi)_\star\left[(\ddc\phi)^{n}\wedge\d_Y\right]
\end{equation}
with $\d_Y=\sum_\a m_\a \d_{Y_\a}$ the integration current on $Y$. Indeed, this formula is well-known when $Y$ is smooth, and the general case follows by pulling back $F$ to an equivariant resolution of singularities of $Y$. 

\smallskip

When $F$ is further ample, the equivariant Riemann--Roch theorem implies that $\DH_F$ is a positive measure which describes the distribution of the weights of the $T$-action on $\Hnot(Y,mF)$ as $m\to\infty$ (see for instance~\cite[Appendix~B]{BHJ1}). Furthermore, choosing $\phi\in\cH^T(F)$ in~\eqref{equ:DH2} shows that $\DH_F$ has support equal to 
$$
m_\phi(Y)=:P_F=\bigcup_\a P_{F|_{Y_\a}}\subset M_\R. 
$$
Here $P_{F|_{Y_\a}}=m_\phi(Y_\a)$ is, for each irreducible component $Y_\a$, a rational (convex) polytope (independent of the choice of $\phi$) by the Atiyah--Guillemin--Sternberg theorem, and is called the \emph{moment polytope} of $(Y_\a,F)$. 

Returning to the case of an arbitrary equivariant $\Q$-line bundle $F$, we introduce:
\begin{defi} Given any `weight' function $v\in C^\infty(M_\R)$, we define the \emph{$v$-weighted degree} of $F$ as 
\begin{equation}\label{equ:wdeg}
\deg_v(F):=\int_{M_\R} v\,\DH_F\in\R. 
\end{equation}
\end{defi}
By~\eqref{equ:DHmom} for $v=1$ we recover the usual degree
$$
\deg_1(F)=\deg(F):=(F^{n}). 
$$
More generally, for a monomial $v=\xi^d/d!$ \eqref{equ:DHmom} yields
\begin{equation}\label{equ:degpol}
\deg_{\xi^d/d!}(F)=n!\frac{(F^{n+d})_{T}}{(n+d)!}(\xi). 
\end{equation}
By~\eqref{equ:DHdec} and~\eqref{equ:DH2}, we have
\begin{equation}\label{equ:degmet}
\deg_v(F)=\sum_\a m_\a \deg_v(F|_{Y_\a})\quad\text{ where}\quad\deg_v(F|_{Y_\a})=\int_{Y_\a} v(m_\phi)(\ddc\phi)^{n}. 
\end{equation}
When $F$ is further ample, $v\mapsto\deg_v(F)$ is monotone increasing; in particular, 
\begin{equation}\label{equ:wvolbd}
(F^{n})\inf_{P_F} v\le\deg_v(F)\le(F^{n})\sup_{P_F} v. 
\end{equation}

\begin{lem}\label{lem:diffdeg} For all $F,G\in\Pic(Y)^{T}_\Q$, the directional derivative
\begin{equation}\label{equ:twdeg}
\deg_v'(F;G):=\frac{d}{ds}\bigg|_{s=0}\deg_v(F+sG)\in\R
\end{equation}
is well-defined, and linear with respect to $G$. 
\end{lem}

\begin{proof} Pick $\phi\in C^\infty(F)^T$, $\tau\in C^\infty(G)^T$. By~\eqref{equ:degmet} we have 
$$
\deg_v(F+sG)=\int_Y v(m_{\phi}+s m_{\tau})(\ddc(\phi+s\tau))^{n}
$$
and the result follows, together with the differential geometric expression
\begin{equation}\label{equ:degder}
\deg_v'(F;G)=\int_Y\left[ n v(m_{\phi})\ddc\tau+\langle v'(m_{\phi}),m_{\tau}\rangle \ddc\phi\right]\wedge(\ddc\phi)^{n-1}. 
\end{equation}
\end{proof}

\begin{rmk}\label{rmk:dconv} Here, and on a few other occasions in this paper, $\phi+s\tau$ really makes sense as a metric on a $\Q$-line bundle only when $s$ is rational. However, this is a purely notational obstacle that is easily circumvented by introducing instead potentials with respect to reference metrics, or by arguing in the tensor product with $\R$ of the appropriate Picard group of isomorphism classes of metrized (equivariant) line bundles.
\end{rmk}

\begin{rmk} While $F\mapsto\deg_v(F)=\int_{M_\R} v\,\DH_F$ is differentiable, the above expression~\eqref{equ:degder} for its derivative involves $v'$, and $F\mapsto\DH_F$ itself is thus not differentiable. When $v$ is a monomial $\xi^{d}/d!$ with $\xi\in N_\R$, Lemma~\ref{lem:diffdeg} also follows from~\eqref{equ:DHmom}, which further yields
$$
\deg_{\xi^d/d!}'(F;G)=n!\frac{(F^{n+d-1}\cdot G)_{T}}{(n+d-1)!}(\xi)
$$
In the general case, a purely algebro-geometric proof of Lemma~\ref{lem:diffdeg} can presumably be provided by approximating $v$ in $C^1$-topology by polynomials, in the spirit of~\cite[\S 5.3]{HLi23}. 
\end{rmk}

As a consequence of \eqref{equ:DHmom} and the invariance of equivariant intersection numbers under flat deformation, we note for later use: 

\begin{lem}\label{lem:degDH} Pick a test configuration $(\cX,\cL)$ for $(X,L)$, with central fiber $\cX_0=\sum_E b_E E$. Then:  

\begin{equation}\label{equ:DHdectc}
\DH_L=\DH_{\cL|_{\cX_0}}=\sum_E b_E\DH_{\cL|_E}, 
\end{equation}
\begin{equation}\label{equ:degdec}
\deg_v(L)=\sum_E b_E \deg_v(\cL|_E),\quad\deg_v'(L;G)=\sum_E b_E\deg_v'(\cL|_E;G_\cX|_E)
\end{equation}
for any equivariant $\Q$-line bundle $G$ on $X$, $v\in C^\infty(M_\R)$. If $\cL$ is further ample, then 
\begin{equation}\label{equ:poldec}
P_L=\bigcup_E P_{\cL|_E}. 
\end{equation}
\end{lem}
Here $G_\cX\in\Pic(\cX)_\Q^T$ denotes the pull-back of $G$ to $\cX$, which is implicitly assumed to dominate $X\times\P^1$. 

%
\subsection{Weighted complex pluripotential theory}\label{sec:weighted}
From now on we return to the general framework of thus papier, \ie $(X,L)$ is a smooth projective variety equipped with an ample $\Q$-line bundle. We assume given a compact torus $T\subset\Aut(X,L)$, with dual lattices $M_\Z=M_\Z(T)$, $N_\Z=N_\Z(T)$, see~\S\ref{sec:lattices}, and denote by 
$$
P=P_L\subset M_\R=N_\R^\vee\simeq(\Lie T)^\vee
$$
the moment polytope. 

We use~\cite[\S 3]{BJT} as a reference for what follows. Fix a weight function $v\in C^\infty(M_\R)$. Following~\cite{BWN,Lah19}, we associate to any smooth $T$-invariant metric $\phi\in C^\infty(L)^T$ its \emph{$v$-weighted \ma measure} 
$$
\MA_v(\phi):=v(m_\phi)(\ddc\phi)^{n}, 
$$
a $T$-invariant signed measure on $X$ of total mass
\begin{equation}\label{equ:totalwMA}
\int_X\MA_v(\phi)=\int_P v\,\DH_L=\deg_v(L), 
\end{equation}
see~\eqref{equ:degmet}. The (unweighted, normalized) \ma measure is thus
$$
\MA(\phi)=V^{-1}\MA_1(\phi)=\MA_{V^{-1}}(\phi). 
$$
The weighted \ma measure of any $\phi\in\cH^T$ satisfies 
\begin{equation}\label{equ:MAvMA}
V(\inf_P v)\MA(\phi)\le\MA_v(\phi)\le V(\sup_P v)\MA(\phi).
\end{equation}
In particular, $\MA_v(\phi)\ge 0$ iff $v\ge 0$ on $P$. 

Given any equivariant $\Q$-line bundle $F$ and $\tau\in C^\infty(F)^T$, the \emph{$\tau$-twisted, $v$-weighted \ma measure} of $\phi\in C^\infty(L)^T$ is defined as the directional derivative\footnote{See Remark~\ref{rmk:dconv}.}
\begin{equation}\label{equ:twder}
\MA_v^\tau(\phi):=\frac{d}{ds}\bigg|_{s=0}\MA_v(\phi+s\tau). 
\end{equation}
Explicitly, 
\begin{equation}\label{equ:twMA} 
\MA_v^\tau(\phi)=n v(m_\phi)\,\ddc\tau\wedge (\ddc\phi)^{n-1}+\langle v'(m_\phi),m_\tau\rangle(\ddc\phi)^n. 
\end{equation}
When $\phi\in\cH^T$ and $v>0$ on $P$, this can also be written 
\begin{equation}\label{equ:twMAMA}
\MA_v^\tau(\phi)=\left(\tr_{\ddc\phi}(\ddc\tau)+\langle(\log v)'(m_\phi),m_\tau\rangle\right)\MA_v(\phi).
\end{equation}
By~\eqref{equ:degder} this defines a signed measure on $X$ of total mass
\begin{equation}\label{equ:totaltwMA}
\int_X\MA_v^\tau(\phi)=\deg_v'(L;F). 
\end{equation}

\begin{rmk} We emphasize that $\MA_v^\tau(\phi)$ depends on the $T$-action on $F$, which determines the moment map $m_\tau$. In the notation of~\cite{BJT} we have $\MA_v^\tau(\phi)=\MA_v^{\ddcT\tau}(\phi)$ with $\ddcT\tau=(\ddc\tau,m_\tau)$ the equivariant curvature form of $\tau$. 
\end{rmk}

When $F$ is trivial, \cite[(3.14),(3.15)]{BJT} yields 
\begin{equation}\label{equ:twquad}
\int_X f\MA_v^g(\phi)=-n\int_X v(m_\phi)(\ddc\phi)^{n-1}\wedge df\wedge d^c g
\end{equation}
for any $f\in C^\infty(X)$ and $g\in C^\infty(X)^T$. This shows that the weighted \ma operator $\MA_v\colon C^\infty(L)^T\to \Cz(X)^\vee$
satisfies the symmetry property~\eqref{equ:sym}, and hence admits an {\EL} functional
$$
\en_v\colon C^\infty(L)^T\to\R,
$$
called the \emph{$v$-weighted \ma energy}. 

Gor any $\tau\in C^\infty(F)^T$, the operator $\MA_v^\tau\colon C^\infty(L)^T\to \Cz(X)^\vee$ also admits an {\EL} functional
$$
\en_v^\tau\colon C^\infty(L)^T\to\R,
$$
the \emph{$\tau$-twisted, $v$-weighted \ma energy}, given by 
$$
\en_v^\tau(\phi)=\frac{d}{ds}\bigg|_{s=0}\en_v(\phi+s\tau)
$$
for $\phi\in C^\infty(L)^T$. Extending~\cite[Proposition~3.40]{BJT} we next show: 

\begin{prop}\label{prop:MAsgh} Assume that $v\in C^\infty(M_\R)$ is positive on $P$, and normalized by $\deg_v(L)=1$, so that $\MA_v$ maps $\cH^T$ into $\cM^{1,T}$. Then: 
\begin{itemize}
\item[(i)] $\MA_v\colon(\cH^T/\R,\ii)\to(\cM^{1,T},\done)$ is {\sgh}; 
\item[(ii)] if $\tau\in C^\infty(F)^T$ is psh, then there exists $a,b>0$ such that $a\MA_v^\tau+b\MA_v$ induces a {\sgh} map $(\cH^T/\R,\ii)\to(\cM^{1,T},\done)$. 
\end{itemize}
\end{prop}

\begin{proof} In what follows $C>0$ denotes a constant only depending on $n$, $v$ and $\tau$, that is allowed to vary from line to line. Pick $\phi_1,\phi_2\in\cH$ and $\phi_3,\phi_4\in\cH^T$, and set $\phi_s:=s\phi_3+(1-s)\phi_4$ for $s\in [0,1]$. As a consequence of~\eqref{equ:twder} and~\eqref{equ:twquad}, we have 
\begin{equation}\label{equ:MAvint}
\int_X(\phi_1-\phi_2)\left(\MA_v(\phi_3)-\MA_v(\phi_4)\right)=-n\int_0^1 ds\int_X v(m_{\phi_s}) d(\phi_1-\phi_2)\wedge d^c (\phi_3-\phi_4)\wedge(\ddc\phi_s)^{n-1}. 
\end{equation}
By the Cauchy--Schwarz inequality, this implies
$$
\left|\int(\phi_1-\phi_2)\left(\MA_v(\phi_3)-\MA_v(\phi_4)\right)\right|\le C\sup_{s\in [0,1]}\left(\int d(\phi_1-\phi_2)\wedge d^c(\phi_1-\phi_2)\wedge(\ddc\phi_s)^{n-1}\right)^{1/2}
$$
$$
\times \left(\int d(\phi_3-\phi_4)\wedge d^c(\phi_3-\phi_4)\wedge(\ddc\phi_s)^{n-1}\right)^{1/2}\le C\ii(\phi_3,\phi_4)^{\a_n}\max_i\ii(\phi_i)^{1-\a_n},
$$
where the last inequality follows from~\eqref{equ:quaden1} and~\eqref{equ:Iconv}. This proves that $\MA_v$ satisfies~\eqref{equ:sgh1}, and (i) follows.  

Next set
$$
B:=\sup_{\a\in P,\,\b\in m_\tau(X)}\left(|v(\a)|+|\langle v'(\a),\b\rangle|\right)
$$
and
$$
\Ga(\phi):=\MA_v^\tau(\phi)+B\MA(\phi)=n v(m_\phi)\ddc\tau\wedge(\ddc\phi)^{n-1}+\left(\langle v'(m_\phi),m_\tau\rangle+B\right)(\ddc\phi)^n\ge 0
$$
for each $\phi\in\cH^T$ (recall $\ddc\tau\ge 0$). In view of~\eqref{equ:twder}, \eqref{equ:twquad} can be rewritten as
\begin{equation}\label{equ:derMA}
\frac{d}{ds}\bigg|_{s=0}\int f\,\MA_v(\phi+sg)=-n\int v(m_\phi)(\ddc\phi)^{n-1}\wedge df\wedge d^c g
\end{equation}
for $f\in C^\infty(X)$ and $g\in C^\infty(X)^T$.  Differentiating with respect to $\phi$ in the direction of $\tau$ and using again~\eqref{equ:twder}, we infer
$$
\frac{d}{ds}\bigg|_{s=0}\int f\,\MA_v^\tau(\phi+s g)=-n\int \left[(n-1) v(m_\phi)\ddc\tau\wedge(\ddc\phi)^{n-2}+\langle v'(m_\phi),m_\tau\rangle(\ddc\phi)^{n-1}\right]\wedge df\wedge d^c g. 
$$
Combined with~\eqref{equ:MAvint} (for $v=1$), this yields 
$$
\int(\phi_1-\phi_2)\left(\Ga(\phi_3)-\Ga(\phi_4)\right)=-n\int_0^1 ds\int d(\phi_1-\phi_2)\wedge d^c(\phi_3-\phi_4)\wedge(\ddc\phi_s)^{n-2}\wedge\eta_s
$$
where
$$
\eta_s:=(n-1)v(m_{\phi_s}) \ddc\tau+(\langle v'(m_{\phi_s}),m_\tau\rangle+B)\ddc\phi_s, 
$$
which satisfies 
$$
0\le\eta_s\le nB\left(\ddc\tau+\ddc\phi_s\right).
$$
By Cauchy--Schwarz we infer
$$
\left|\int(\phi_1-\phi_2)\left(\Ga(\phi_3)-\Ga(\phi_4)\right)\right|\le n\sup_{s\in [0,1]}\left(\int d(\phi_1-\phi_2)\wedge d^c(\phi_1-\phi_2)\wedge(\ddc\phi_s)^{n-2}\wedge\eta_s\right)^{1/2}
$$
$$
\times\left(\int d(\phi_3-\phi_4)\wedge d^c(\phi_3-\phi_4)\wedge(\ddc\phi_s)^{n-2}\wedge\eta_s\right)^{1/2}. 
$$
Here 
$$
\int d(\phi_1-\phi_2)\wedge d^c(\phi_1-\phi_2)\wedge(\ddc\phi_s)^{n-2}\wedge\eta_s\le nB\int d(\phi_1-\phi_2)\wedge d^c(\phi_1-\phi_2)\wedge(\ddc\phi_s)^{n-2}\wedge\ddc\tau
$$
$$
+nB\int d(\phi_1-\phi_2)\wedge d^c(\phi_1-\phi_2)\wedge(\ddc\phi_s)^{n-1}\le C\max_i\ii(\phi_i)
$$
by~\eqref{equ:quaden2}, \eqref{equ:quaden1} and~\eqref{equ:Iconv}. Similarly, 
$$
\int d(\phi_3-\phi_4)\wedge d^c(\phi_3-\phi_4)\wedge(\ddc\phi_s)^{n-2}\wedge\eta_s\le C\ii(\phi_3,\phi_4)^{\a_n}\max_i\ii(\phi_i)^{1-\a_n}, 
$$
and we conclude 
$$
\left|\int(\phi_1-\phi_2)\left(\Ga(\phi_3)-\Ga(\phi_4)\right)\right|\le C\ii(\phi_3,\phi_3)^{\a_n}\max_i\ii(\phi_i)^{1-\a_n}.
$$
Since, for any $\phi\in\cH^T$, $\Ga(\phi)=\MA_v^\tau(\phi)+B\MA_v(\phi)$ has total mass $\deg_v'(L;F)+B\deg_v(L)$, setting 
$$
a:=(\deg_v'(L;F)+B \deg_v(L))^{-1},\quad b:=aB
$$
ensures that $a\MA_v^\tau+b\MA_v=a\Ga$ takes values in $\cM^1$ and satisfies~\eqref{equ:sgh1}, which concludes the proof of (ii).  
\end{proof}

\begin{prop}\label{prop:MAext} For any $v\in C^\infty(M_\R)$ and $\tau\in C^\infty(F)^T$, the weighted \ma operators $\MA_v,\,\MA_v^\tau\colon\cH^T\to \Cz(X)^\vee$ and their {\EL} functionals $\en_v,\,\en_v^\tau\colon\cH^T\to\R$ admit unique continuous\footnote{For the weak topology of $\Cz(X)^\vee$.} extensions 
$$
\MA_v,\,\MA_v^\tau\colon\cET\to \Cz(X)^\vee,\quad \en_v,\,\en_v^\tau\colon\cET\to\R,
$$
the latter still being {\EL} functionals for the former, that are further {\sgh}. Moreover, for any $\phi\in\cET$, the signed measures $\MA_v(\phi)$, $\MA_v^\tau(\phi)$ have finite energy (see Definition~\ref{defi:finen}). 
\end{prop}
\begin{proof} By linearity, we may assume $v>0$ on $P$ and $\tau$ psh. By Proposition~\ref{prop:MAsgh}, $\MA_v$ is {\sgh}, as well as $a\MA^\tau_v+b\MA_v$ for some $a,b>0$. Since $\cH^T$ is dense in $\cE^{1,T}$, the result is thus proved just as in Lemma~\ref{lem:NAtrans2}. 
\end{proof}

We next show: 
\begin{prop}\label{prop:raden} For any $v\in C^\infty(M_\R)$ and $\tau\in C^\infty(M)^T$, the functionals $\en_v,\,\en_v^\tau\colon\cET\to\R$ admit {\sgh} radial transforms $\en_{v,\ra},\,\en_v^{\tau,\ra}\colon\cE^{1,T}_\ra\to\R$. For any almost geodesic ray $\{\phi_t\}$ in $\cET$ we further have 
\begin{equation}\label{equ:slopen}
\en_{v,\ra}(\phi_\infty)=\lim_{t\to\infty} t^{-1}\en_v(\phi_t),\quad\en_v^{\tau,\ra}(\phi_\infty)=\lim_{t\to\infty} t^{-1}\en_v^\tau(\phi_t). 
\end{equation}
\end{prop}

\begin{lem}\label{lem:enconc} For any $v\in C^\infty(M_\R)$ such that $v\ge 0$ on $P$ and any smooth psh metric $\tau\in C^\infty(F)^T$, the weighted energy functionals $\en_v,\,\en_v^\tau\colon\cET\to\R$ satisfy: 
\begin{itemize}
\item[(i)] $\en_v$ is monotone increasing, and concave on affine segments;
\item[(ii)] $\en_v$ is affine linear on psh geodesics;
\item[(iii)] $\en_v^\tau$ is convex on psh geodesics. 
\end{itemize}
\end{lem}
\begin{proof} See~\cite[Proposition~2.17]{BWN} for (i), (ii), and~\cite[Lemma~7]{Lah23} for (iii). 
\end{proof}
Recall from Lemma~\ref{lem:Fconc} that concavity in (i) is equivalent to 
\begin{equation}\label{equ:enconc}
\int(\phi-\p)\MA_v(\phi)\le\en_v(\phi)-\en_v(\p)\le\int(\phi-\p)\MA_v(\p)
\end{equation}
for all $\phi,\p\in\cET$, and also
\begin{equation}\label{equ:Iv}
\ii_v(\phi,\p):=\int(\phi-\p)\left(\MA_v(\p)-\MA_v(\phi)\right)\ge 0, 
\end{equation}
\begin{equation}\label{equ:Jv}
\jj_v(\phi,\p):=\int(\phi-\p)\MA_v(\p)-\en_v(\phi)+\en_v(\p)\ge 0. 
\end{equation}

\begin{proof}[Proof of Proposition~\ref{prop:raden}] As above, we may assume $v>0$ on $P$ and $\tau$ psh, by linearity. Lemma~\ref{lem:enconc} then guarantees the existence of radial transforms, and the rest follows from Proposition~\ref{prop:MAext} and Proposition~\ref{prop:sgh}.
\end{proof} 
 
\subsection{The weighted Calabi--Yau theorem}

We assume in what follows that $v$ is positive on $P$, and normalized so that 
$$
\deg_v(L)=\int_P v\,\DH_L=1,
$$
so that the weighted Monge--Amp\`ere operator $\MA_v\colon\cET\to C^0(X)^\vee$ takes values in the space $\cM^{1,T}$ of $T$-invariant measures of finite energy. 

\begin{thm}\label{thm:wMA} The weighted \ma operator induces a homeomorphism 
$$
\MA_v\colon\cET/\R\simto\cM^{1,T}
$$
with respect to the strong topology on both sides. 
\end{thm}

While the proof is basically contained in~\cite{BWN,HLi23}, we provide some details that will directly carry over to \na case in the next section.

\begin{lem}\label{lem:I} For all $\phi,\p\in\cET$ we have 
\begin{equation}\label{equ:IvI}
(\inf_P v)V \ii(\phi,\p)\le\ii_v(\phi,\p)\le (\sup_P v)V \ii(\phi,\p)
\end{equation}

\end{lem}
\begin{proof} By continuity it is enough to show this for $\phi,\p\in\cH^T$. By~\eqref{equ:MAvint} we then have
$$
\ii_v(\phi,\p)=n\int_0^1 ds\int_X v(m_{\phi_s}) d(\phi-\p)\wedge d^c (\phi-\p)\wedge(\ddc\phi_s)^{n-1}
$$
with $\phi_s:=s\phi+(1-s)\p$. The result follows since $d(\phi-\p)\wedge d^c (\phi-\p)\wedge(\ddc\phi_s)^{n-1}\ge 0$ for all $s$.
\end{proof}

\begin{lem}\label{lem:topemb} Pick $\phi_1,\phi_2\in\cET$ and set $\mu_i:=\MA_v(\phi_i)$. Then there exists $C>0$ only depending on $n$, $v$ and $\phi_\refe$ such that 
$$
\ii(\phi_1,\phi_2)\le C\done(\mu_1,\mu_2)\max_i\left(1+\done(\mu_i)\right). 
$$
In particular, $\MA_v\colon\cET/\R\hto\cM^{1,T}$ is a topological embedding. 
\end{lem}
\begin{proof} Set $\mu_{v,\refe}:=\MA_v(\phi_\refe)$. By~\eqref{equ:IvI} and~\eqref{equ:estMA} we have 
$$
\ii(\phi_1,\phi_2)\le C\ii_v(\phi_1,\phi_2)=C\int(\phi_1-\phi_2)(\mu_2-\mu_1)\le C\done(\mu_1,\mu_2)\max_i\left(1+\ii(\phi_i)\right)^{1/2}. 
$$
Similarly, 
$$
\ii(\phi_i)\le C\ii_v(\phi_i)=C\int(\phi_i-\phi_\refe)(\mu_{v,\refe}-\mu_i)\le C\done(\mu_i,\mu_{v,\refe})(1+\ii(\phi_i))^{1/2}. 
$$
This implies $\ii(\phi_i)\le C(1+\done(\mu_i,\mu_{v,\refe}))^2$, and hence $\ii(\phi_i)\le C(1+\done(\mu_i))^2$, by the triangle inequality. The desired estimate follows, and implies the final point since the topologies of $\cE^1/\R$ and $\cM^1$ are respectively defined by the quasi-metrics $\ii$ and $\done$ (see~\S\ref{sec:comppt}). 
\end{proof}

\begin{lem}\label{lem:ortho} Pick $\phi\in\cET$, $f\in C^0(X)^T$, and consider the singular usc metric $\p:=\phi+f$ on $L$ and its psh envelope $\env(\p)$, see~\eqref{equ:pshenv}. Then $\env(\p)\in\cET$, and we have the \emph{orthogonality relation}
\begin{equation}\label{equ:wortho}
\int_X\left(\p-\env(\p)\right)\MA_v(\env(\p))=0. 
\end{equation}
\end{lem}
\begin{proof} The first point is clear from the definition of $\env(\p)$. Since $0\le\p-\env(\p)$ and $0\le\MA_v(\env(\p))\le C\MA(\env(\p))$ for some $C>0$ (see~\eqref{equ:MAcomp}), we have 
$$
0\le\int\left(\p-\env(\p)\right)\MA_v(\env(\p))\le C\int\left(\p-\env(\p)\right)\MA(\env(\p)),
$$
where the right-hand integral vanishes by the usual (unweighted) orthogonality relation~\eqref{equ:ortho}. 
\end{proof}

\begin{lem}\label{lem:diffen} For any $\phi\in\cET$ and $f\in C^0(X)^T$ we have 
$$
\frac{d}{ds}\bigg|_{s=0}\en_v(\env(\phi+s f))=\int_X f\,\MA_v(\phi).
$$
\end{lem}
\begin{proof} Pick $s>0$.  By~\eqref{equ:envlip}, we have 
\begin{equation}\label{equ:envest}
\env(\phi+s f)=\phi+O(s),
\end{equation}
and hence $\env(\phi+s f)\in\cET$. 
By concavity of $\en_v$ we further have 
$$
\int\left(\env(\phi+s f)-\phi\right)\MA_v(\env(\phi+sf)\le\en_v(\env(\phi+s f))-\en_v(\phi)\le\int\left(\env(\phi+s f)-\phi\right)\MA_v(\phi),
$$
see~\eqref{equ:enconc}. Since $\env(\phi+s f)\le\phi+s f$, the right-hand side satisfies 
$$
\int\left(\env(\phi+s f)-\phi\right)\MA_v(\phi)\le s\int f\,\MA_v(\phi),
$$
while~\eqref{equ:wortho} yields 
$$
\int\left(\env(\phi+s f)-\phi\right)\MA_v(\env(\phi+sf)=s\int f\MA_v(\env(\phi+sf). 
$$
Thus 
$$
\int f\,\MA_v(\env(\phi+s f))\le\frac{\en_v(\env(\phi+s f))-\en_v(\phi)}{s}\le\int f\,\MA_v(\phi).
$$
Now \eqref{equ:envest} implies $\dd_1\left(\env(\phi+s f),\phi\right)\to 0$ as $s\to 0$, and hence $\MA_v(\env(\phi+s f))\to\MA(\phi)$ weakly. We infer
$$
\lim_{s\to 0}\frac{\en_v(\env(\phi+s f))-\en_v(\phi)}{s}=\int f\,\MA_v(\phi),
$$
and the result follows. 
\end{proof}

\begin{lem}\label{lem:monusc} Extend $\en_v$ to $\PSH^T(L)$ by setting $\en_v\equiv-\infty$ outside $\cET$. Then 
$$
\en_v\colon\PSH^T(L)\to\R\cup\{-\infty\}
$$
is monotone increasing, and usc in the weak topology.
\end{lem}
\begin{proof} Since $\en_v$ is monotone increasing and (strongly) continuous on $\cET$, for any $\phi\in\PSH^T(L)$ we have 
$$
\en_v(\phi)=\inf\left\{\en_v(\p)\mid\p\in\cH^T,\,\p\ge\phi\right\}. 
$$
Now pick a weakly convergent sequence $\phi_j\to\phi$ in $\PSH^T(L)$, and $\p\in\cH^T$ such that $\p\ge\phi$. Then $\sup(\phi_j-\p)\to\sup(\phi-\p)\le 0$. For any $\e>0$ we thus have $\phi_j\le\p+\e$ for all $j\gg 1$, and hence $\en_v(\phi_j)\le\en_v(\p)+\e$. This yields, as desired, $\limsup_j\en_v(\phi_j)\le\en_v(\phi)$. 
\end{proof}

\begin{proof}[Proof of Theorem~\ref{thm:wMA}] We pick a reference metric $\phi_\refe\in\cH^T$, and denote by $\cF$ the set of sup normalized metrics $\phi\in\cET$, \ie $\sup(\phi-\phi_\refe)=0$. Recall that $\en,\en_v$ are normalized by 
$$
\en(\phi_\refe)=\en_v(\phi_\refe)=0,
$$
so that 
\begin{equation}\label{equ:encomp}
C^{-1}\en(\phi)\le\en_v(\phi)\le C\en(\phi)\le 0
\end{equation}
for any $\phi\in\cF$ and a uniform constant $C>0$. As a consequence of~\eqref{equ:estMA}, there exists $A,B>0$ such that 
\begin{equation}\label{equ:quad}
\left|\int(\phi-\phi_\refe)\,\mu\right|\le A(-\en(\phi))^{1/2}+B
\end{equation}
for all $\phi\in\cF$. Now pick $\mu\in\cM^{1,T}$, and choose a sequence $(\phi_j)$ in $\cF$ such that $\en_v(\phi_j)-\int(\phi_j-\phi_\refe)\mu$ converges to 
\begin{equation}\label{equ:envee}
\jj_v(\mu):=\sup_{\phi\in\cET}\left(\en_v(\phi)-\int(\phi-\phi_\refe)\,\mu\right),
\end{equation}
where the supremum can equivalently be taken over $\cF$, by translation invariance of $\phi\mapsto\en_v(\phi)-\int(\phi-\phi_\refe)\,\mu$. 
By~\eqref{equ:encomp} and~\eqref{equ:quad}, we have $\jj_v(\mu)<\infty$ and $\en(\phi_j)\ge-C$, so that $(\phi_j)$ stays in a bounded subset of $\cF$. By weak compactness of the space of normalized psh metrics, we may assume $(\phi_j)$ converges weakly to some $\phi\in\PSH(L)$, which necessarily lies in $\cF$ since $(\phi_j)$ is bounded in $\cET$. We claim that $\phi$ achieves~\eqref{equ:envee}. Indeed, Lemma~\ref{lem:holdmu} and Lemma~\ref{lem:monusc} respectively yield $\int(\phi_j-\phi_\refe)\,\mu\to\int(\phi-\phi_\refe)\,\mu$ and $\limsup_j\en_v(\phi_j)\le\en_v(\phi)$, which yields
$$
\en_v(\phi)-\int(\phi-\phi_\refe)\,\mu\ge \limsup_j\left(\en_v(\phi_j)-\int(\phi_j-\phi_\refe)\,\mu\right)=\jj_v(\mu)
$$
and proves the claim. Finally, we claim that $\MA_v(\phi)=\mu$. Since both sides are $T$-invariant, this amounts to $\int f\,\MA_v(\phi)=\int f\,\mu$ for $f\in C^0(X)^T$. Define $g\colon\R\to\R$ by 
$$
g(s):=\en_v(\env(\phi+s f))-\int(\phi+sf-\phi_\refe)\,\mu.
$$
Since $\phi+s f\ge\env(\phi+s f)=:\phi_s\in\cET$, we have 
$$
g(s)\le \en_v(\phi_s)-\int(\phi_s-\phi_\refe)\,\mu\le\jj_v(\mu)=g(0). 
$$
By Lemma~\ref{lem:diffen} this implies, as desired, $\int f\,\MA_v(\phi)-\int f\,\mu=g'(0)=0$. 
\end{proof}

%
\subsection{Non-Archimedean weighted pluripotential theory}
Building on the machinery developed in~\S\ref{sec:NAtrans}, we introduce here \na analogues of the weighted operators and functionals of~\S\ref{sec:weighted} (recovering in particular the \na weighted \ma energy of~\cite{HLi23}). 

We fix an arbitrary weight function $v\in C^\infty(M_\R)$, and a smooth $T$-invariant metric $\tau\in C^\infty(F)^T$ on an equivariant $\Q$-line bundle $F$. 

\begin{prop-def}\label{prop:wNAMA} The operators $\MA_v,\,\MA_v^\tau\colon\cH^T\to \Cz(X)^\vee$ both admit non-Archimedean transforms 
$$
\MA_{v,\NA},\,\MA_{v,\NA}^F\colon\cH^T_{\NA}\to \Cz(X_\NA)^\vee, 
$$
which are further described as follows. Pick $\f\in\cH^T_{\NA}$, and write $\f=\f_\cL$ for an equivariant semiample snc test configuration $(\cX,\cL)$ for $(X,L)$ that dominates $X\times\P^1$, with central fiber $\cX_0=\sum_E b_E E$. Then 
\begin{equation}\label{equ:wNAMA}
\MA_{v,\NA}(\f)=\sum_E b_E \deg_v(\cL|_E)\,\d_{v_{E}}
\end{equation}
and
\begin{equation}\label{equ:twNAMA}
\MA_{v,\NA}^F(\f)=\sum_E b_E\deg_v'(\cL|_E;F_\cX)\,\d_{v_{E}},
\end{equation}
see~\S\ref{sec:DH}. 
\end{prop-def}
Note that $\MA_{v,\NA}(\f)=\MA_{L,v,\NA}(\f)$ and $\MA_{v,\NA}^F(\f)=\MA_{L,v,\NA}^F(\f)$ both depend on $L$, and satisfy\footnote{See Remark~\ref{rmk:dconv}.}
$$
\MA_{L,v,\NA}^F(\f)=\frac{d}{ds}\bigg|_{s=0}\MA_{L+sM,v,\NA}(\f), 
$$
see~\eqref{equ:twdeg}. Like their complex counterparts, these signed measures have total mass 
\begin{equation}\label{equ:totalNAMA}
\int_{X_\NA}\MA_{v,\NA}(\f)=\deg_v(L),\quad\int_{X_\NA}\MA_{v,\NA}^F(\f)=\deg_v'(L;F). 
\end{equation}

For any $\f\in\cH^T_{\NA}$ and $v\in C^\infty(M_\R)$ we further have 
\begin{equation}\label{equ:MAcomp}
V (\inf_Pv)\MA_\NA(\f)\le\MA_{v,\NA}(\f)\le V(\sup_Pv)\MA_\NA(\f). 
\end{equation}

\begin{exam} The function $0\in\cH^T_{\NA}$ satisfies 
$$
\MA_{v,\NA}(0)=\deg_v(L)\,\d_{v_\triv},\quad\MA_{v,\NA}^F(0)=\deg_v'(L;F)\,\d_{v_\triv}
$$
with $v_\triv\in X_\NA$ the trivial valuation.
\end{exam}

\begin{rmk} Taking into account~\eqref{equ:wdeg}, the expression~\eqref{equ:wNAMA} for the \na weighted \ma measure is similar to the formula appearing in the work of E.~Inoue~\cite[Theorem~3.71~(1)]{InoEnt2}. However, the Duisteermaat--Heckman measure of $E$ therein is with respect to the induced $\C^\times$-action on $\cX_0$, not with respect to an ambient torus action as in our case. 
\end{rmk}

\begin{proof}[Proof of Proposition~\ref{prop:wNAMA}] Pick a ray with analytic singularities $\{\phi_t\}$ in $\cH^T$, with \na limit $\phi_\NA\in\cH^T_{\NA}$. Thus $\{\phi_t\}$ is induced by a smooth $(S^1\times T)$-invariant metric $\Phi$ on $\cL$ for an equivariant semiample test configuration $(\cX,\cL)$ such that $\phi_\NA=\f_\cL$. 

We need to show that the rays of signed measures $\{\MA_v(\phi_t)\}_{t\ge 0}$ and $\{\MA_v^\tau(\phi_t)\}_{t\ge 0}$ respectively admit~\eqref{equ:wNAMA}, \eqref{equ:twNAMA} as their \na limits. 

By linearity with respect to $v$, we may assume $v>0$ on $P$, and hence $\mu_t:=\MA_v(\phi_t)\ge 0$. The corresponding $S^1$-invariant family of measures $\{\tmu_z\}_{z\in\DD^\times}$ on $X\times\DD^\times$ is given by 
$$
\tmu_z= v(m_\Phi)(\ddc\Phi)^n\wedge\d_{\cX_z},
$$
where $m_\Phi:\cX\to M_\R$ is defined as in~\eqref{equ:moment} and $\d_{\cX_z}$ denotes the integration current on $\cX_z$. As $z\to 0$ in $\DD$ we thus have 
$$
\tmu_z\to\tmu_0:=v(m_\Phi)(dd\Phi)^n\wedge\d_{\cX_0}=\sum_E b_E v(m_\Phi)(\ddc\Phi)^n\wedge\d_{E}
$$
weakly in $\cX$. By Pille-Schneider's criterion (Proposition~\ref{prop:PS}), it follows, as desired, that $\{\mu_t\}$ admits 
$$
\mu_\NA:=\sum_E b_E \left(\int_{E} v(m_\Phi)(\ddc\Phi)^n\right)\d_{v_{E}}=\sum_E b_E\deg_v(\cL|_E)\d_{v_{E}}
$$
as \na limit. 

As in the proof of Proposition~\ref{prop:sgh}, we may further find $B>0$ such that 
$$
\nu_t:=\MA_v^\tau(\phi_t)+B\mu_t\ge 0
$$
for all $t$. By~\eqref{equ:twMA}, the corresponding $S^1$-invariant family $\{\tnu_z\}_{z\in\DD^\times}$ satisfies
$$
\tnu_z=\left[n v(m_\Phi)\ddc\tau\wedge(\ddc\Phi)^{n-1}+\left(\langle v'(m_\Phi),m_\tau\rangle+B\right)(\ddc\Phi)^n\right]\wedge\d_{\cX_z},
$$
which converges weakly in $\cX$ to 
$$
\tnu_0:=\left[n v(m_\Phi)\ddc\tau\wedge(\ddc\Phi)^{n-1}+\left(\langle v'(m_\Phi),m_\tau\rangle+B\right)(\ddc\Phi)^n\right]\wedge\d_{\cX_0}
$$
$$
=\sum_E b_E \left[n v(m_\Phi)\ddc\tau\wedge(\ddc\Phi)^{n-1}+\left(\langle v'(m_\Phi),m_\tau\rangle+B\right)(\ddc\Phi)^n\right]\wedge\d_{E}.
$$
Using again Proposition~\ref{prop:PS}, it follows that the \na limit of $\{\nu_t\}$ is 
$$
\nu_\NA:=\sum_E b_E\left(\int_{E}n v(m_\Phi)\ddc\tau\wedge(\ddc\Phi)^{n-1}+\left(\langle v'(m_\Phi),m_\tau\rangle+B\right)(\ddc\Phi)^n\right)\d_{v_{E}}
$$
$$
=\sum_E b_E \deg_v'(cL|_E;F_\cX)\d_{v_{E}}+B\mu_\NA,
$$
see~\eqref{equ:degder}. The result follows. 
\end{proof}

Recall from Corollary~\ref{cor:invdense} that $\cH^T_\NA$ is dense in $\cE^{1,T}_\NA$. We may now state: 

\begin{thm}\label{thm:NAMAext} Pick $v\in C^\infty(M_\R)$, an equivariant $\Q$-line bundle $F$, and $\tau\in C^\infty(F)^T$. Then:  
\begin{itemize}
\item[(i)] $\MA_{v,\NA},\,\MA_{v,\NA}^F\colon\cH^T_{\NA}\to \Cz(X_\NA)^\vee$ uniquely extend to continuous\footnote{Recall that $\Cz(X)^\vee$ is equipped with the weak topology.} operators 
\begin{equation}\label{equ:MAext}
\MA_{v,\NA},\,\MA_{v,\NA}^F\colon\cE^{1,T}_{\NA}\to \Cz(X_\NA)^\vee, 
\end{equation}
and $\MA_{v,\NA}(\f)$, $\MA_{v,\NA}^F(\f)$ further have finite energy for each $\f\in\cE^{1,T}_{\NA}$ (see Definition~\ref{defi:finen});\item[(ii)] the operators~\eqref{equ:MAext} admit {\sgh} {\EL} functionals 
$$
\en_{v,\NA},\,\en_{v,\NA}^F\colon\cE^{1,T}_{\NA}\to\R; 
$$
\item[(iii)] for any almost geodesic rays $\{\phi_{1,t}\}$, $\{\phi_{2,t}\}$ in $\cE^1$, $\{\phi_{3,t}\}$ in $\cET$ such that $\phi_{1,\infty},\phi_{2,\infty}\in\cE^1_\NA$, $\phi_{3,\infty}\in\cE^{1,T}_{\NA}$, we have
$$
t^{-1}\int_X(\phi_{1,t}-\phi_{2,t})\MA_v(\phi_{3,t})\to \int_{X_\NA}(\phi_{1,\infty}-\phi_{2,\infty})\MA_{v,\NA}(\phi_{3,\infty}),
$$
$$
t^{-1}\int_X(\phi_{1,t}-\phi_{2,t})\MA_v^\tau(\phi_{3,t})\to \int_{X_\NA}(\phi_{1,\infty}-\phi_{2,\infty})\MA_{v,\NA}^F(\phi_{3,\infty}),
$$
and
$$
t^{-1}\en_v(\phi_{3,t})\to\en_{v,\NA}(\phi_{3,\infty}),\quad t^{-1}\en_v^\tau(\phi_{3,t})\to\en_{v,\NA}^F(\phi_{3,\infty}).
$$
In particular, the rays of signed measures $\{\MA_v(\phi_{3,t})\}$, $\{\MA_v^\tau(\phi_{3,t})\}$ admit $\MA_{v,\NA}(\phi_{3,\infty})$, $\MA_{v,\NA}^F(\phi_{3,\infty})$ as their \na limits, and $\en_{v,\NA}$, $\en_{v,\NA}^F$ coincide with the restrictions to $\cE^{1,T}_{\NA}$ of the radial transforms $\en_{v,\ra},\en_v^{\tau,\ra}\colon\cE^{1,T}_\ra\to\R$ (see Proposition~\ref{prop:raden}). 
\end{itemize}
\end{thm}
Note that~\eqref{equ:MAcomp} remains valid for $\f\in\cE^{1,T}_{\NA}$; in particular $\MA_{v,\NA}(\f)\ge 0$ if $v\ge 0$ on $P$. By Proposition~\ref{prop:MAsgh}, if $v>0$ on $P$ and $F$ is ample, then $\MA_{v,\NA}^F(\f)+B\MA_{v,\NA}(\f)\ge 0$ for a constant $B>0$ independent of $\f$. 

\begin{proof}[Proof of Theorem~\ref{thm:NAMAext}] By linearity, we may again assume $v>0$ on $P$ and $\tau$ psh. Then Proposition~\ref{prop:MAsgh} shows that $\MA_v$ and $\MA_v+B\MA_v^\tau$ are {\sgh} operators $\cH^T\to\cM^1$ for some $B>0$, and the rest follows from Lemmas~\ref{lem:NAtrans1} and~\ref{lem:NAtrans2}.
\end{proof}

\begin{rmk} In~\cite[\S 5.3]{HLi23}, $\en_{v,\NA}\colon\cH^T_{\NA}\to\R$ is defined by approximating $v$ by polynomials and setting 
$$
\en_{\xi^d/d!,\NA}(\f_\cL):=(n+1)!\frac{(\cL^{n+1+d})_T}{(n+1+d)!}(\xi)
$$
for any ample test configuration $(\cX,\cL)$ (compare~\eqref{equ:degpol}). By~\cite[Proposition~5.8]{HLi23}, this definition also computes the \na transform of $\en_v\colon\cH^T\to\R$, and hence concides with our construction of $\en_{v,\NA}$. 
\end{rmk}

Finally we may state the following \na analogue of Theorem~\ref{thm:wMA}. 

\begin{thm}\label{thm:NAwMA} If $v>0$ on $P$ and $\deg_v(L)=1$, then the weighted \na \ma operator induces a homeomorphism 
$$
\MA_{v,\NA}\colon\cE^{1,T}_{\NA}/\R\simto\cM^{1,T}_{\NA}. 
$$
\end{thm}

The proof of Theorem~\ref{thm:wMA} carries over to the \na case without change, noting that the \na analogues of~\eqref{equ:enconc} and~\eqref{equ:IvI} follow by passing to slopes at infinity along geodesic rays, using Theorem~\ref{thm:NAMAext}. 

For later use we also show: 

\begin{prop}\label{prop:orthoPL} Assume $v>0$ on $P$. Pick $\f\in\cE^{1,T}_{\NA}$, and a $T$-equivariant snc test configuration $(\cX,\cL)$, with associated dual complex $\D_\cX\subset X_\NA$ and retraction map $p_\cX\colon X_\NA\to\D_\cX$. Then: 
\begin{itemize}
\item[(i)] $\MA_{v,\NA}(\f)$ is supported in $\D_\cX$ iff $\f=\env(\f\circ p_\cX)$;
\item[(ii)] if $\f=\env(\f_\cL)$ then $\MA_{v,\NA}(\f)$ is supported in the (finite) set of divisorial valuations $v_E\in X_\div$ associated to the irreducible components $E$ of $\cX_0$ that are not contained in the augmented base locus $\B_+(\cL/\P^1)$ relative to the morphism $\cX\to\P^1$. 
\end{itemize}
\end{prop}
\begin{proof} By~\eqref{equ:MAcomp}, we have $\supp\MA_{v,\NA}(\f)=\supp\MA_\NA(\f)$. It therefore suffices to prove (i) and (ii) in the unweighted case, which respectively follow from~\cite[Theorem~8.10]{NAGreen} and Lemma~\ref{lem:suppenv}. 
\end{proof}

\begin{rmk} It is natural to look for purely algebro-geometric proofs of the \na results of this section, and in particular of Theorem~\ref{thm:NAwMA}, working over an arbitrary algebraically closed field (of characteristic $0$) as in~\cite{trivval}. Using~\cite[Theorem~7.4]{HLi23}, one can likely show that $\en_{v,\NA}\colon\cH^T_{\NA}\to\R$ as defined in~\cite[\S 5.3]{HLi23} is an {\EL} functional of $\MA_{v,\NA}\colon\cH^T_\NA\to \Cz(X_\NA)^\vee$, as defined by~\eqref{equ:wNAMA}. The main remaining obstacles are then to show directly (\ie without passing to slopes at infinity from the complex case) that: 
\begin{itemize}
\item[(a)] the operator $\MA_{v,\NA}\colon\cH^T_{\NA}\to\cM^{1,T}_{\NA}$ defined by~\eqref{equ:wNAMA} is {\sgh}; 
\item[(b)] $\en_{v,\NA}\colon\cH^T_{\NA}\to\R$ is concave, and the \na analogue of~\eqref{equ:IvI} holds.
\end{itemize}
As in~\cite{HLi23}, this can presumably be approached by first dealing with the case of a polynomial weight, using Hodge index type inequalities for equivariant intersection numbers, and then treating the general case by polynomial approximation (see~\cite{HLiuklt} for a step in that direction.) 
\end{rmk}

%
%
\section{Entropy asymptotics}\label{sec:entasym}
In this section we keep the setup of~\S\ref{sec:wppt}. We fix a positive weight $v\in C^\infty(P,\R_{>0})$ on the moment polytope $P\subset M_\R$. We also fix a reference metric $\p_{K_X}\in C^\infty(K_X)^T$, and denote by $\mu_\refe=e^{2\p_{K_X}}$ the associated $T$-invariant volume form. 
%
%
%
\subsection{Radial vs.~\na weighted entropy}

\begin{defi} The \emph{weighted entropy functional} $\ent_v\colon\cET\to\R\cup\{+\infty\}$ is defined by setting for $\phi\in\cET$ 
$$
\ent_v(\phi):=\Ent(\MA_v(\phi)|\mu_\refe). 
$$
\end{defi}
Recall that the right-hand side denotes the relative entropy of $\MA_v(\phi)$ and $\mu_\refe$ (see~\S\ref{sec:ent}). For any $a\in\R_{>0}$ and $\phi\in\cET$ we have $\MA_{av}(\phi)=a\MA_v(\phi)$, and hence 
\begin{equation}\label{equ:entscale}
\ent_{av}(\phi)=a\ent_v(\phi)+a\log a. 
\end{equation}

\begin{prop}\label{prop:hent} The weighted entropy $\ent_v\colon\cET\to\R\cup\{+\infty\}$ is strongly lsc (see Definition~\ref{defi:stronglsc}), bounded below, and it admits a radial transform 
$\ent_{v,\ra}\colon\cE^{1,T}_\ra\to[0,+\infty]$. Further: 
\begin{itemize}
\item[(i)] $\ent_{v,\ra}$ is lsc; 
\item[(ii)] for any almost geodesic ray $\{\phi_t\}$ in $\cET$ we have 
$$
\ent_{v,\ra}(\phi_\infty)\le\liminf_{t\to\infty} t^{-1} \ent_v(\phi_t);
$$
\item[(iii)] on $\cE^{1,T}_\ra\setminus\cE^{1,T}_{\NA}$ we have $\ent_{v,\ra}\equiv+\infty$. 
\end{itemize} 
\end{prop}

By~\eqref{equ:entscale}, we may assume after scaling $v$ and $\mu_\refe$ that 
$$
\deg_v(L)=\int_P v\,\DH_L=1,\quad\int_X\mu_\refe=1, 
$$
which implies $\ent_v\ge 0$ on $\cET$, see~\eqref{equ:entlow}. 

\begin{lem}\label{lem:hent} Pick $\phi_\refe\in\cH^T$. For any $C\gg 1$, $\ent_v+C\en_v^{\phi_\refe}$ is then convex on $\cET$ with respect to psh geodesics. 
\end{lem}
\begin{proof} By~\cite[Theorem~6.1]{AJL}, the weighted Mabuchi functional $\mab_{v,0}:=\ent_v+\enR_v$ is convex, where $\enR_v:=\en_v^{\p_{K_X}}$ is the weighted Ricci energy. By Lemma~\ref{lem:enconc}~(iii), it follows that $\ent_v+C\en_v^{\phi_\refe}$ is convex as well for any $C>0$ such that $\p_{K_X}+C\phi_\refe$ is psh. 
\end{proof} 

\begin{proof}[Proof of Proposition~\ref{prop:hent}] By continuity of $\MA_v\colon\cET\to\cM^1$, $\ent_v\colon\cET\to [0,+\infty]$ is lsc. In fact, it follows from~\cite[Theorem~2.17]{BBEGZ} and~\eqref{equ:MAvMA}that $\ent_v$ is strongly lsc, see~\cite[Theorem~6.17]{AJL}. 

Pick $C\gg 1$ such that $F:=\ent_v+C\en_v^{\phi_\refe}$ is convex and lsc. Then $F_\ra$ is lsc (see Proposition~\ref{prop:radial}), and satisfies $F_\ra(\phi_\infty)\le\liminf_{t\to\infty} t^{-1}  F(\phi_t)$ for any almost geodesic ray (see Lemma~\ref{lem:slopemono}). Since $\ent_{v,\ra}=F_\ra-C\left(\en_v^{\phi_\refe}\right)_\ra$ where $\left(\en_v^{\phi_\refe}\right)_\ra$ is {\sgh} and satisfies $\left(\en_v^{\phi_\refe}\right)_\ra(\phi_\infty)=\lim_{t\to\infty} t^{-1}\en_v^{\phi_\refe}(\phi_t)$ (see Proposition~\ref{prop:raden}), this proves the first assertion as well as (i) and (ii). 
 
In the unweighted case, (iii) is a reformulation of~\cite[Theorem~1.2]{LiGeod}, and this implies the general case since $\MA_v\approx\MA$ (see~\eqref{equ:MAvMA}). For the sake of completeness, let us nevertheless briefly repeat Chi Li's argument, which directly adapts to the weighted case. Given a geodesic ray $\{\phi_t\}$ such that $\ent_{v,\ra}(\phi_\infty)<\infty$, we need to show that the associated \na potential $\phi_\NA\in\cE^{1,T}_{\NA}\subset\cE^{1,T}_\ra$ coincides with $\phi_\infty$. Denote by $\{\tphi_t\}$ the unique geodesic ray in $\cET$ emanating from $\phi_0$ and directed by $\phi_\NA$. Then $\tphi_t\ge\phi_t$, and the $S^1$-invariant psh metrics $\Phi$, $\widetilde{\Phi}$ on $(X,L)\times\DD^\times$ induced by $\{\phi_t\}$, $\{\tphi_t\}$ have the same Lelong numbers on all test configurations (see~\cite[Definition~6.5]{YTD}). By~\cite{hiro} it follows that $e^{\widetilde{\Phi}-\Phi}\in L^p$ for all $p\in [1,\infty)$. After passing to logarithmic polar coordinates, this amounts to $\int_0^\infty e^{-2t}\left(\int_X e^{p(\tphi_t-\phi_t)}\,\mu_\refe\right)dt<\infty$, which implies
$$
\liminf_{t\to\infty} t^{-1}\log\int_X e^{p(\tphi_t-\phi_t)}\,\mu_\refe\le 2. 
$$
On the other hand, \eqref{equ:entleg} yields
$$
\ent_v(\phi_t)=\Ent(\MA_v(\phi_t)|\mu_\refe)\ge p\int_X(\tphi_t-\phi_t)\MA_v(\phi_t)-\log\int_X e^{p(\tphi_t-\phi_t)}\,\mu_\refe
$$
Now $\tphi_t\ge\phi_t$ implies 
$$
\int_X(\tphi_t-\phi_t)\MA_v(\phi_t)\ge c\int_X(\tphi_t-\phi_t)\MA(\phi_t)\ge c\left(\en(\tphi_t)-\en(\phi_t)\right)=c\,\dd_1(\tphi_t,\phi_t). 
$$
with $c:=\inf_P v>0$, using~\eqref{equ:Econc}. Altogether this yields
$$
\ent_{v,\ra}(\phi_\infty)\ge p\,c\,\dd_{1,\ra}(\phi_\NA,\phi_\infty)-2.
$$
Since this holds for all $p<\infty$ and $\ent_{v,\ra}(\phi_\infty)$ is assumed to be finite, we conclude $\dd_{1,\ra}(\phi_\NA,\phi_\infty)=0$, \ie $\phi_\NA=\phi_\infty$, and the result follows.
\end{proof}

\begin{defi}\label{defi:wNAent} The \emph{$v$-weighted \na entropy functional} $\ent_{v,\NA}\colon\cE^{1,T}_{\NA}\to[0,+\infty]$ is defined by setting, for $\f\in\cE^{1,T}_{\NA}$, 
$$
\ent_{v,\NA}(\f):=\ent_\NA(\MA_{v,\NA}(\f))=\int_{X_\NA} A_X\MA_{v,\NA}(\f). 
$$
\end{defi}

We are now in a position to state the main technical result of this article. 

\begin{thm}\label{thm:main} The radial and \na entropy functionals 
$$
\ent_{v,\ra}\colon\cE^{1,T}_\ra\to[0,+\infty],\quad\ent_{v,\NA}\colon\cE^{1,T}_{\NA}\to[0,+\infty]
$$
coincide on $\cE^{1,T}_{\NA}\subset\cE^{1,T}_\ra$. 
\end{thm}
Explicitly, for any geodesic ray $\{\phi_t\}$ in $\cET$ such that $\phi_\infty\in\cE^{1,T}_{\NA}\subset\cE^{1,T}_\ra$, 
$$
t^{-1}\ent_v(\phi_t)=t^{-1} \int_X\log\left(\frac{\MA_v(\phi_t)}{\mu_\refe}\right)\MA_v(\phi_t)
$$
converges to
$$
\ent_{v,\NA}(\phi_\infty)=\int_{X_\NA} A_X\MA_{v,\NA}(\phi_\infty).
$$
The proof of Theorem~\ref{thm:main} will be completed in~\S\ref{sec:thmmain} below. As a first step, extending~\cite[Theorem~5.3]{LiGeod} to the weighted setting we prove: 

\begin{lem}\label{lem:entcomp} On $\cE^{1,T}_{\NA}\subset\cE^{1,T}_\ra$ we have $\ent_{v,\NA}\le\ent_{v,\ra}$. 
\end{lem}
\begin{proof} Again, using~\eqref{equ:entscale} we assume for simplicity $\deg_v(L)=1$. Pick a geodesic ray $\{\phi_t\}$ such that $\phi_\infty\in\cE^{1,T}_{\NA}\subset\cE^{1,T}_\ra$ and set $\mu_t:=\MA_v(\phi_t)$. Pick also an snc test configuration $\cX\to\P^1$ for $X$,  a ray $\{\p_t\}_{t\ge0}$ in $C^\infty(K_X)^T$ corresponding to an $(S^1\times T)$-invariant smooth metric on $K_{\cX/\P^1}^{\log}$, and denote by $\nu_t=e^{2\p_t}$ the associated ray of smooth volume forms on $X$. Recall that $\mu_\refe=e^{2\p_{K_X}}$. Then
$$
\ent_v(\phi_t)=\frac12\int_X\log\left(\frac{\mu_t}{\mu_\refe}\right)\mu_t=\Ent(\mu_t|\nu_t)+\int_X\left(\p_t-\p_{K_X}\right)\mu_t.
$$
Since $\nu_t(X)\sim t^d$ for some $d\in\{0,\dots,n\}$ (see Example~\ref{exam:adapted}) and $\mu_t$ is a probability measure, \eqref{equ:entlow} yields 
$$
\Ent(\mu_t|\nu_t)\ge -\log \nu_t(X)\ge-C\log t
$$
for some uniform constant $C>0$. On the other hand, setting $f_t:=\p_t-\p_{K_X}$ defines a ray with analytic singularities in $C^\infty(X)^T$ such that $f_\NA=\f_{K_{\cX/\P^1}^{\log}}=A_\cX$. By Theorem~\ref{thm:NAMAext}~(iii), we infer
$$
t^{-1} \int(\p_t-\p_{K_X})\mu_t\to\int_{X_\NA}A_\cX\,\MA_{v,\NA}(\phi_\infty), 
$$
and hence 
$$
\ent_{v,\ra}(\phi_\infty)\ge\limsup_{t\to\infty} t^{-1}\left(\int_X(\p_t-\p_{K_X})\mu_t-C\log t\right)=\int_{X_\NA}A_\cX\MA_{v,\NA}(\phi_\infty). 
$$
Thus
\begin{align*}
\ent_{v,\ra}(\phi_\infty) & \ge\sup_\cX\int_{X_\NA}A_\cX\MA_{v,\NA}(\phi_\infty)=\ent_{v,\NA}(\phi_\infty),
\end{align*}
see~\eqref{equ:NAentsup}. This completes the proof.
\end{proof}

To avoid repetition it will be convenient to introduce the following local terminology. 

\begin{defi}\label{defi:good} We shall say that $\f\in\cE^{1,T}_{\NA}\subset\cE^{1,T}_\ra$ is \emph{good} if $\ent_{v,\ra}(\f)\le\ent_{v,\NA}(\f)$.
\end{defi}
In view of Lemma~\ref{lem:entcomp}, Theorem~\ref{thm:main} says that all $\f\in\cE^{1,T}_{\NA}$ are good.  

%
\subsection{Envelopes of PL functions are good}

The following result is the main step towards the proof of Theorem~\ref{thm:main}. 

\begin{thm}\label{thm:adapted} For any $f\in\PL(X_\NA)^T$, the psh envelope $\f:=\env(f)\in\cE^{1,T}_{\NA}$ is good. 
\end{thm}

Our proof relies on the following observation: 

\begin{lem}\label{lem:crit} Pick $\f\in\cE^{1,T}_{\NA}$, and assume given an almost geodesic ray $\{\phi_t\}$ in $\cET$ directed by $\f$ such that 
\begin{equation}\label{equ:ugood}
\ent_{v,\NA}(\f)\ge\liminf_{t\to\infty}t^{-1} \ent_v(\phi_t). 
\end{equation}
Then $\f$ is good. 
\end{lem}
\begin{proof} This is a direct consequence of Proposition~\ref{prop:hent}~(ii), 
which yields $\liminf_{t\to\infty} t^{-1} \ent_v(\phi_t)\ge\ent_{v,\ra}(\f)$. 
\end{proof}

It will thus suffice to construct an almost geodesic ray $\{\phi_t\}$ in $\cET$ directed by $\f$ that satisfies~\eqref{equ:ugood}. Pick an equivariant snc test configuration $(\cX,\cL)$ such that $f=\f_\cL$. By translation invariance of $\ent_v$ and $\ent_{v,\NA}$, we may assume, after adding to $\cL$ a large enough multiple of $\cX_0$, that $f=\f_\cX>0$ on $X_\NA$, and hence 
$$
\B_+(\cL/\P^1)=\B_+(\cL)\subset\cX_0,
$$
see Lemma~\ref{lem:B+tc}. In particular, $\cL$ is big (globally on $\cX$). Fix a smooth $(S^1\times T)$-invariant metric $\Psi_\cL$ on  $\cL$, and denote by $\Phi:=\env(\Psi_\cL)$ its (global) psh envelope. Since $\B_+(\cL)\subset\cX_0$, $\Phi$ is locally bounded away from $\cX_0$, and its restriction to $\cL|_{\DD^\times}$ thus determines a psh ray $\{\phi_t\}$ in $\cET$. We claim that $\{\phi_t\}$ does the job. 

\begin{lem}\label{lem:tworays} The ray $\{\phi_t\}$ is almost geodesic and directed by $\f\in\cE^{1,T}_{\NA}\subset\cE^{1,T}_\ra$.
\end{lem}
\begin{proof} Denote by $\{\tphi_t\}$ the geodesic ray in $\cET$ emanating from $\phi_\refe$ and directed by $\f$. We need to show $t^{-1}\dd_1(\phi_t,\tphi_t)\to 0$ (see Definition~\ref{defi:almost}). We are going to show $\sup_X|\phi_t-\tphi_t|=O(1)$, which will yield, as desired, 
$$
t^{-1}\dd_1(\phi_t,\tphi_t)\le t^{-1}\sup_X|\phi_t-\tphi_t|\to 0.
$$
The $(S^1\times T)$-invariant psh metric $\tPhi$ on $\cL|_{\overline{\DD}}$ corresponding to $\{\tphi_t\}$ is characterized as the largest such metric satisfying $\tPhi|_{\cL_1=L}\le\phi_\refe$ (see Proposition~\ref{prop:tcgray}). For $C\gg1$, $\Phi|_{\overline{\DD}}-C$ has these properties, so $\Phi|_{\DD}\le\tPhi+O(1)$. 

Conversely, since $\B_+(\cL)\subset\cX_0$, the psh metric $\Phi$ on $\cL$ is locally bounded away from $\cX_0$, and hence there exists $A\gg1$ such that 
    \[
    \tPhi\le \Phi+A
    \quad \text{on}\quad 
    \cL|_{\DD\setminus\overline{\DD}_{1/2}}.
    \]
    We can therefore define a psh metric $\Psi$ globally on $\cL$ by 
    \[
    \Psi
    :=\begin{cases}
        \Phi+A
        &\text{on}\quad\cL|_{\P^1\setminus\overline{\DD}_{1/2}}\\
        \max\{\Phi+A,\tPhi\}
        &\text{on}\quad\cL|_\DD.
    \end{cases}
    \]
    On the one hand, $\tPhi\le\Psi$ on $\cL|_\DD$. On the other hand, $\Psi$ is psh and $\Psi_\cL$ is smooth, so $\Psi\le\Psi_\cL+O(1)$, and hence $\Psi\le\Phi+O(1)$ by the definition of $\Phi=\env(\Psi_\cL)$. Thus $\tPhi\le\Phi+O(1)$ on $\cL|_\DD$. We have thus proved $\tPhi=\Phi|_\DD+O(1)$, \ie $\sup_X|\phi_t-\tphi_t|=O(1)$. 
\end{proof}

In what follows we denote by 
$$
\iota_t\colon X=\cX_1\simto\cX_{e^{-t}}\hto\cX
$$
the canonical embedding induced by the $\C^\times$-action on $\cX$. 

\begin{lem}\label{lem:adapted} Let $\{\p_t\}_{t\ge0}$ be the ray with analytic singularities in $C^\infty(K_X)^T$ corresponding to an $S^1$-invariant smooth metric on $K_{\cX/\P^1}^{\log}$. Denote by $\nu_t:=e^{2\p_t}$ the associated ray of smooth volume forms on $X$, pick a K\"ahler form $\om_\cX$ on $\cX$, and set $\om_t:=\iota_t^\star\om_\cX$. Then $\nu_t\ge\e\om_t^n$ for some uniform constant $\e>0$. 
\end{lem}
\begin{proof} This follows from a simple computation in local  coordinates, see the final part of the proof of~\cite[Lemma~3.10]{BHJ2}.
\end{proof}

\begin{proof}[Proof of Theorem~\ref{thm:adapted}] As explained above, it suffices to check~\eqref{equ:ugood}. As in Lemma~\ref{lem:adapted}, set $\mu_t:=\MA_v(\phi_t)$, pick a ray of smooth metrics $\p_t$ on $K_X$ corresponding to an $(S^1\times T)$-invariant smooth metric on $K_{\cX/\P^1}^{\log}$, denote by $\nu_t=e^{2\p_t}$ the associated ray of smooth volume forms on $X$, and write
$$
\ent_v(\phi_t)=\frac12\int_X\log\left(\frac{\mu_t}{\mu_\refe}\right)\mu_t=\Ent(\mu_t|\nu_t)+\int_X\left(\p_t-\p_{K_X}\right)\mu_t. 
$$
Theorem~\ref{thm:NAMAext}~(iii) yields here again
$$
\lim_{t\to\infty}t^{-1}\int_X\left(\p_t-\p_{K_X}\right)\mu_t=\int_{X_\NA} A_\cX\MA_{v,\NA}(\f). 
$$
Since $\f=\env(\f_\cL)$, Proposition~\ref{prop:orthoPL}~(i) further implies that $\MA_{v,\NA}(\f)$ is supported in the dual complex $\D_\cX\subset X_\NA$. Since $A_\cX=A_X$ on $\D_\cX$, we conclude
\begin{equation}\label{equ:slopeMAg}
\lim_{t\to\infty}t^{-1}\int_X\left(\p_t-\p_{K_X}\right)\MA_v(\phi_t)=\int_{X_\NA} A_X\MA_{v,\NA}(\f)=\ent_{v,\NA}(\f).
\end{equation}
Next we claim that 
\begin{equation}\label{equ:entsup}
\limsup_{t\to\infty}t^{-1}\Ent(\mu_t|\nu_t)\le 0.
\end{equation}
Combined with~\eqref{equ:slopeMAg} this will conclude the proof of~\eqref{equ:ugood}, and hence of Theorem~\ref{thm:adapted}. Our proof of~\eqref{equ:entsup} relies on the regularity of psh envelopes in a big $(1,1)$-class~\cite{BerDem} (where a gap in some step of the proof was later fixed in~\cite{DNT}). 

Pick an $(S^1\times T)$-invariant, strictly psh metric $\Psi_+$ on $\cL$ with analytic singularities along $\B_+(\cL)$, and smooth on its complement. Consider the corresponding quasi-psh function $F:=\Psi_+-\Psi_\cL$ on $\cX$, and denote by $\{f_t\}$ the associated ray with analytic singularities in $C^\infty(X)^T$, \ie $f_t=\iota_t^\star F$. 

Choose also an $(S^1\times T)$-invariant K\"ahler form $\om_\cX$ on $\cX$ and set $\om_t:=\iota_t^\star\om_\cX$. By~\cite[Theorem~1.4]{BerDem}, the psh envelope $\Phi=\env(\Psi_\cL)$ is $C^{1,\bar 1}$ outside $\B_+(\cL)$, and further satisfies the curvature estimate 
$$
0\le\ddc\Phi\le Ae^{-BF}\om_\cX
$$
on $\cX\setminus\B_+(\cL)$ for some constants $A,B>0$. Since 
$$
\mu_t=\MA_v(\phi_t)\le C\MA(\phi_t)=C'\iota_t^\star(\ddc\Phi)^n,
$$
this implies $\mu_t\le A' e^{-nBf_t}\om_t^n$. On the other hand, Lemma~\ref{lem:adapted} yields $\om_t^n\le C\nu_t$, and hence 
$$
\log\left(\frac{\mu_t}{\nu_t}\right)\le A-Bf_t
$$
for some uniform constants $A,B>0$. Since the mass of $\mu_t=\MA_v(\phi_t)$ is independent of $t$, this yields 
\begin{equation}\label{equ:EntMA}
\Ent(\mu_t|\nu_t)=\int_X\log\left(\frac{\mu_t}{\nu_t}\right)\mu_t\le C\left(1+\int_X (-f_t)\MA_v(\phi_t)\right). 
\end{equation}
Now $\Psi_+$ has analytic singularities along $\B_+(\cL)$, hence $\{f_t\}$ is a ray with analytic singularities in $C^\infty(X)^T$ whose \na limit $f_\NA\in\PL(X_\NA)$ satisfies $f_\NA(v_E)=0$ for each component $E$ of $\cX_0$ not contained in $\B_+(\cL)$.
By Theorem~\ref{thm:NAMAext}~(iii), we have 
$$
\lim_{t\to\infty}t^{-1}\int_X f_t\MA_v(\phi_t)=\int_{X_\NA} f_\NA\MA_{v,\NA}(\f), 
$$
which vanishes since $\MA_{v,\NA}(\f)$ is supported in the set of $v_E$'s such that $E$ is not contained in $\B_+(\cL/\P^1)=\B_+(\cL)$ (see Proposition~\ref{prop:orthoPL}~(ii)). As a result, \eqref{equ:EntMA} yields
$$
\limsup_{t\to\infty}t^{-1}\Ent(\mu_t|\nu_t)\le C\lim_{t\to\infty}t^{-1}\int_X (-f_t)\MA_v(\phi_t)=0, 
$$
completing the proof. 
\end{proof}

%
\subsection{Proof of Theorem~\ref{thm:main}}\label{sec:thmmain}
Recall that Theorem~\ref{thm:main} asserts that all $\f\in\cE^{1,T}_{\NA}$ are good (see Definition~\ref{defi:good}). To prove this, it will be convenient to introduce the following terminology.

\begin{defi} We say that a sequence (or net) $(\f_j)$ in $\cE^{1,T}_{\NA}$ \emph{converges in entropy} to $\f\in\cE^{1,T}_{\NA}$ if $\f_j\to\f$ in $\cE^{1,T}_{\NA}$ and $\ent_{v,\NA}(\f_j)\to\ent_{v,\NA}(\f)$.
\end{defi}
We then have: 

\begin{lem}\label{lem:good} The set of good functions $\f\in\cE^{1,T}_{\NA}$ is closed with respect to convergence in entropy.
\end{lem}

\begin{proof} Assume $\f_j\to\f$ in entropy and $\f_j$ is good for all $j$. Since $\ent_{v,\ra}$ is lsc on $\cE^{1,T}_\ra\supset\cE^{1,T}_{\NA}$ (see Proposition~\ref{prop:hent}), we have 
$$
\ent_{v,\ra}(\f)\le\liminf_j\ent_{v,\ra}(\f_j)\le\liminf_j\ent_{v,\NA}(\f_j)=\ent_\NA(\f),
$$
and the result follows. 
\end{proof}
In view of Theorem~\ref{thm:adapted}, Theorem~\ref{thm:main} will thus be a consequence of the following result. 

\begin{prop}\label{prop:entenv} The set of $\f\in\cE^{1,T}_{\NA}$ of the form $\f=\env(f)$ with $f\in\PL(X_\NA)^T$ is dense in $\cE^{1,T}_{\NA}$ with respect to convergence in entropy. 
\end{prop}

\begin{lem}\label{lem:dualapp} The set of $\f\in\cE^{1,T}_\NA$ such that $\MA_{v,\NA}(\f)$ is supported in some dual complex is dense in $\cE^{1,T}_\NA$ with respect to convergence in entropy. 
\end{lem}
\begin{proof} Set $\mu:=\MA_v(\f)$. For each  equivariant snc test configuration $\cX$ with retraction map $p_\cX\colon X_\NA\to\D_\cX$, set $\mu_\cX:=(p_\cX)_\star\mu$. Then $(\mu_\cX)_\cX$ forms a net in $\cM^{1,T}$ that converges strongly to $\mu$, and satisfies $\Ent_\NA(\mu_\cX)\to\Ent_\NA(\mu)$ (see Lemma~\ref{lem:cvent}). By Theorem~\ref{thm:NAwMA} we can thus find $\f_\cX\in\cE^{1,T}_{\NA}$ such that $\mu_\cX=\MA_v(\f_\cX)$ and $\f_\cX\to\f$ in $\cE^{1,T}_{\NA}$, and we have in fact $\f_\cX\to\f$ in entropy since $\ent_{v,\NA}(\f_\cX)=\Ent_\NA(\mu_\cX)\to\Ent_\NA(\mu)=\ent_{v,\NA}(\f)$. 
\end{proof}

\begin{lem}\label{lem:entsupp} Consider a convergent sequence $\f_j\to\f$ in $\cE^{1,T}_{\NA}$, and assume given an snc test configuration $\cX$ such that $\f_j=\env(\f_j\circ p_\cX)$ for all $j$. Then $\f_j\to\f$ in entropy.
\end{lem}
\begin{proof} By Proposition~\ref{prop:orthoPL}~(i), all measures $\MA_v(\f_j)$ are supported in $\D_\cX$. Since $\f_j\to\f$ in $\cE^{1,T}_{\NA}$ we have $\MA_v(\f_j)\to\MA_v(\f)$ weakly in $\Cz(X)^\vee$. Thus $\MA_v(\f)$ is supported in $\D_\cX$ as well, and hence 
$$
\ent_v(\f_j)=\int_{\D_\cX} A_X\,\MA_v(\f_j)\to\int_{\D_\cX} A_X\,\MA_v(\f)=\ent_v(\f), 
$$
by continuity of $A_X|_{\D_\cX}$. 
\end{proof}

In view of Lemma~\ref{lem:dualapp}, Proposition~\ref{prop:entenv} follows from: 
\begin{lem}\label{lem:PLapp} Pick $\f\in\cE^{1,T}_{\NA}$ such that $\MA_v(\f)$ is supported in a dual complex. Then there exists a sequence $(f_j)$ in $\PL(X_\NA)^T$ such that $\env(f_j)\to\f$ in entropy. 
\end{lem}
\begin{proof} By Proposition~\ref{prop:orthoPL}~(i), we have $\f=\env(\f\circ p_\cX)$. 
Pick a sequence of $\Q$-PL functions $f_j$ on $\D_\cX$ such that $f_j\to\f|_{\D_\cX}$, and still denote by $f_j=f_j\circ p_\cX\in\PL(X_\NA)^T$ their extensions to $X_\NA$. Then $\env(f_j)\to\env(\f\circ p_\cX)=\f$ uniformly on $X_\NA$, and hence $\env(f_j)\to\f$ in entropy, by Lemma~\ref{lem:entsupp}. 
\end{proof}

%

\section{Stability notions and proof of the main results}\label{sec:wstab}
We are finally ready to provide precise definitions of the stability notions that appear in Theorem~B, and to prove the latter. Along the way, we prove Theorem~C. 

In what follows we fix a  (not necessarily maximal) compact torus $T\subset\Aut(X,L)$, with moment polytope by $P\subset M_\R(T)$, and also fix weight functions $v,w\in C^\infty(M_\R)$ such that $v>0$ on $P$ (from~\S\ref{sec:wext2} on we also assume $w>0$ on $P$.)

\subsection{Weighted cscK metrics and the weighted Mabuchi functional}\label{sec:wext1}

We define the \emph{$v$-weighted scalar curvature}
$$
S_v(\phi)\in C^\infty(X)^T
$$
of $\phi\in\cH^T$ as that of the equivariant K\"ahler form $\ddcT\phi:=(\ddc\phi,m_\phi)$ in the sense of~\cite{BJT} (whose conventions differ slightly from~\cite{Lah19,AJL}), \ie as the weighted trace 
$$
S_v(\phi):=\tr_{\phi,v}\Ric_v^T(\ddcT\phi)
$$
of the equivariant weighted Ricci curvature 
$$
\Ric_v^T(\ddcT\phi):=-\ddcT\p_v\quad\text{where}\quad\p_v:=\tfrac 12\log\MA_v(\phi)\in C^\infty(K_X)^T. 
$$
Recall that the weighted trace of the equivariant curvature form $\ddcT\tau=(\ddc\tau,m_\tau)$ of any metric $\tau\in C^\infty(F)^T$ on some equivariant $\Q$-line bundle $F$ is given by
$$
\tr_{\phi,v}(\ddcT\tau)=\tr_\phi(\ddc\tau)+\langle (\log v)'(m_\phi),m_\tau\rangle=\frac{\MA_v^\tau(\phi)}{\MA_v(\phi)}. 
$$
Thus $S_v(\phi)\MA_v(\phi)=-\MA_v^{\p_v}(\phi)$, and hence
\begin{equation}\label{equ:scalmass}
\int_X S_v(\phi)\MA_v(\phi)=-\deg_v'(L;K_X)
\end{equation}
for any $\phi\in\cH_T$, see~\eqref{equ:totaltwMA}. The following more explicit expression of the weighted scalar curvature (which is essentially equivalent to~\cite[Lemma~3.24]{BJT}) shows compatibility with Lahdili's original definition in~\cite{Lah19} (up to a factor $v$.)

\begin{lem}\label{lem:Sv} Pick a basis $(\xi_\a)$ of $N_\R$. For any $\phi\in\cH^T$, the weighted scalar curvature $S_v(\phi)$ is then related to the usual scalar curvature $S(\phi)$ by 
$$
S_v(\phi)=S(\phi)-\tfrac 12|(\log v)'(m_\phi)|^2_\phi-\tfrac12\sum_{\a\b}(\log v)''_{\a\b}(m_\phi)\langle \xi_a,\xi_b\rangle_\phi. 
$$
\end{lem}
Here $\langle\xi_\a,\xi_\b\rangle_\phi\in C^\infty(X)$ denotes the pointwise scalar product with respect to the K\"ahler metric $\ddc\phi$ of $\xi_a,\xi_\b\in N_\R$, viewed as vector fields on $X$, and $|(\log v)'(m_\phi)|_\phi$ similarly denotes the pointwise length of $(\log v)'(m_\phi)\colon X\to M_\R^\vee=N_\R$. 

\begin{proof} Set $h:=\log v$ and $\p:=\tfrac12\log\MA(\phi)\in C^\infty(K_X)^T$. Then $\p_v=\p+\tfrac 12 h(m_\phi)$, $S(\phi)=-\tr_\phi\ddcT\p$, and hence 
\begin{multline*}
-S_v(\phi)=\tr_{\phi,v}\left(\ddcT\p_v\right)=\tr_{\phi,v}(\ddcT\p)+\tfrac12\tr_{\phi,v}(\ddcT h(m_\phi))\\
=-S(\phi)+\langle h'(m_\phi),m_\p\rangle+\tfrac12\D_\phi h(m_\phi)+\tfrac12\langle h'(m_\phi),m_{h(m_\phi)}\rangle. 
\end{multline*}
Using  
$$
\D_\phi h(m_\phi)=\sum_\a h'_\a(m_\phi)\D_\phi m_\phi^{\xi_\a}+\sum_{\a\b} h''_{\a\b}(m_\phi)\tr_\phi(dm_\phi^{\xi_\a}\wedge d^c m_\phi^{\xi_\b}), 
$$
$m_\p=-\tfrac 12\D_\phi m_\phi$ (see~\cite[(3.7)]{BJT}), and~\eqref{equ:mom}, we get
\begin{equation}\label{equ:Dg}
\D_\phi h(m_\phi)=-2\langle h'(m_\phi),m_\p\rangle+\sum_{\a\b} h''_{\a\b}(m_\phi)\langle\xi_\a,\xi_\b\rangle_\phi,  
\end{equation}
and hence $S_v(\phi)=S(\phi)-\tfrac12\langle h'(m_\phi),m_{h(m_\phi)}\rangle-\tfrac 12\sum_{\a\b} h''_{\a\b}(m_\phi)\langle\xi_\a,\xi_\b\rangle_\phi$. Finally~\eqref{equ:moment} and~\eqref{equ:mom} yield 
$$
\langle h'(m_\phi),m_{h(m_\phi)}\rangle=-\sum_\a h'_\a(m_\phi) d(h(m_\phi))(J\xi_\a)=-\sum_{\a\b} h'_\a(m_\phi) h'_\b(m_\phi) dm^{\xi_\b}_\phi(J\xi_\a)
$$
$$
=\sum_{\a\b}h'_\a(m_\phi) h'_\b(m_\phi)\langle\xi_\a,\xi_\b\rangle_\phi=|h'(m_\phi)|_\phi^2,
$$
and the result follows. 
\end{proof}
We shall say for convenience that $\phi\in\cH^T$ (instead of $\ddcT\phi$) is a \emph{$(v,w)$-weighted csc metric} if 
$$
S_v(\phi)=w(m_\phi).
$$
To find weighted csc metrics, we use the \emph{weighted Mabuchi functional} 
$$
\mab_{v,w}\colon\cH^T\to\R, 
$$
first introduced by Lahdili~\cite{Lah19,Lah23}. By construction, it is an Euler--Lagrange functional for the measure-valued operator
$$
\phi\mapsto \left(w(m_\phi)-S_v(\phi)\right)\MA_v(\phi), 
$$
and its critical points are thus precisely the $(v,w)$-weighted csc metrics. It is explicitly given by the \emph{weighted Chen--Tian formula}
\begin{equation}\label{equ:wtrelmab}
\mab_{v,w}(\phi)=\ent_v(\phi)+\rr_v(\phi)+\en_{vw}(\phi),
\end{equation}
where 
\begin{itemize}
\item the weighed entropy $\ent_v(\phi)=\Ent(\MA_v(\phi)|\mu_\refe)$ is the relative entropy of $\MA_v(\phi)$ with respect to a $T$-invariant reference volume form $\mu_\refe$; 

\item the \emph{weighted Ricci energy} $\rr_v(\phi)=\en_v^{\p_\refe}(\phi)$ is a twisted weighted energy with $\p_\refe:=\tfrac12\log\mu_\refe\in C^\infty(K_X)^T$.

\end{itemize}

The Chen--Tian formula provides a natural lsc extension 
$$
\mab_{v,w}\colon\cE^{1,T}\to\R\cup\{+\infty\}
$$
of the weighted Mabuchi functional, which is in fact characterized as the maximal lsc extension (see~\cite[Lemma~6.16]{AJL}). 

By~\cite[Theorem~6.1]{AJL}, $\mab_{v,w}$ is geodesically convex on $\cE^{1,T}$. In view of Lemma~\ref{lem:holdpert} it is also strongly lsc. Indeed, the entropy functional $\ent_v$ is strongly lsc by Proposition~\ref{prop:hent}, and the remaining two terms in~\eqref{equ:wtrelmab} are \sgh, see Proposition~\ref{prop:MAext}. 

\medskip

The following variational characterization of weighted csc metrics plays a key role in the proof of Theorem~B.

\begin{lem}\label{lem:mabmin} Assume $v$ is log-concave. Then the weighted csc metrics in $\cH^S$ are precisely the minimizers of $\mab_{v,w}$ over $\cE^{1,S}$. Furthermore, they are unique modulo $\Aut^S(X,L)$. 
\end{lem}
\begin{proof} Adapting~\cite{BerBer}, 
    Lahdili proved that any weighted csc metric $\phi\in\cH^S$ minimizes $\mab_{v,w}$ over $\cH^T$~\cite[Theorem~1]{Lah19}. By approximation (see~\cite[Theorem~6.1]{AJL}), it follows that $\phi$ is also a minimizer over $\cE^{1,T}$, and hence over $\cE^{1,S}$. Conversely, it was proven by Chen--Cheng~\cite{CC2} in the unweighted case that any minimizer lies in $\cH$. In Corollary~\ref{cor:regmin} we adapt their strategy to the weighted case, building upon~\cite{DJL2}, showing that any minimizer over $\cE^{1,S}$ lies in $\cH^S$, and hence is a weighted csc metric. The uniqueness part goes back to Bando and Mabuchi~\cite{BM} in the Fano case. In our case, the proof is a simple adaptation of the argument in~\cite{BerBer,CPZ15,Lah23}. 
\end{proof}

Finally, we establish the following general unipotent invariance property. 

\begin{prop}\label{prop:mabinv} Assume $v$ is log-concave on $P$. Then $\mab_{v,w}\colon\cH^T\to\R$ is automatically invariant under any unipotent subgroup of $\Aut^T(X,L)$. 
\end{prop}
Our proof is inspired by~\cite[\S6]{Mab90}, which contains a similar result for the Monge--Amp\`ere energy $\en$. 

\begin{lem}\label{lem:unipot} Assume $F\colon\cH^T\to\R$ is an Euler--Lagrange functional of a measure-valued operator $\Ga\colon\cH^T\to C^0(X)^\vee$, such that:
\begin{itemize}
    \item[(i)] $\Ga$ is equivariant with respect to $G:=\Aut^S(X,L)$; 
    \item[(ii)] for $m\gg 1$, $M_m:=\sup_{\phi\in\cH_m^T}\int_X|\Ga(\phi)|$ is finite, where $\cH_m^T\subset\cH^T$ denotes the set of Fubini--Study metrics associated to $T$-invariant Hermitian norms on $\Hnot(X,mL)$.
\end{itemize}
Then $F$ is automatically invariant under the action of any unipotent subgroup $U\subset G$. 
\end{lem}

\begin{proof} As $U$ is unipotent, its image in $\GL(\Hnot(X,mL))$ is contained in the standard unipotent group with respect to some basis $(s_j)$. As $U$ commutes with $T$, we may further assume that each $s_j$ is a $T$-eigensection. Indeed, the $T$-weight decomposition of $\Hnot(X,mL)$ is preserved by $U$, which is thus represented by a standard unipotent group in some basis of eigensections. Pick $\xi
\in\Lie U$, and denote by $(\xi_{ij})$ its matrix in the basis $(s_j)$, so that $\xi\cdot s_j=\sum_{i<j} \xi_{ij} s_i$. Consider the ray of Fubini--Study metrics
$$
\phi_t:=\frac 12\log\sum_j|s_j|^2 e^{-2jt},\quad t\ge 0,
$$
which lies in $\cH^T$ since each $s_j$ is a $T$-eigensection. Then 
$\cL_\xi\phi_t=\frac{\sum_j\langle \xi\cdot s_j,s_j\rangle e^{-2jt}}{\sum_j|s_j|^2 e^{-2jt}}$, and Cauchy--Schwarz yields  
$$
|\cL_\xi\phi_t|^2\le\frac{\sum_j|\xi\cdot s_j|^2 e^{-2jt}}{\sum_j |s_j|^2 e^{-2jt}}\le\frac{\sum_j\sum_{i<j}|\xi_{ij}|^2|s_i|^2 e^{-2it}}{\sum_j|s_j|^2 e^{-2jt}}\le C e^{-2t}. 
$$
Using~\eqref{equ:FutLie} and~(ii), the corresponding Futaki invariant (see~\ref{sec:EL}) satisfies  
$$
|\Fut_F(\xi)|=\left|\int(\cL_\xi\phi_t)F'(\phi_t)\right|\le M_m C e^{-t}. 
$$
Letting $t\to+\infty$ it follows that $\Fut_F\equiv 0$ on $\Lie U$, and the result follows. 
\end{proof}

\begin{lem}\label{lem:curvbound} Assume $v$ is log-concave on $P$. Pick $\phi\in\cH^T$, and assume a Ricci upper bound 
$\Ric(\ddc\phi)\le A\ddc\phi$ with $A>0$. Then $\int_X|S_v(\phi)|\MA_v(\phi)\le C$ with $C>0$ only depending on $L$, $v$ and $A$.  
\end{lem}
\begin{proof} The assumption implies $S(\phi)=\tr_\phi\Ric(\ddc\phi)\le An$, which shows that 
$$
\int_X|S(\phi)|\MA(\phi)\le\int_X(2An- S(\phi))\MA(\phi)=2An-\bar S
$$
is bounded in terms of $L$ and $A$. Thus
\begin{equation}\label{equ:intSbd}
\int_X|S(\phi)|\MA_v(\phi)\le V\sup_P|v|\int_X|S(\phi)\MA(\phi)\le C(L,v,A).
\end{equation}
Using the notation of Lemma~\ref{lem:Sv}, set 
$$
f:=\tfrac 12\sum_{\a\b}h''_{\a\b}(m_\phi)\langle\xi_\a,\xi_\b\rangle_\phi\in C^\infty(X)^T,
$$
so that 
$$
S_v(\phi)=S(\phi)-\tfrac 12|h'(m_\phi)|^2_\phi-f.
$$
As $h$ is concave on $P$, we have $f\le 0$, and it will suffice to prove $\int_X f\MA_v(\phi)\ge -C(L,v,A)$. Indeed, using~\eqref{equ:intSbd}, this will imply 
$$
\tfrac 12\int_X|h'(m_\phi)|^2_\phi\MA_v(\phi)=\int_X S(\phi)\MA_v(\phi)-\int_X S_v(\phi)\MA_v(\phi)-\int_X f\MA_v(\phi)\le C(L,v,A), 
$$
since $\int_X S_v(\phi)\MA_v(\phi)$ is independent of $\phi\in\cH^T$ (see~\eqref{equ:scalmass}), and hence 
$$
\int_X|S_v(\phi)|\MA_v(\phi)\le\int_X\left(|S(\phi)|+\tfrac 12\int_X|h'(m_\phi)|^2_\phi+|f|\right)\MA_v(\phi)\le C(L,v,A).
$$
Set $\p:=\tfrac 12\log\MA(\phi)\in C^\infty(K_X)^T$. Then $\ddc\p=-\Ric(\ddc\phi)\ge-A\ddc\phi$. After slightly increasing $A$, $\p+A\phi\in \cH(K_X+AL)^T$ is strictly psh, and its moment map $m_\p+Am_\phi$ thus takes values in the fixed moment polytope $P_{K_X+AL}$. Since $m_\phi$ takes values in $P_L$, it follows that $m_\p$ is bounded in terms of $L$ and $A$ only. 

By~\eqref{equ:Dg} we have 
$$
f=\langle h'(m_\phi),m_\p\rangle+\tfrac12\D_\phi h(m_\phi).
$$
Since $\int_X\D_\phi h(m_\phi)\MA(\phi)=0$, and 
$$
|\langle h'(m_\phi),m_\p\rangle|\le\sup_P\|h'\|\sup_X|m_\p|\le C(L,v,A)
$$ 
we conclude, as desired, $\int_X f\MA_v(\phi)\ge -C(L,v,A)$. 
\end{proof}

\begin{proof}[Proof of Proposition~\ref{prop:mabinv}] By Lemma~\ref{lem:unipot}, it suffices to show that 
$$
\Ga(\phi):=\mab_{v,w}'(\phi)=\left(w(m_\phi)-S_v(\phi)\right)\MA_v(\phi)
$$
satisfies $\sup_{\phi\in\cH_m^T}\int_X|\Ga(\phi)|<\infty$. On the one hand,  
$$
\int_X|w(m_\phi)|\MA_v(\phi)=\int_P|vw|\DH(X,L)
$$
is independent of $\phi$. On the other hand, for each $\phi\in\cH_m^T$ the Ricci curvature of the Fubini--Study metric $\ddc\phi$ is bounded above by $N_m$ (see for instance~\cite[Lemma~6.6]{BHJ2}). Lemma~\ref{lem:curvbound} thus yields a uniform bound for $\int_X|S_v(\phi)|\MA_v(\phi)$, and the result follows. 
\end{proof}

\begin{rmk} The proof of Lemma~\ref{lem:curvbound}, and hence Proposition~\ref{prop:mabinv}, apply just as well if $v$ is instead assumed to be log-convex.
\end{rmk}

\subsection{The \na weighted Mabuchi functional}\label{sec:NAMab}
The non-Archi\-medean cousin of $\mab_{v,w}$ is the \emph{\na weighted Mabuchi functional} 
$$
\mab_{v,w,\NA}\colon\cE^{1,T}_\NA\to\R\cup\{+\infty\},
$$
given by 
\begin{equation}\label{equ:wtNArelmab}
\mab_{v,w,\NA}=\ent_{v,\NA}+\rr_{v,\NA}+\en_{vw,\NA}.
\end{equation}
Here the $v$-weighted entropy $\ent_{v,\NA}$ is the integral of the log discrepancy function against the $v$-weighted non-Archimedean Monge--Amp\`ere measure, see Definition~\ref{defi:wNAent}. The definition of the other two terms is part of Theorem~\ref{thm:NAMAext}; the \na weighted Ricci energy $\rr_{v,\NA}$ equals $\en_{v,\NA}^{K_X}$.

The next result generalizes Theorem~C in the introduction.
\begin{thm}\label{thm:slwmab} The radial weighted Mabuchi functional $\mab_{v,w,\ra}\colon\cE^{1,T}_\ra\to\R\cup\{+\infty\}$ satisfies 
   \begin{equation*}
\mab_{v,w,\ra}(\f)
=\begin{cases}
    \mab_{v,w,\NA}(\f)\ &\text{if }\f\in\cE^{1,T}_\NA\\
    +\infty\ &\text{otherwise}.
\end{cases}
\end{equation*}
\end{thm}
\begin{proof}
    It follows from Theorem~\ref{thm:NAMAext} that $\rr_{v,\ra}=\rr_{v,\NA}$ and $\en_{vw,\ra}=\en_{vw,\NA}$ on $\cE^{1,T}_\NA$. Moreover, Theorem~\ref{thm:main} shows that $\ent_{v,\ra}=\ent_{v,\NA}$ on $\cE^{1,T}_\NA$, whereas $\ent_{v,\ra}=+\infty$ on $\cE^{1,T}_\ra\setminus\cE^{1,T}_\NA$. The result follows.
\end{proof}

As an important technical consequence, we obtain the following invariance result.
 
\begin{cor}\label{cor:mabinv} For any compact subgroup $S\subset\Aut^T(X,L)$ containing $T$, the following conditions are equivalent: 
\begin{itemize}
    \item[(i)] $\mab_{v,w,\NA}\ge 0$ on the set $\cP_\R^S\subset\cE^{1,S}_\NA$ of $S$-equivariant real product test configurations; 
    \item[(ii)] $\mab_{v,w,\NA}\equiv 0$ on $\cP_\R^S$;
    \item[(iii)] $\mab_{v,w}\colon\cE^{1,S}\to\R\cup\{+\infty\}$ is invariant under $\Aut^S(X,L)$.  
\end{itemize}
 
\end{cor}
We emphasize that $\Aut^S(X,L)$ is not \emph{a priori} assumed to be reductive here. 

\begin{proof} Set $G:=\Aut^S(X,L)$. Denote by $U$ its unipotent radical, and pick a Levi subgroup $L\subset G$, \ie a reductive subgroup mapping isomorphically to $G/U$. 

Given any compact torus $T\subset T'\subset G$, recall that $\xi\mapsto\f_\xi$ defines an embedding $N_\R(T')\hto\cE^{1,T}_\NA$. As a consequence of Theorem~\ref{thm:slwmab}, we see that 
$$
\mab_{v,w,\NA}(\f_\xi)=\frac{d}{dt}\bigg|_{t=0}\mab_{v,w}((e^{it\xi})^\star\phi)=\int_X \cL_{J\xi}\phi\left(w(m_\phi)-S_v(\phi)\right)\MA_v(\phi)
$$
for any $\xi\in N_\R(T')$ and $\phi\in\cH^{T'}$. By linearity in $\xi$, this is non-negative on $N_\R(T')$ iff it vanishes, which is also equivalent to $\mab_{v,w}$ being $T'_\C$-invariant (see~\S\ref{sec:EL}). 

Since $\cP_\R^S=\bigcup_{T'} N_\R(T')$, it follows that (i) and (ii) are equivalent, and also equivalent to the invariance of $\mab_{v,w}\colon\cE^{1,S}\to\R\cup\{+\infty\}$ under any algebraic torus of $\Aut^S(X,L)$, which is implied by (iii).

Conversely, assume (ii). The union of all algebraic tori of the reductive group $L$ contains the set of its regular semisimple elements, which is a non-empty Zariski open subset (see~\cite[Theorem~12.3]{Bor}). It is thus dense in $L$ in the analytic topology, which implies that $\mab_{v,w}$ is $L$-invariant. 

By Proposition~\ref{prop:mabinv}, $\mab_{v,w}$ is also automatically $U$-invariant. It is thus invariant under $G=LU$, which proves (ii)$\Rightarrow$(iii). 
\end{proof}

\subsection{Weighted extremal metrics and the relative weighted Mabuchi functional}\label{sec:wext2}
In general, there are obstructions to the existence of a $(v,w)$-weighted csc metric $\phi\in\cH^T$, the most basic one being the mass compatibility condition
$$
-\deg_v'(L;K_X)=\int_X S_v(\phi)\MA_v(\phi)=\int_X w(m_\phi)\MA_v(\phi)=\deg_{vw}(L), 
$$
which is equivalent to $\mab_{v,w}$ being translation invariant. 

\medskip

From now we further assume $w>0$ on $P$. Then one can look instead for a \emph{$(v,w)$-weighted extremal metric} $\phi\in\cH^T$, defined as a $(v,w\ell^\ext)$-weighted csc metric, where $\ell^\ext\in N_\R(T)\oplus\R$ is the \emph{extremal function}, \ie the unique affine linear function on $M_\R(T)$ such that the \emph{relative weighted Mabuchi functional}
$$
\mab^\rel:=\mab_{v,w\ell^\ext}\colon\cE^{1,T}\to\R\cup\{+\infty\} 
$$
is both translation invariant and $T_\C$-invariant. 

The \emph{relative weighted non-Archimedean Mabuchi functional} is naturally defined as
$$
\mab^\rel_\NA:=\mab_{v,w\ell^\ext,\NA}\colon\cE^{1,T}_\NA\to\R\cup\{+\infty\}
$$
Generalizing a classical result of Calabi~\cite{Cal}, results of Lahdili imply:  

\begin{prop}\label{prop:wtred} Pick a compact subgroup $S\subset\Aut^T(X,L)$ containing $T$, and suppose there exists an $S$-invariant $(v,w)$-weighted extremal metric $\phi\in\cH^S$. Then $\Aut^S(X,L)$ is reductive, and the stabilizer of $\phi$ is a maximal compact subgroup. 
\end{prop}
\begin{proof} As a consequence of~\cite[Theorem~B.1]{Lah19}, $H:=\Aut^T(X,L)$ is the complexification of the identity component of the group of holomorphic isometries of $\phi$ that commute with $T$. Now $\Aut^S(X,L)$ coincides with the identity component $H^S$ of the centralizer of the compact subgroup $S\subset H$, and the result thus follows from Lemma~\ref{lem:reduc}. 
\end{proof}

\begin{rmk} It seems to be unknown whether reductivity still holds in the setting of~\S\ref{sec:wext1}, \ie for weighted csc metrics with $w\in C^\infty(P)$ arbitrary. This is actually the only obstruction to stating Theorem~B in this greater generality. 
\end{rmk}

%
\subsection{Stability notions}\label{sec:stabnot}
As above, pick a compact subgroup $S\subset\Aut^T(X,L)$ containing $T$. Recall that $\cP_\R^S=\cP_\R\cap\cE^{1,S}_\NA$ denotes the set of $S$-equivariant real product test configurations. 

We are now ready to introduce the key stability notions.
\begin{defi}\label{defi:stab}
We say that $(X,L)$ is:
\begin{itemize}
  \item[(i)] \emph{$S$-equivariantly relatively weighted $\hK$-semistable} if $\mab^\rel_\NA(\f)\ge0$ for all $\f\in\cE^{1,S}_\NA$; 
  
  \item[(ii)] \emph{$S$-equivariantly relatively weighted $\hK$-polystable} if (i) holds, with equality only if $\f\in\cP_\R^S$;
  \item[(iii)] \emph{uniformly $S$-equivariantly relatively weighted $\hK$-polystable} if $\Aut^S(X,L)$ is reductive, and there exists $\sigma>0$ such that $\mab^\rel_\NA(\f)
  \ge\sigma\dd_{1,\NA}(\f,\cP_\R^S)$ for all $\f\in\cE^{1,S}_\NA$.
  \end{itemize}
\end{defi}
The reductivity assumption in (iii) guarantees that $\cP_\R^S$ is closed (see Corollary~\ref{cor:prodclosed}), so that (iii)$\Rightarrow$(ii)$\Rightarrow$(i).

\begin{rmk} We simplify the (admittedly quite heavy) terminology as follows: 
\begin{itemize}
    \item if $S=T$, we drop `$S$-equivariantly';
    \item if $v=w=1$, we drop `weighted'; 
    \item if $\ell^\ext$ is constant, we drop `relatively'. 
\end{itemize}
Thus the notion of (uniform) $\hK$-polystability in Theorem~A corresponds to $S=T=\{e\}$ and $v=w=1$. 
\end{rmk}

Next we compare our notions from the usual ones in K-stability. We consider the non-relative, unweighted case $T=\{e\}$ and $v=w=1$, so that $\mab^\rel=\mab$ and $\mab^\rel_\NA=\mab_\NA$ are the usual Mabuchi functional. 

By definition, $(X,L)$ is $S$-equivariantly K-semistable if the Donaldson--Futaki invariant $\DF(\cX,\cL)\ge0$ for all $S$-equivariant ample test configurations $(\cX,\cL)$, and K-polystable if, furthermore, equality holds iff $(\cX,\cL)$ is a product test configuration. 

Recall also that $\cP_\Q^S=\cP_\R\cap\cH^S_\NA$ parametrizes the set of $S$-equivariant rational product test configurations $(\cX,\cL)$, \ie such that the base change of $(\cX,\cL)$ by $z\mapsto z^d$ is a (usual) product test configuration for any sufficiently divisible $d>0$. 

\begin{prop}\label{prop:Kpscrit}
    We have that:
    \begin{itemize}
    \item[(i)] $(X,L)$ is $S$-equivariantly K-semistable iff $\mab_\NA(\f)\ge0$ for $\f\in\cH^S_\NA$;
    \item[(ii)]
    $(X,L)$ is $S$-equivariantly K-polystable if (i) holds, with equality only if $\f\in\cP_\Q^S$. 
    \end{itemize}
\end{prop}

\begin{lem}\label{lem:ptccrit}
    Let $(\cX,\cL)$ be a rational product test configuration. Then $(\cX,\cL)$ is a product test configuration iff $\cX_0$ is reduced.
\end{lem}
\begin{proof} Recall that $(\cX,\cL)$ is a product test configuration iff $\cX_0$ is isomorphic to $X$ (and hence is reduced). Assume $\cX_0$ is reduced. Since $\cX_0$ is isomorphic to the central fiber of the base change $(\cX_d,\cL_d)$ of $(\cX,\cL)$ by $z\mapsto z^d$, it follows from~\cite[Corollary~1.16]{trivval} that $(\cX_d,\cL_d)$ is normal. By assumption, the central fiber of $\cX_d$ is isomorphic to $X$ for $d$ sufficiently divisible. It is also isomorphic to $\cX_0$, and the result follows. 
\end{proof}

\begin{proof}[Proof of Proposition~\ref{prop:Kpscrit}]
    The equivalence in~(i) was proved in~\cite[Proposition~8.2]{BHJ1}.   Let us recall the argument. First, $\DF(\cX,\cL)\ge\mab(\f_\cL)$ with equality iff $\cX_0$ is reduced. Thus $\mab\ge0$ on $\cH^S_\NA$ implies $S$-equivariant K-semistability, and the converse follows from the observation that the Mabuchi functional is equivariant under the scaling action of $\Q_{>0}$, and scaling by an integer corresponds to normalized base change $(\cX_d,\cL_d)$ under $z\mapsto z^d$, whose central fiber is reduced when $d$ is sufficiently divisible.    

    We now turn to polystability. The subset of $\cH^S_\NA$ given by product test configurations is exactly $\cP_\Z^S\subset\cP_\Q^S$, and by Lemma~\ref{lem:ptccrit}, these are exactly the elements of $\cP_\Q^S$ whose associated ample test configuration has reduced central fiber.

    If $(X,L)$ is $S$-equivariantly K-polystable, then it is K-semistable, and the same argument as for Corollary~\ref{cor:mabinv} shows that $\mab_\NA\equiv0$ on $\cP_\Q^S$. Now pick $\f\in\cH^S_\NA\setminus\cP_\Q^S$. For any $d\ge1$, we have $d\cdot\f\in\cH^S_\NA\setminus\cP_\Q^S$, and if $d$ is sufficiently divisible, the  ample test configuration $(\cX_d,\cL_d)$ associated to $d\cdot\f\in\cH^S_\NA$ has reduced special fiber. This yields 
    $$
    \mab_\NA(\f)=d^{-1}\mab_\NA(d\cdot\f)=d^{-1}\DH(\cX_d,\cL_d)>0.
    $$
    Conversely, suppose $\mab_\NA\ge0$ on $\cH^S_\NA$ with equality precisely on $\cP_\Q^S$. Then $(X,L)$ is K-semistable. Now pick any non-product  ample test configuration $(\cX,\cL)$, and set $\f=\f_\cL$. If $\f\not\in\cP_\Q$, then $\DH(\cX,\cL)\ge\mab_\NA(f)>0$. If instead $\f\in\cP_\Q\setminus\cP_\Z$, then $\cX_0$ is non-reduced by Lemma~\ref{lem:ptccrit}, so $\DH(\cX,\cL)>\mab_\NA(\f)\ge0$, and we are done.
\end{proof}
As a direct consequence of Proposition~\ref{prop:Kpscrit}, we get: 
\begin{cor}\label{cor:pstabs}
    If $(X,L)$ is \^ K-semi/polystable, then it is K-semi/polystable. 
\end{cor}

\begin{rmk}\label{rmk:uKps}
    In view of Proposition~\ref{prop:Kpscrit} it is natural to define $(X,L)$ to be \emph{uniformly K-polystable} iff $\Aut^0(X,L)$ is reductive and there exists $\sigma>0$ such that $\mab_\NA(\f)\ge\sigma\dd_{1,\NA}(\f,\cP_\R)$ for all $\f\in\cH_\NA$. This is clearly implied by uniform $\hK$-polystability, and it coincides with uniform K-stability as defined in~\cite{Der,BHJ1} when $\Aut^0(X,L)$ is trivial. 
\end{rmk}

\smallskip
Finally we compare our notions to the one introduced by Chi Li in~\cite{LiGeod}. Again we consider the unweighted case. Li fixes a compact subgroup $S\subset\Aut(X,L)$, and lets $T_\C$ be the identity component of the center of the reductive group $S_\C$, which is an algebraic torus. In our notation, $(X,L)$ being \emph{$S_\C$-uniformly K-stable for models} in the sense of~\cite[Definition~2.25]{LiGeod} amounts to the existence of $\sigma>0$ such that 
\begin{equation}\label{equ:LiCrit}
\mab_\NA(\f)\ge\sigma\dd_{1,\NA}(\f,N_\R(T))
\end{equation}
for all $\f\in\cE^{1,S}_\NA$ of the form $\f=\env(\f_\cL)$, where $(\cX,\cL)$ is a \emph{not necessarily ample} $S$-equivariant test configuration (\ie a \emph{model} in the sense of~\cite{LiGeod}) for $(X,L)$. Now, the argument used in the proof of Proposition~\ref{prop:entenv} (as well as in~\cite[Proposition~6.3]{LiGeod}) shows that it is equivalent to demand that~\eqref{equ:LiCrit} hold for all $\f\in\cE^{1,S}_\NA$. The precise relation to our stability notion is as follows: 

\begin{lem}\label{lem:Licomp} Given a compact subgroup $S\subset\Aut(X,L)$, the following are equivalent:
\begin{itemize}
    \item[(i)] $(X,L)$ is $S_\C$-uniformly K-stable for models;
    \item[(ii)] $(X,L)$ is uniformly $S$-equivariantly $\hK$-polystable, and $\Aut^S(X,L)$ further coincides with the identity component $T_\C$ of the center of $S_\C$. 
\end{itemize}
\end{lem}
Thus (i) is in general strictly stronger than the first half of (ii), since $T_\C=\{e\}$ as soon as $S_\C$ is semisimple. 

\begin{proof} Assume (i). By~\cite[Proposition~6.2]{LiGeod} (and its proof), $\mab\colon\cE^{1,S}\to\R\cup\{+\infty\}$ is then coercive modulo $Z_\C$. It is thus coercive modulo $G$, and $G/T_\C$ is further compact, \ie $T_\C$ is a parabolic subgroup (see Lemma~\ref{lem:coersub}). Since $T_\C\subset G$ is a torus, it follows that $G=T_\C$. This shows (i)$\Rightarrow$(ii), the converse being clear since $G=T_\C$ implies that the image of $N_\R(T)\hto\cE^{1,S}_\NA$ concides with $\cP_\R^S$. 
\end{proof}

\subsection{Proof of Theorems~A \& B}
Theorem~A is a special case of Theorem~B. To prove the latter, set $G:=\Aut^S(X,L)$. Since the set $\cP_\R^S\subset\cE^{1,S}_\NA$ of $S$-equivariant real product test configurations corresponds to the set $\cD_G\subset\cE^{1,S}_\ra$ of toric $G$-directions (see Proposition~\ref{prop:qi5}), Theorem~\ref{thm:slwmab} together with Corollary~\ref{cor:mabinv} show that $(X,L)$ is (uniformly) $S$-equivariantly relatively weighted $\hK$-polystable iff the functional $\mab^\rel\colon\cE^{1,S}\to\R\cup\{+\infty\}$ is $G$-invariant and (uniformly) geodesically stable modulo $G$ (see Definition~\ref{defi:geodstab}). By Proposition~\ref{prop:wtred} and Lemma~\ref{lem:mabmin}, the latter functional further satisfies the assumptions of Theorem~\ref{thm:mincrit}, which thus implies Theorem~B.

As a consequence of Theorem~B, we obtain the following generalization of Zhuang's results in the Fano case~\cite{Zhua21}:
\begin{cor}\label{cor:equivstab} The following are equivalent:
\begin{itemize}
\item[(i)] $(X,L)$ is (uniformly) $\hK$-polystable; 
\item[(ii)] $(X,L)$ is (uniformly) $S$-equivariantly $\hK$-polystable for some compact subgroup $S\subset\Aut(X,L)$;
\item[(iii)] $(X,L)$ is (uniformly) $S$-equivariantly $\hK$-polystable for any compact subgroup $S\subset\Aut(X,L)$.
\end{itemize}
\end{cor}
In particular, one can take $S=T$ to be a maximal compact torus in $\Aut^0(X,L)$. As the the latter is reductive when (i)---(iii) hold, we then have $\Aut^T(X,L)=T_\C$. 
\begin{proof} Pick two (possibly trivial) compact subgroups $S,S'\subset\Aut(X,L)$, and assume $(X,L)$ is (uniformly) $S$-equivariantly $\hK$-polystable. By (the equivariant version of) Theorem~A, there exists an $S$-invariant csc metric $\phi\in\cH^S$. The Matsushima--Lichnerowicz theorem now shows that $\Aut^0(X,L)$ is reductive, and that the stabilizer of $\phi$ in $\Aut^0(X,L)$ is a maximal compact subgroup: see Proposition~\ref{prop:wtred}. After acting on $\phi$ by an automorphism, we may thus assume that its stabilizer contains $S'$ (but perhaps not $S$ anymore). The converse direction of Theorem~A then implies that $(X,L)$ is also (uniformly) $S'$-equivariantly $\hK$-polystable. 
\end{proof}

%
\subsection{A valuative criterion for $\hK$-polystability}\label{sec:divstab}

The stability notions defined in~\S\ref{sec:stabnot} can appear rather intractable in practice, as they involve conditions on the $\dd_{1,\NA}$-completion of the space of ample test configurations, \ie the space $\cE^1_\NA$ of finite energy potentials. 

Building on~\cite{nakstab2}, we show here that uniform $\hK$-polystability admits a somewhat more concrete characterization in terms of convex combinations of divisorial valuations, \ie elements of the set $\cM_\div\subset\cM^1_\NA$ of divisorial measures. For simplicity we treat only the unweighted case.

As in~\cite{nakstab2}, we define 
$$
\b\colon\cM^1_\NA\to\R\cup\{+\infty\}
$$
by transporting the translation invariant functional $\mab_\NA\colon\cE^1_\NA\to\R\cup\{+\infty\}$ via the isomorphism $\MA_\NA\colon\cE^1_\NA/\R\simto\cM^1_\NA$. 

The restriction of $\b$ to divisorial measures can be computed as follows. Pick a finite set $\{v_1,\dots,v_r\}\subset X_\div$ of divisorial valuations, written as $v_i=d_i^{-1}\ord_{D_i}$ with $d_i\in\Q_{>0}$ and $D_1,\dots,D_r\subset Y$ prime divisors on a birational model $\rho\colon Y\to X$. 

By Proposition~\ref{prop:endivmeas}, the energy of any divisorial measure $\mu=\sum_{i=1}^r a_i\d_{v_i}$ with support in the $v_i$ can be expressed as the Legendre transform
$$
\jj_{L,\NA}(\mu)=\sup_{t\in\R^r_{\ge 0}}\left(S_L(t)-a\cdot t\right)
$$
of the concave function 
$$
S_L(t):=V_L^{-1}\int_0^\infty\vol\left(\rho^\star L-\sum_i (\la-t_i)_+ d_iD_i\right)d\la, 
$$
while~\cite[\S 4.2]{nakstab2} yields
$$
\b(\mu)=\sum_i a_i A_X(v_i)+\frac{d}{dt}\bigg|_{s=0}\jj_{L+s K_X,\NA}(\mu). 
$$
The $\b$-invariant of a divisorial measure can thus in principle be computed from the volume function on divisor classes, by integration and differentiation.

\medskip

Next consider a compact torus $T\subset\Aut^0(X,L)$, and set $N_\R=N_\R(T)$. There is a natural \emph{twisting action} $(\xi,v)\mapsto\xi\star v$ of $N_\R$ on $X_\NA^T$, such that $N_\Q$ preserves the subspace $X_\div^T$ of $T$-invariant divisorial valuations (see for instance~\cite[\S 2.4]{LiKE}). 

This induces a twisting action $(\xi,\mu)\mapsto\xi\star\mu$ of $N_\Q$ on the space $\cM_\div^T$ of $T$-invariant divisorial measures, and our valuative criterion can now be stated as follows (compare~\cite[Theorem~1.3]{LiKE}). 

\begin{thm}\label{thm:valcrit} Assume $\Aut^0(X,L)$ is reductive, and pick a maximal compact torus $T\subset\Aut(X,L)$. The following are equivalent: 
\begin{itemize}
    \item[(i)] $(X,L)$ is uniformly $\hK$-polystable; 
    \item[(ii)] there exists $\sigma>0$ such that 
    $$
    \b(\mu)\ge\sigma\inf_{\xi\in N_\Q} \jj_\NA(\xi\star\mu)\text{ for all }\mu\in\cM_\div^T.
    $$
\end{itemize}
\end{thm}

\begin{proof} Twisting the $\C^\times$-action of a $T$-equivariant test configuration by a cocharacter of $T$ induces an action 
$$
N_\Q\times\cH^T_\NA\to\cH^T_\NA\quad (\xi,\f)\mapsto\xi\star\f, 
$$
with respect to which $\MA_\NA\colon\cH^T_\NA\to\cM^T_\div$ is equivariant (see~\cite[Lemma~2.19]{LiKE}). Arguing in terms of spectral measures as in the proof of Proposition~\ref{prop:qi4}, one easily checks that 
$$
\dd_{1,\NA}(\xi\star\f,\xi\star\f')=\dd_{1,\NA}(\f,\f'),\quad\dd_{1,\NA}(\xi\star\f,\xi'\star\f)\le C\|\xi-\xi'\|
$$
for all $\xi,\xi'\in N_\Q$, $\f,\f'\in\cH^T_\NA$ and a uniform constant $C>0$. As a consequence, the $N_\Q$-action on $\cH^T_\NA$ uniquely extends to an isometric action 
$$
N_\R\times\cE^{1,T}_\NA\to\cE^{1,T}_\NA\quad (\xi,\f)\mapsto\xi\star\f. 
$$
Note that $\xi\star 0=\f_\xi$ is the function considered in~\eqref{equ:FSprod}. Similarly, the action of $N_\Q$ on $\cM^T_\div$ extends to a continuous action of $N_\R$ on $\cM^{1,T}_\NA$, with respect to which $\MA_\NA\colon\cE^{1,T}_\NA/\R\simto\cM^{1,T}_\NA$ remains equivariant. 

Now (i) is equivalent to uniform $T$-equivariant $\hK$-stability (see Corollary~\ref{cor:equivstab}), \ie 
$$
\mab_\NA(\f)\ge\sigma\inf_{\xi\in N_\R}\dd_{1,\NA}(\f,\xi\star 0)=\sigma\inf_{\xi\in N_\R}\dd_{1,\NA}(\xi\star\f,0)
$$
for all $\f\in\cE^{1,T}_\NA$ and $\sigma>0$. By translation invariance of $\mab_\NA$, we may further assume that $\f\in\cE^{1,T}_\NA$ is normalized by $\sup\f=0$. Then 
$$
\dd_{1,\NA}(\xi\star\f,0)\approx\jj_\NA(\xi\star\f)=\jj_\NA(\xi\star\mu)
$$ 
with $\mu=\MA_\NA(\f)$, see~\cite[Lemma~5.6]{nakstab1}. Using the above discussion, we conclude that (i) is equivalent to 
$\b(\mu)\ge\sigma\inf_{\xi\in N_\R}\jj_\NA(\xi\star\mu)$ for $\mu\in\cM^{1,T}_\NA$ and $\sigma>0$, which is also equivalent to (ii) since any $\mu\in\cM^{1,T}_\NA$ can be written as a limit of $\mu_i\in\cM^T_\div$ such that $\b(\mu_i)\to\b(\mu)$ (see~\cite[Theorem~3.10]{nakstab2}). 
\end{proof}


%
%
\appendix

%
\section{Reductive groups, symmetric spaces, and Tits buildings}\label{sec:reductive}
In this appendix we collect some results about complex reductive groups and their associated geometric objects: symmetric spaces and conical Tits buildings. All the material here is surely known to experts, but we provide some details for the benefit of the reader.
%
\subsection{The symmetric space of a complex reductive group} 
Recall that a connected linear complex algebraic group $G$ is \emph{reductive} if its unipotent radical is trivial. Equivalently, $G$ is reductive iff it arises as the complexification of a compact (connected) Lie group $G_c$. The latter then sits as a maximal compact subgroup of $G$, and the Lie group exponential induces a polar decomposition 
$$
G_c\times\Lie G_c\simto G=G_ce^{i\Lie G_c}. 
$$
For later use we note:
\begin{lem}\label{lem:reduc} Assume $G$ is a complex reductive group with maximal compact subgroup $G_c\subset G$, and $S\subset G_c$ a compact subgroup. Then the identity component $G^S$ of the centralizer of $S$ is the complexification of the compact group $G_c^S$; in particular, it is reductive. 
\end{lem}
\begin{proof} We are grateful to Michel Brion for the following elementary argument. By uniqueness of the polar decomposition $G=G_c e^{i\Lie G_c}$, we have  
$$
G^{S}=G_c^{S}e^{i(\Lie G_c)^{S}}
$$
where $(\Lie G_c)^{S}$ denotes the fixed point set of the adjoint action of $S$ on $\Lie G_c$, \ie the Lie algebra of $G_c^{S}$. We conclude
$$
G^S=G_c^{S} e^{i\Lie(G_c^{S})}, 
$$
and the result follows. 
\end{proof}

To a complex reductive group $G$ and a maximal compact subgroup $G_c$ is associated the symmetric space 
\[
\Sigma(G):=G/G_c,
\]
with basepoint $[e]\in \Sigma(G)$. Since all maximal compact subgroups of $G$ are conjugate, $\Sigma(G)$ is essentially independent of the choice of $G_c$, which is recovered as the stabilizer in $G$ of the basepoint. The polar decomposition induces a diffeomorphism 
\begin{equation}\label{equ:expred}
\Lie G_c\simto \Sigma(G)\quad \xi\mapsto [e^{i\xi}]. 
\end{equation}
The choice of a $G_c$-invariant scalar product on $\Lie G_c$ further turns $\Sigma(G)$ into a Riemannian symmetric space of non-compact type, and~\eqref{equ:expred} then coincides with the Riemannian exponential map at the basepoint $[e]\in \Sigma(G)$. 

The metric space $\Sigma(G)$ is $\CAT$ and in particular Busemann convex, so we can consider its asymptotic cone $\Sigma_\ra(G)$, consisting of geodesic rays in $\Sigma(G)$ modulo parallelism, see~\S\ref{sec:asym}. The exponential map induces an identification of $\Sigma_\ra(G)$ with the set $\Lie G_c$ endowed with the \emph{Tits metric}
$$
\dT(\xi,\xi')=\lim_{t\to\infty} t^{-1}\dd([e^{it\xi}],[e^{it\xi'}]).
$$
Note that $\dT(\xi,0)=\|\xi\|$ coincides with the length of $\xi$ in the given Euclidean norm $\|\cdot\|$ on $\Lie G_c$. 

\begin{exam}\label{exam:symtorus} Assume $G=T_\C$ is an algebraic torus, with maximal compact torus $T$. For any choice of $T$-invariant scalar product on $\Lie T$ we get canonical isometries
$$
\Sigma(T_\C)\simeq\Lie T\simeq\Sigma_\ra(T_\C). 
$$
\end{exam}
In the general case, by Busemann convexity of $\Sigma(G)$, for all $\xi,\xi'\in\Lie G_c$ we have 
$$
\dd([e^{i\xi}],[e^{i\xi'}])\le\dT(\xi,\xi').
$$
Combined with~\eqref{equ:expred}, this shows that the Tits topology of $\Lie G_c\simeq\Sigma_\ra(G)$ is stronger than its finite dimensional vector space topology. It is in fact strictly stronger in general, as the Tits topology is typically not locally compact (probably as soon as $G$ is not a torus). Indeed, as we will see shortly, $\Sigma_\ra(G)$ can be identified with the conical Tits building $\Sigma_\NA(G)$, and the latter generally fails to be locally compact, see Example~\ref{exam:building}.

%
\subsection{The conical Tits building}\label{sec:Tits}
Now consider the ground field $\C$ equipped with the (non-Archimedean) trivial absolute value. An analogue of the symmetric space $\Sigma(G)$ is then the \emph{conical Tits building}\footnote{The space $\Sigma_\NA(G)$ can be viewed as the cone over the usual (spherical) Tits building of $G$. In French, the conical Tits building is known as an \emph{immeuble vectoriel}, see~\cite{Rou78}.} $\Sigma_\NA(G)$ defined as in~\cite[\S7.4.2]{BruhatTitsI}. 

To describe it, fix a maximal compact torus $T\subset G$, with normalizer $N_G(T)$. The (finite) Weyl group $W=N_G(T_\C)/T_\C$ acts on the character and cocharacter lattices
$$
M_\Z=M_\Z(T),\quad N_\Z=N_\Z(T),
$$
preserving the (finite) set $\Phi\subset M_\Z$ of roots of $(G,T)$, that is, nonzero eigenvalues for the adjoint action of $T$ on $\Lie G$. 

To each root $\g\in\Phi$ is associated a `transvection' subgroup $U_\g\simeq(\C,+)$ of $G$, normalized by $T$ and such that, for any $G$-module $V$ with weight decomposition $V=\bigoplus_{\a\in M_\Z(T)} V_\a$ with respect to $T$, we have 
\begin{equation}\label{equ:unipact}
U_\g\cdot V_\a\subset\bigoplus_{k\ge0} V_{\a+k\g}.
\end{equation}
for any $\a\in M_\Z(T)$. 

Up to conjugation, each parabolic subgroup of $G$ can now be written as 
\begin{equation}\label{equ:Pxi}
P_\xi
:=T_\C\prod_{\g\in\Phi,\,\langle\g,\xi\rangle\ge0}U_\g.
\end{equation}
for some $\xi\in N_\R$. When $\xi\in N_\Z$ is a cocharacter $\C^\times\to G$, $P_\xi$ can also be described as the set of $g\in G$ such that $\xi(\tau)g\xi(\tau)^{-1}$ admits a limit in $G$ as $\tau\to 0$ in $\C^\times$, see~\cite[\S 5.2]{GIT} (and Lemma~\ref{lem:Iwa} below for the general case). 

Following~\cite[\S5.4]{BerkBook}, we may now describe the conical Tits building of $G$ as 
\[
\Sigma_\NA(G):=G\times  N_\R/\!\sim,
\] 
where $(g,\xi)\sim(g',\xi')$ iff there exists $w\in N_G(T)$ such that $w\cdot\xi=\xi'$ and $g^{-1}g'w\in P_\xi$. This is designed so that:
\begin{itemize}
   \item the $G$-action on the first factor of $G\times N_\R$ descends to a $G$-action on $\Sigma_\NA(G)$, whose stabilizers range over all parabolic subgroups of $G$; more precisely, the stabilizer of $[(g,\xi)]$ coincides with $gP_\xi g^{-1}$.  
    \item the map $\xi\mapsto[(e,\xi)]$ defines an embedding $N_\R\hto\Sigma_\NA(G)$, whose images under the $G$-action, called \emph{apartments}, cover $\Sigma_\NA(G)$; 
    \item if a point of $\Sigma_\NA(G)$ lies in two different apartments, then the two corresponding points of $N_\R$ lie in the same $W$-orbit; 
\end{itemize}
For later use we note:
\begin{lem}\label{lem:descend} A map $f\colon N_\R\to\cS$ to a $G$-set $\cS$ (uniquely) extends to a $G$-equivariant map $f\colon\Sigma_\NA(G)\to\cS$ iff $f$ is $W$-equivariant and, for any $\xi\in N_\R$, $f(\xi)$ is fixed by $P_\xi$, \ie by $T_\C$ and by $U_\g$ for each root $\g\in\Phi$ such that $\langle\g,\xi\rangle\ge 0$. 
\end{lem}

The $G$-action on $\Sigma_\NA(G)$ commutes with the scaling action of $\R_{>0}$ on $\Sigma_\NA(G)$, whose restriction to each apartment corresponds to the standard scaling action on the vector space $N_\R$. 
 
Crucially, the building property of $\Sigma_\NA(G)$ guarantees that any two points lie in a common apartment. Any $G$-invariant metric on $\Sigma_\NA(G)$ is thus determined by its restriction to $N_\R\hto\Sigma_\NA(G)$, which is necessarily $W$-invariant. Conversely, any $W$-invariant norm on $N_\R$ gives rise to a (unique) $G$-invariant metric on $\Sigma_\NA(G)$, which is further complete (see for instance~\cite[\S 3]{BE}). 

\medskip

Now fix as above a maximal compact subgroup $G_c\subset G$, a $G_c$-invariant Euclidean norm on $\Lie G_c$, and a maximal compact torus $T\subset G_c$. 

Since $W$ is also equal to $N_{G_c}(T)/T$, the restriction of the Euclidean norm on $\Lie G_c$ to $\Lie T$ is $W$-invariant, and we get a $W$-equivariant isometric embedding
$$
N_\R\simeq\Lie T\simto\Sigma_\ra(T_\C)\hto\Sigma_\ra(G),
$$
which maps $\xi\in\Lie T$ to the direction of the geodesic ray $\{[e^{it\xi}]\}_{t\ge 0}$ of $\Sigma(G)$, see Example~\ref{exam:symtorus}. Since every one-parameter subgroup of $G_c$ lies in a maximal compact torus, which is conjugated in $G_c$ to $T$, we further have $\Lie G_c=\bigcup_{g\in G_c} g\cdot\Lie T$, and hence 
\begin{equation}\label{equ:Sradu}
\Sigma_\ra(G)=\bigcup_{g\in G_c} g\cdot N_\R. 
\end{equation}

\begin{thm}\label{thm:isom} The isometric embedding $N_\R\hto\Sigma_\ra(G)$ uniquely extends to a $G$-equivariant isometric isomorphism $\Sigma_\NA(G)\simto\Sigma_\ra(G)$. 
\end{thm}
The existence of such an isomorphism is well known, see e.g.~\cite{KL97}. The key point is the following result, which can be proved as in~\cite[pp.245--248]{BGS85} using the Iwasawa decomposition of $G$ with respect to any Weyl chamber in $\Lie T\simeq N_\R$ containing $\xi$. 

\begin{lem}\label{lem:Iwa} 
For any $\xi\in\Lie T\simeq N_\R$, the parabolic subgroup $P_\xi\subset G$ coincides with the stabilizer of $\xi\in\Sigma_\ra(G)$  \ie the set of $g\in G$ such that $\{[e^{-it\xi} g e^{it\xi}]\}_{t\ge 0}$ remains bounded in $\Sigma(G)$. 
\end{lem}

\begin{proof}[Proof of Theorem~\ref{thm:isom}] The only thing left to show is that $N_\R\to\Sigma_\ra(G)$ satisfies the two conditions of Lemma~\ref{lem:descend}. We already saw that it is $W$-equivariant, and the second condition follows from Lemma~\ref{lem:Iwa}. 
\end{proof}

\begin{exam}\label{exam:building} Assume $G=\GL(r,\C)$, $G_c=\mathrm{U}(r)$. Then $\Sigma(G)$ and $\Sigma_\NA(G)$ can respectively be identified with the space of Hermitian norms on $\C^r$ and the space of \na norms on $\C^r$ with respect to the trivial absolute value on $\C$. The latter space is not locally compact as soon as $r>1$. 
\end{exam}
%
%
\section{Regularity of minimizers}\label{sec:minreg}
In this section we fix a compact torus $T\subset\Aut(X,L)$, with moment polytope $P\subset M_\R$, and two weights $v,w\in C^\infty(M_\R)$ with $v$ is positive and log-concave on $P$ (and no further assumption on $w$ here). 

We also pick a compact subgroup $S\subset\Aut^T(X,L)$ containing $T$. Using~\cite{DJL2} and following the strategy of~\cite{CC2}, our goal in this appendix is to show:

\begin{thm}\label{thm:regmin} Any minimizer of the weighted Mabuchi functional $\mab_{v,w}\colon\cE^{1,S}\to\R\cup\{+\infty\}$ lies in $\cH^S$, and hence is an $S$-invariant $(v,w)$-weighted csc metric. 
\end{thm}

\begin{cor}\label{cor:regmin} Assume $w>0$ on $P$. Then any minimizer of the relative weighted Mabuchi functional $\mab^\rel\colon\cE^{1,S}\to\R\cup\{+\infty\}$ is an $S$-invariant $(v,w)$-weighted extremal metric. 
\end{cor}
\begin{proof} By definition, $\mab^\rel=\mab_{v,w\ell^\ext}$, and a $(v,w)$-weighted extremal metric is a $(v,w\ell^\ext)$-weighted csc metric. The result is thus a special case of Theorem~\ref{thm:regmin}. 
\end{proof}

For the moment fix $\tau\in\cH^S$, and recall that
$$
\jj_v(\tau,\phi)=\int(\tau-\phi)\MA_v(\phi)-\en_v(\tau)+\en_v(\phi)\ge 0.
$$
for $\phi\in\cE^{1,S}$. 

\begin{lem}\label{lem:JEL} The restriction of $\jj_v(\tau,\cdot)$ to $\cH^S$ is an {\EL} functional of 
$$
\phi\mapsto\MA_v^{\tau}(\phi)-\MA_v^\phi(\phi)=\left(\tr_{\phi,v}(\ddcT\tau)-\tilde v(m_\phi)\right)\MA_v(\phi)
$$
where 
$$
\tr_{\phi,v}(\ddcT\tau):=\tr_{\ddc\phi}(\ddc\tau)+\langle(\log v)'(m_\phi),m_\tau\rangle
$$
is the weighted trace of the equivariant curvature form $\ddcT\tau$, and $\tilde v\in C^\infty(P)$ is defined by 
$$
\tilde v(\a):=n+\langle(\log v)'(\a),\a\rangle. 
$$
Equivalently, we have 
\begin{equation}\label{equ:Jen}
\jj_v(\tau,\phi)=\en_v^\tau(\phi)-\en_v^\tau(\tau)+\en_{\tilde v}(\phi)-\en_{\tilde v}(\tau)
\end{equation}
for all $\phi\in\cH^S$ (and hence also on $\cE^{1,S}$). 
\end{lem} 
\begin{proof} Pick $f\in C^\infty(X)^T$. By~\eqref{equ:twder} and~\eqref{equ:twquad}, we have 
\begin{multline*}
\frac{d}{ds}\bigg|_{s=0}\int(\tau-(\phi+sf))\MA_v(\phi+sf))=-\int f\,\MA_v(\phi)+\int(\tau-\phi)\MA_v^f(\phi) \\ =\int f\left(\MA_v^{\tau-\phi}(\phi)-\MA_v(\phi)\right)=\int f\left(\MA_v^\tau(\phi)-\MA_v^\phi(\phi)-\MA_v(\phi)\right). 
\end{multline*}
Since $\en_v'(\phi)=\MA_v(\phi)$, we infer
$$
\frac{d}{ds}\bigg|_{s=0}\jj_v(\tau,\phi+s f)=\int f\left(\MA_v^{\tau}(\phi)-\MA_v^\phi(\phi)\right),
$$
which also equals 
$$
\int f\left(\tr_{\phi,v}(\ddcT\tau)-\tilde v(m_\phi)\right)\MA_v(\phi)
$$
in view of~\eqref{equ:twMAMA}. The result follows. 
\end{proof}

\begin{lem}\label{lem:Jconv} The functional $\jj_v(\tau,\cdot)\colon\cE^{1,S}\to\R$ is geodesically strictly convex (modulo the translation action of $\R$). 
\end{lem}
\begin{proof} Since $\tau\in\cH^S$, $\en_v^\tau$ is geodesically strictly convex mod $\R$ on $\cE^{1,T}$, see~\cite[Corollary~4]{Lah23}. On the other hand, $\en_{\tv}$ is affine linear on geodesics, and the result thus follows from~\eqref{equ:Jen}. 
\end{proof}

For each $s\in\R_{>0}$ consider the functional $F_s^\tau\colon\cE^{1,S}\to\R\cup\{+\infty\}$ defined by 
$$
F_s^\tau(\phi):=\mab_{v,w}(\phi)+s\jj_v(\tau,\phi). 
$$
By Lemma~\ref{lem:JEL}, $F_s^\tau\colon\cH^S\to\R$ satisfies 
$$
(F_s^\tau)'(\phi)=\left(w(m_\phi)-S_v(\phi)+s(\tr_{\phi,v}(\ddcT\tau)-\tilde v(m_\phi))\right)\MA_v(\phi),
$$
and hence gives rise to the weighted version of Chen's continuity path considered\footnote{Note that the weighted scalar curvature in~\cite{DJL2}, which follows the convention of~\cite{Lah19}, differs from ours by a factor $v$.} in~\cite[\S 3.1]{DJL2}. 
Denote by 
$$
\cC_s:=\{\phi\in\cH^S\mid (F_s^\tau)'(\phi)=0\}
$$
the set of critical points of $F_s^\tau$ on $\cH^S$. It consists of all $\phi\in\cH^S$ such that 
$$
S_v(\phi)=(w-s\tilde v)(m_\phi)+s\tr_{\phi,v}(\ddcT\tau),
$$
\ie $(s\tau)$-twisted, $(v,w-s\tilde v)$-weighted csc metrics.

\begin{lem}\label{lem:Cs} For any $s>0$, $\cC_s$ consists of at most one point (modulo $\R$) which minimizes $F_s^\tau$. 
\end{lem}
\begin{proof} Since $\mab_{v,w}$ is geodesically convex, Lemma~\ref{lem:Jconv} implies that $F_s^\tau$ also is geodesically strictly convex mod $\R$ on $\cE^{1,T}$ for $s>0$, and the result follows. 
\end{proof}

The existence of a minimizer of $\mab_{v,w}$ on $\cE^{1,S}$ implies that it is bounded below; it is thus automatically translation invariant, which amounts to saying that 
\begin{equation}\label{equ:normvvw}
\int \mab_{v,w}'(\phi)=\int_X (w(m_\phi)-S_v(\phi))\MA_v(\phi)=\deg_{vw}(L)+\deg_v'(L;K_X)=0. 
\end{equation}

\begin{lem}\label{lem:regmin} The set $\cS:=\{s\in\R_{>0}\mid\cC_s\ne\emptyset\}$ is non-empty and open. Furthermore, if $s\in \cS$, $\phi_s\in\cC_s$ and $C>0$ satisfy $s\le C$, $\dd_1(\phi_s)\le C$, then $\phi_s$ is $C^\infty$-bounded, \ie $\|\phi_s\|_{C^k}\le A_k(C,\tau)$ for all $k\in\N$. 
\end{lem}
\begin{proof} The first part follows from the weighted version of~\cite{Has19,Zen} established in~\cite[Corollary~3.4, Theorem~3.5]{DJL2}. Strictly speaking, their proof deals with the case $S=T$, but the argument can be restricted without change to $S$-invariant metrics. To see the second part, note that $\jj_v(\tau,\phi_s)$ and the energy part of $\mab_{v,w}(\phi_s)$ are both bounded in terms of $C$. Since $F_s^\tau(\phi_s)\le F_s^\tau(\phi_\refe)=0$, it follows that $\ent_v(\phi_s)$ is bounded in terms of $C$. The desired $C^\infty$-bound now follows from~\cite[Theorem~2.1]{DJL2} combined with standard ellipticity estimates. 
\end{proof}

\begin{proof}[Proof of Theorem~\ref{thm:regmin}] Assume $\phi\in\cE^{1,S}$ minimizes $\mab_{v,w}$. Pick as above $\tau\in\cH^S$, and use the notation of Lemma~\ref{lem:regmin}. For each $s\in\cS$, $\phi$ and $\phi_s$ respectively minimize $\mab_{v,w}$ and $F_s^\tau=\mab_{v,w}+s\jj_v(\tau,\cdot)$. Thus 
$$
\mab_{v,w}(\phi_s)+s\jj_v(\tau,\phi_s)\le\mab_{v,w}(\phi)+s\jj_v(\tau,\phi)\le\mab_{v,w}(\phi_s)+s\jj_v(\tau,\phi), 
$$
and hence $\jj_v(\tau,\phi_s)\le\jj_v(\tau,\phi)$. This yields an a priori bound $\dd_1(\phi_s)\le C$, and hence an \emph{a priori} $C^\infty$-bound on $\phi_s$, by Lemma~\ref{lem:regmin}. Now take a sequence $s_j\in S$ that converges to $s_\infty:=\inf\cS$, and $\phi_j\in\cC_{s_j}$. Then $\phi_{s_j}$ converges smoothly, up to a subsequence, to $\tilde\phi\in\cC_{s_\infty}$ such that $\jj_v(\tau,\tilde\phi)\le\jj_v(\tau,\phi)$. By openness of $\cS$, we must therefore have $s_\infty=0$, \ie $\tilde\phi$ is a weighted csc metric.  

Now write $\phi$ as the $\dd_1$-limit of a sequence $\phi_i\in\cH^S$. For each $i$, the first part of the proof applied to $\tau:=\phi_i$ yields a weighted csc metric $\tilde\phi_i\in\cH^S$ such that $\jj_v(\phi_i,\tilde\phi_i)\le\jj_v(\phi_i,\phi)$. Thus $(\tilde\phi_i)$ is bounded in $\cE^1$. Since $\tilde\phi_i$ is weighted csc, it follows from~\cite[Theorem~2.1]{DJL2}  that is bounded in $C^\infty$, and we may thus assume that $\tilde\phi_i$ converges smoothly to a weighted csc metric $\tilde\phi\in\cH^S$. Now $\jj_v(\phi_i,\tilde\phi_i)\to\jj_v(\phi,\tilde\phi)$, $\jj_v(\phi_i,\phi)\to 0$, thus $\jj_v(\phi,\tilde\phi)=0$, which implies, as desired, that $\phi=\tilde\phi+\mathrm{cst}$ is smooth. 
\end{proof}

%
%
%
%

%
%
%
%
%
%
\end{document}